\documentclass[10pt]{amsart}
\usepackage[dvipsnames]{xcolor}

\usepackage[labelsep=quad,indention=10pt]{subfig}
\usepackage{latexsym, amsmath, mathtools, amssymb,amsthm,amsopn,amsfonts, eucal, dsfont, bbm,float}

\usepackage{soul,caption}

\usepackage[utf8]{inputenc}
\usepackage{version}
\usepackage[colorinlistoftodos]{todonotes}
\usepackage{yfonts}

\makeatletter
 \providecommand\@dotsep{5}
 \def\listtodoname{List of Todos}
 \def\listoftodos{\@starttoc{tdo}\listtodoname}
 \makeatother

\definecolor{skyblue}{rgb}{0.85,0.85,1} 

\ifx\pdfoutput\undefined
  \DeclareGraphicsExtensions{.eps}
\else
  \ifx\pdfoutput\relax
    \DeclareGraphicsExtensions{.eps}
  \else
    \ifnum\pdfoutput>0
      \DeclareGraphicsExtensions{.pdf}
    \else
      \DeclareGraphicsExtensions{.eps}
    \fi
  \fi
\fi

\newtheorem{theorem}{Theorem}[section]

\newtheorem{remark}[theorem]{Remark}

\newtheorem{definition}[theorem]{Definition}

\newtheorem{proposition}[theorem]{Proposition}

\newtheorem{lemma}[theorem]{Lemma}

\newtheorem{corollary}[theorem]{Corollary}

\newcommand{\tr}{\operatorname{tr}}

\newcommand{\dist}{\operatorname{dist}}

\newcommand{\R}{{\mathbb R}}
\newcommand{\C}{{\mathbb C}}

\newcommand{\Hsym}{\mathbb{H}_{\mathrm{sym}}}
\newcommand{\Htrace}{\mathbb{H}_{\mathrm{trace}}}

\newcommand{\e}{\varepsilon}

\theoremstyle{plain}

\theoremstyle{definition}

\numberwithin{equation}{section}

\def\squarebox#1{\hbox to #1{\hfill\vbox to #1{\vfill}}}

\newcommand{\p}{\partial}
\newcommand{\abs}[1]{\left\vert{#1}\right\vert}

\usepackage{amsxtra}

\author{Dmitry Golovaty}
\address{Department of Mathematics, The University of Akron, Akron, OH 44325, USA}
\email{dmitry@uakron.edu}

\author{Matthias Kurzke}
\address{School of Mathematical Sciences, University of Nottingham, University Park, Nottingham, NG7 2RD, UK}
\email{matthias.kurzke@nottingham.ac.uk}

\author{Jose Alberto Montero}
\address{Santiago, Chile}
\email{j.alberto.montero.z@gmail.com}

\author{Daniel Spirn}
\address{School of Mathematics, University of Minnesota, Minneapolis, MN 55455}
\email{spirn@umn.edu}

\title
{
Tetrahedral frame fields via constrained third order symmetric tensors}

\begin{document}    

\begin{abstract}
    Tetrahedral frame fields have applications to certain classes of nematic liquid crystals and frustrated media.
    We consider the problem of constructing a tetrahedral frame field in three dimensional domains in which the boundary normal vector is included in the frame on the boundary. 
    To do this we identify an isomorphism between a given tetrahedral frame and a symmetric, traceless third order tensor under a particular nonlinear constraint.  We then define a Ginzburg-Landau-type functional which penalizes the associated nonlinear constraint.  Using gradient descent, one retrieves a globally defined limiting tensor outside of a singular set. The tetrahedral frame can then be recovered from this tensor by a determinant maximization method, developed in this work. The resulting numerically generated frame fields are smooth outside of one dimensional filaments that join together at triple junctions.  
\end{abstract}

\maketitle

\setcounter{tocdepth}{1} 

\tableofcontents
\section{Introduction}
In this paper we continue with our program that aims to use variational methods for tensor-valued functions in order to describe frame-valued fields in $\R^n$. Here a \emph{frame} $\textfrak{F}$ is a fixed set of $m$ vectors in $\R^n$ that satisfies some symmetry conditions, while a \emph{frame-valued field} $R(x)\textfrak{F}$ assigns a rigid rotation $R(x)\in SO(n)$ of $\textfrak{F}$ to every point $x\in\Omega\subset\R^n$. 

Whenever $-a\in\textfrak{F}$ for all $a\in\textfrak{F},$ the frame $\textfrak{F}$ of $m$ vectors can be identified with a frame composed of $m/2$ lines. Here of particular interest is a set of $n$ orthogonal lines in $\R^n$, known as an $n$-\emph{cross}. An $n$-\emph{cross field} associates an $n$-cross with every point in $\R^n$. In \cite{GMS} we investigated whether it is possible to construct a smooth field of $n$-crosses in $\Omega$, assuming certain behavior of that field on $\partial\Omega$. This problem has received a considerable attention in computer graphics and mesh generation \cite{Vaxman}. 

In two dimensions (or on surfaces in three dimensions)  quad meshes can be obtained by finding proper parametrization based on a $2$-cross field defined over a triangulated surface \cite{Li:2012:AMU:2366145.2366196}.  A similar hexahedral mesh generation approach in three dimensions is typically accomplished by constructing a $3$-cross field on a tetrahedral mesh and then using a parametrization algorithm to produce a hexahedral mesh \cite{05_IMR23_Kowalski,Nieser_2011}. From a mathematical perspective, the first step in this procedure requires a $3$-cross field in $\Omega\subset\R^3$ that is sufficiently smooth and properly fits to $\partial\Omega$, e.g., by requiring that one of the lines of the field is orthogonal to $\partial\Omega$. Generally, an $n$-cross field of this type has singularities on $\partial\Omega$ and/or in $\Omega$ due to topological constraints \cite{GMS}. 

A number of approaches have been proposed to construct a $2$- or $3$-cross fields \cite{BeaufortPoincare,BERNARD2014175,bommes,Huang:2011:BAS:2070781.2024177,05_IMR23_Kowalski,Li:2012:AMU:2366145.2366196,ViertelOsting} but of particular interest to us in \cite{GMS} was a promising direction identified in \cite{BeaufortPoincare,ViertelOsting} for $2$-cross fields where a connection to the harmonic map relaxation, i.e., asymptotic limits in Ginzburg-Landau theory was noticed. While this connection is transparent in two dimensions, the appropriate descriptors in three dimensions, however, were not known until very recently \cite{Chemin2019,palmer2019algebraic}. One of our contributions in \cite{GMS} was to propose a unified tensor-based approach to constructing $n$-cross fields that takes advantage of classical PDE theory. 

Our framework in \cite{GMS} applies in arbitrary dimensions and associates an $n$-cross with a symmetric $4$-tensor that satisfies certain trace conditions and a nonlinear constraint. We relax this constraint by introducing an appropriate penalty term to obtain a global Ginzburg-Landau-type variational problem for relaxed, tensor-valued maps.  The Ginzburg-Landau relaxation embeds the problem into a global steepest descent that allows for a new selection principle for the limiting $n$-cross field that we were able to explore numerically.  
 
{In this paper we adapt this procedure to optimal generation of tetrahedral frame fields in Lipschitz domains.} Here a tetrahedral frame $\textfrak{T}$ is a set of four position vectors of the vertices of a tetrahedron in $\mathbb{R}^3$. The frame $\textfrak{T}$ can be identified with a constrained $3$-tensor and, using the ideas of \cite{GMS}, we develop a scheme for constructing tetrahedral frame fields using energetic relaxation for $3$-tensor-valued functions.  The primary novel feature of our approach is the recovery procedure that allows for extraction of the tetrahedral frame from a constrained 3-tensor (Theorem \ref{t:recovery_3_d}). In fact, our construction gives an \emph{isometric} embedding of tetrahedral frames into the space of constrained 3-tensors (Theorem \ref{thm55}).

When we constrain a frame field to contain the normal on the boundary for topological reasons singularities emerge, as is in the case of 3-cross fields, \cite{GMS}. Our variational relaxation approach allows us to observe formation of singularities numerically, Fig.~\ref{fig:junctions}, and to study their topological properties. Here unexpected features arise as consequence of noncommutativity of the fundamental group of the target manifold.
\begin{figure}[H]
        \centering
           \subfloat[Singular sets connecting in triple junctions for $SO(3)/T$-targets.]{%
              \includegraphics[height=2.5in]{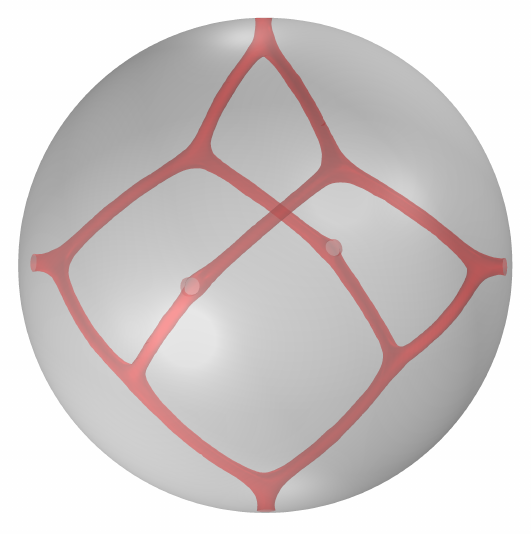}%
              \label{fig:sphere.tetra}%
           }\qquad
           \subfloat[Singular sets connecting in quadruple junctions for $SO(3)/O$-targets.]{%
              \includegraphics[height=2.5in]{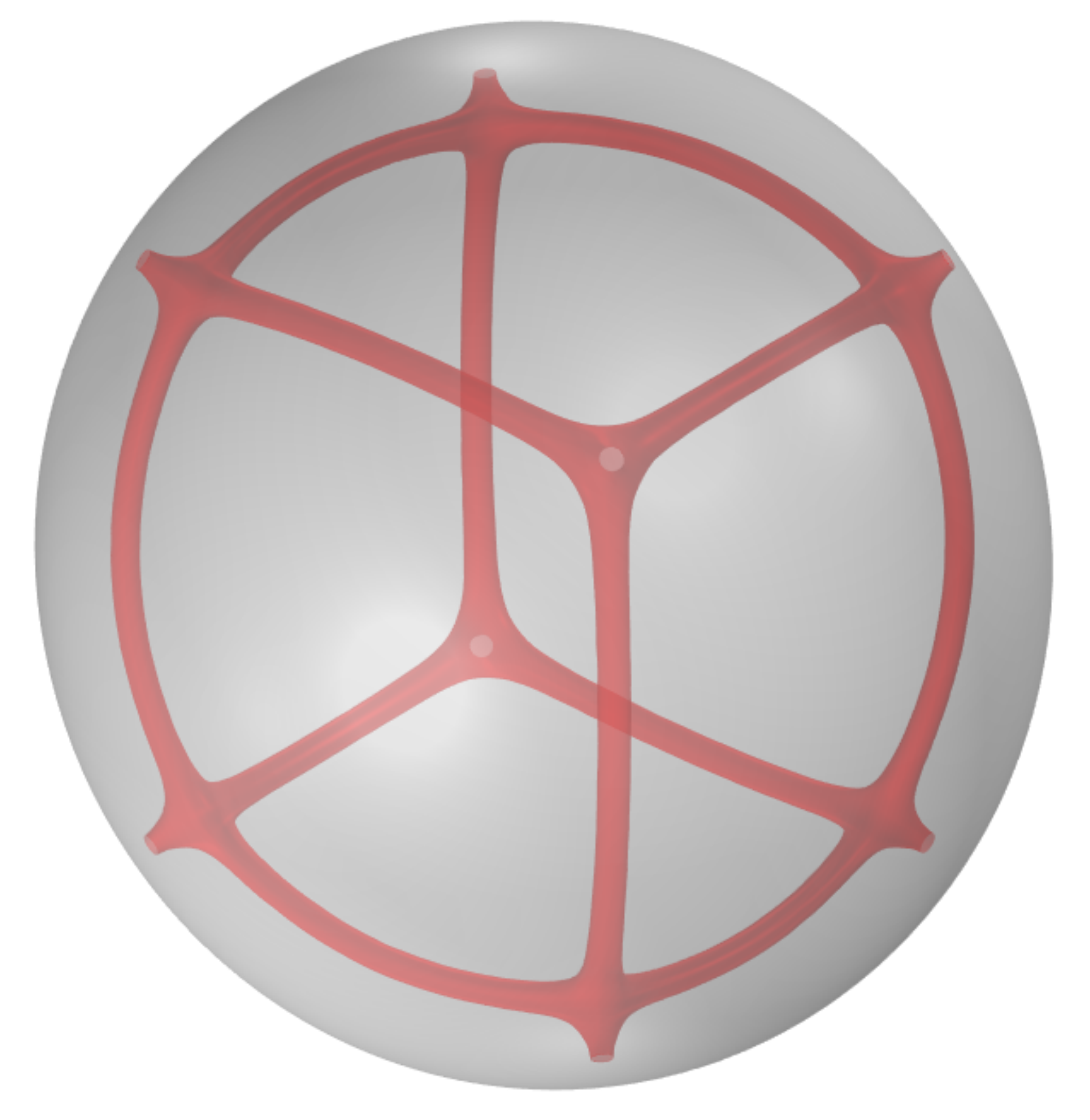}%
              \label{fig:sphere.cross}%
           }
           \caption{Different quotients of $SO(3)$ lead to different types of junctions.  See \cite{GMS} for more details on the $SO(3)/O$ computations and Section \ref{sec:numerics} of this paper for more details on $SO(3)/T$ computations.}
           \label{fig:junctions}
\end{figure}
In particular, unlike the standard Ginzburg-Landau where the number of singularities is dictated by the degree of the boundary data, here we observe local minimizers with different number of vortices for the same boundary conditions. Further, local minimizers in three dimensional domains exhibit one-dimensional singular sets that meet at triple junctions (as opposed to quadruple junctions in \cite{GMS}). These junctions and their structure are another novel feature of our numerical experiments. Investigation of properties of minimizers of our relaxed problem is a fascinating challenge for further analysis. 

Note that tetrahedral frame fields on $\Omega$
can be written as maps  $\Omega\to SO(3)/T=SU(2)/2T$, 
where $T$ is the tetrahedral group and $2T$ the binary tetrahedral group, its pre-image under the covering map $SU(2)\to SO(3)$. It is well known since the work of F. Klein \cite{KleinIkosaeder} that the quotient of $\C^{2}$ by a finite subgroup $G$ of $SU(2)$ can be realized as an algebraic variety in $\C^3$, using the theory of invariant polynomials. Because here we are interested in the quotient $SU(2)/G$ which is a closed subset of $\C^2/G$, the Klein construction would require us to impose additional polynomial constraints.  Rather than use this approach, we instead embed $SU(2)/G$ into an algebraic variety of a Euclidean space of a large enough dimension so that the resulting polynomial equations are of order two, the lowest possible for a non-linear polynomial. As it turns out, this embedding is into a set that can be endowed with a natural matrix structure that leads to nice compactness properties in weak topologies, see \cite{GMS2}. 

Our interest in tetrahedral frame fields represented by third order symmetric traceless tensors is not purely mathematical as they have drawn significant attention from the physicists since the early 80's
\cite{brand2005tetrahedratic,Turnbull1978,LuiPhysX,Fel,Nelson1983,NelsonToner1981,Trebin1984}. As we will discuss in the next section, of particular relevance to this work is modeling of bent-core liquid crystals, where phases with tetrahedral symmetry can arise in certain temperature regimes. 

\subsection{Results and organization of the paper}

{ 
We now describe the organization of this paper. Section~\ref{sec:nlc} provides necessary background and motivation for our work from the modeling of nematic liquid crystals with symmetries  and other physical applications.

In Section \ref{s:nD}, we start with $n+1$ vectors, $\{ \mathbf{u}^j\}$ in $\R^n$ that are equally spaced out on $\mathbb{S}^{n-1}$.  Not only does this include tetrahedral frames, it also includes  the two-dimensional analog of a tetrahedral frame. This analogous frame  consists of the position vectors of vertices of an equilateral triangle with the center of mass at the origin, and we refer to such a three-pronged shape as an \emph{MB frame}. MB frame fields appears naturally when we discuss traces of tetrahedral frame fields on the boundary of a three-dimensional domain. 
Associated with the $n+1$ vectors is a symmetric three tensor, 
\begin{equation} \label{e:introDefQ}
\mathcal{Q}_{ijk} = \sum_{\ell=1}^{n+1} \mathbf{u}^\ell_i \otimes \mathbf{u}^\ell_j \otimes  \mathbf{u}^\ell_k.
\end{equation}
The rest of Section~\ref{s:nD}  identifies invariants of tensors \eqref{e:introDefQ}: not only are the $\mathcal{Q}$'s symmetric and traceless, they also satisfy an SVD-type identity, 
\begin{equation} \label{e:IntroSVD}
\mathcal{Q} \mathcal{Q}^T = \lambda_n^2 I(n) ,
\end{equation}
where $I(n)$ is the $n\times n$ identity matrix and $\lambda^2_n ={(n+1) (n^2 - 1) \over n^2}$ , see Proposition~\ref{p:3tensorlaws}.  Identity \eqref{e:IntroSVD} proves to be a crucial tool, as we will use it to "push" our linear space towards $SO(3)/T$.  Additionally, we show that these $\mathcal{Q}$'s enjoy an eigenvector-eigentensor structure, see  Proposition~\ref{p:etevpairs2}.  A particularly useful consequence of the eigenvector-eigentensor pairing is a mechanism to ensure that the normal vector is contained within and MB or tetrahedral frame field, 
\begin{equation} \label{e:normalvectorIntro}
\mathcal{Q} ( \nu \otimes \nu) = {n^2 - 1\over n^2} \nu.
\end{equation}

Suppose that $\Htrace(n,k)$ is the set of $k$-th order tensors in $\mathbb R^n$ that are symmetric and traceless, i.e.  invariant with respect to permutation of indices and such that a contraction using the last two indices produces a zero $k-2$-order tensor in $\mathbb R^n$. In Section~\ref{sec:rec} we establish two theorems that enable us to recover frames of interest in two and three dimensions. When $n=2,$ we prove that
\begin{equation}  
\label{e:isomercedes0}
    SO(2)/D_3 \equiv  \Htrace(2,3) \cap \left\{ \mathcal{Q} \mathcal{Q}^T =  \lambda_2^2 I(2) \right\},
\end{equation} 
where $D_3$ is the ${2 \pi \over 3}$-rotation group. When $n=3,$ we establish that
\begin{equation}
\label{e:isotetra0}
    SO(3)/T \equiv  \Htrace(3,3) \cap\left\{ \mathcal{Q} \mathcal{Q}^T =  {32 \over 27} I(3)   \right\},
\end{equation} 
where $T$ is the tetrahedral group. For tensors in $\Htrace(3,3) \cap\left\{ \mathcal{Q} \mathcal{Q}^T =  {32 \over 27} I(3)   \right\}$ we also provide an algorithm for computing four vectors of a tetrahedral frame from a given tensor. These theorems are proved in Section~\ref{sec:rec} and Appendix~B. 

In Section~\ref{sec:relax} we apply our results from Section~\ref{sec:rec} to generate a frame field in Lipschitz domains. 
Simple constructions in this section  show that requiring the normal vector on the boundary be included in either an MB or tetrahedral frame  induces nonexistence of a smooth frame field in the interior.
To avoid this, we work in the linear subspace $\Htrace(n,3)$ and push towards the constraint \eqref{e:IntroSVD}. Indeed the relaxation procedure is defined  by generating a sequence, 
\[
\mathcal{Q}_\e = \arg \min_{\mathcal{A} \in H^1_\nu(\Omega; \Htrace(n,3))} \int_\Omega 
{1\over 2} \left| \nabla \mathcal{A} \right|^2 + {1\over \e^2} \left| \mathcal{A} \mathcal{A}^T - {\lambda_n^2 } I(n)\right|^2   dx,
\]
where the space $H^1_\nu(\Omega;\Htrace(n,3))$ is described earlier in the section. This space includes constraints to ensure \eqref{e:normalvectorIntro} holds, or is satisfied in the limit.
In later sections, we examine the limiting $\mathcal{Q}_\e$ via computational experiments and show that  MB and tetrahedral frame fields are smooth outside of co-dimension 2 singular sets. 

In Section~\ref{sec:qua} we develop a connection between tetrahedral frames and quaternions in $S^3/2T$, where $S^3$ are the unit quaternions and $2T$ is a specific finite
subgroup. We show that 
a natural map from quaternions
to symmetric traceless tensors
induces an isometric
embedding of the space of tetrahedra. We also compute the fundamental 
group of the space of tetrahedra. As 
the group is non-abelian, the free 
homotopy classes are characterized
by the conjugacy classes of the 
fundamental group.

In Section~\ref{sec:poi-hop} we provide some global geometric information on tetrahedral frame fields in smooth three dimensional domains.  In particular an adaptation of the classical  Poincare-Hopf theorem to frame fields, see \cite{ray:NSDF:2006}, provides a constraint on the total index of the tangential MB field induced by the requirement that  the normal vector being contained in the tetrahedral frame field on boundary.  If one defines the index of the tangential MB field on the surface as the angular change about a singular point divided by $2\pi$, then this results in a formula, 
\begin{equation} \label{e:indexformulaIntro}
\sum_{x\in A} i(x) = 2-2g,
\end{equation}
where $i(x)$ is the index of the zero, $A$ is the set of singularities  of the tangential MB field on the surface, and $g$ is the genus of the bounding surface, 
see Remark~\ref{rmk:PH}.  Applying \eqref{e:indexformulaIntro} to a ball in 3D, one finds $i(x) = 6$, which  corresponds to the number of boundary point singularities in simulation $(a)$ in Figure~\ref{fig:junctions}.

Section~\ref{sec:numerics} describes typical examples of tetrahedron-valued critical points for both two- and three-dimensional energies obtained numerically via gradient flow. In numerical simulations, the trace of each competitor on the boundary of the domain is assumed to contain the normal to the boundary. Topological obstructions associated with these boundary conditions give rise to formation of both boundary point- and interior line singularities. 

In prior sections we have looked at energy minimizing sequences in which the Dirichlet energy is appended with a particular fourth-order nonlinear potential $\widetilde{W}(\mathcal{Q}) = {1\over \e^2} |\mathcal{Q} \mathcal{Q}^T - \lambda_n^2 I(n)|^2$ that pushes our linear space towards a tetrahedral (or MB) frame field.  
In Appendix~C we explore the connection between our work and a more general fourth order potential for $\Htrace(3,3)$ introduced in \cite{LR2002} to describe bent-core nematic liquid crystals. 
This potential,
\begin{align*}
    \label{eq:radlub}
    W(\mathcal{Q}) = \frac{\abs{\mathcal{Q}}^4}{4} - \frac{\alpha}{2}\abs{\mathcal{Q}}^2 + \frac{\beta}{4}\sum_{i, j = 1}^3 \langle \mathcal{Q}_i, \mathcal{Q}_j\rangle ^2.
\end{align*}
allows for much richer sets of minimizers.  We characterize some features of energy minimizers with this more general potential in terms of the $\alpha$ and $\beta$.  In fact, we find both MB and tetrahedral frames in three dimensional domains, depending on the values of the parameters.
}

\section{Bent-core nematic liquid crystals}
\label{sec:nlc}
A liquid crystal is a state of matter intermediate between a solid and a liquid in that it retains some degree of order characteristic of a solid, yet in can flow like a liquid. For example, a \emph{nematic} liquid crystal---typically composed of molecules that have highly anisotropic shapes---possesses orientational order for a certain range of temperatures or concentrations. Two other common types of liquid crystals include \emph{cholesterics} formed by screw-shaped molecules that exhibit orientational order with a spontaneous twist and \emph{smectics} where, in addition, to orientational order, the molecules tend to assemble into layers. 

Suppose that a nematic occupies the domain $\Omega\subset\mathbb{R}^3$. {Locally, orientational order can be described by a parametrized probability density function $f:\Omega\times\mathbb{S}^2\to\mathbb{R}$ that measures the likelihood that a liquid crystalline molecule near $x\in\Omega$ is oriented within a given solid angle in $\mathbb S^2$.} For nematics, the probability of finding the head or the tail of a molecule pointing in a given direction are always the same, hence $f(x,-\mathbf{m})=f(x,\mathbf{m})$ for every $\mathbf{m}\in\mathbb{S}^2$ and $x\in\Omega$. 

A practically useful approach to describe a probability distribution over $\mathbb{S}^2$ is to generate its moments over the sphere.  In classical Landau-de Gennes theory for nematics the invariance of $f(x,\cdot)$ with respect to inversions guarantees that the first nontrivial moment of $f(x,\cdot)$ for every $x\in\Omega$ is the second moment 
\[Q(x):=\left<\mathbf{m}\otimes\mathbf{m}-\frac{1}{3}I(3)\right>_{f(x,\cdot)}.\]
Here the second-order tensor $Q(x)$ is symmetric and traceless and
$$
\left<h\right>_{f(x,\cdot)}:=\int_{\mathbb{S}^2}h(\mathbf{m})f(x,\mathbf{m})\,d\mathbf{m},
$$
for any map $h$ defined on $\mathbb{S}^2$. The liquid crystal is in the {\emph{uniaxial nematic}} phase if exactly two eigenvalues of $Q,$ e.g., $\lambda_1,\lambda_2$ are equal so that \begin{equation}
    \label{eq:nematic}
    Q=s\left(\mathbf{n}\otimes\mathbf{n}-\frac{1}{3}I(3)\right),
\end{equation}
where $s=\frac{3}{2}\lambda_3$ is \emph{the degree of orientation} of the nematic, $\mathbf{n}$ is the nematic \emph{director} and $\left(\lambda_3,\mathbf{n}\right)$ is an eigenvalue-eigenvector pair for $\mathcal{Q}$. On the other hand, when $Q=0,$ the liquid crystal has no orientational order and it is said to be in the \emph{isotropic} phase.  
The tensor $Q$ is the \emph{order parameter} of the Landau-de Gennes variational theory, in which equilibrium configurations of a nematic liquid crystals are assumed to minimize the (nondimensional) energy 
\begin{equation}
    \label{eq:ldgen}
E[Q]:=\int_\Omega\left[F\left(Q,\nabla Q\right)+\frac{1}{\delta^2}W\left(Q,\tau\right)\right]\,dx. 
\end{equation}
In this expression,  $F\left(Q,\nabla{Q}\right)$ is the orientational elastic energy, $W\left({Q},\tau\right)$ is the potential energy, $\tau$ is temperature and $0<\delta\ll1$ is the nematic coherence length. For a thermotropic nematic, there exists a critical temperature $\tau_c$ such that $W$ is minimized by the isotropic phase ${Q}=0$ when $\tau>\tau_c$ while it is minimized by any ${Q}$ of the form \eqref{eq:nematic} in the manifold of nematic states when $\tau<\tau_c$. We say that the liquid crystal undergoes an isotropic-to-nematic phase transition at $\tau_c$.

The most striking feature of a liquid crystal in a nematic phase are defect patterns of points, lines and walls that can be observed optically under crossed polarizers. Mathematically nematic defects are topological singularities of minimizers of \eqref{eq:ldgen}, associated with the nonlinear constraint $W({Q},\tau)=\min_{{Q}}\,W({Q},\tau)$. This constraint ensures that \eqref{eq:nematic} holds approximately on the entire domain $\Omega$, except for a singular set of a small measure (determined by the size of $\delta$) where the tensor ${Q}$ is {either biaxial or isotropic.} Understanding singularities of minimizers of the Landau-de Gennes energy has been a subject of extensive investigations in the last decade \cite{MR4314141},\cite{Can01}\nocite{Can02,CMR,CZ02,CZ01}-\cite{CTZ}, \cite{dFRSZ02,dFRSZ01,GM,MR4473115,HM,HMP}, \cite{INSZ01}\nocite{INSZ02,INSZ03}-\cite{INSZ04}, \cite{KRSZ,MZ,NZ}.

A relatively recent discovery of novel liquid crystalline phases formed by bent-core, banana-shaped molecules \cite{RevModPhys.90.045004,doi:10.1080/21680396.2013.803701} prompted modifications to the Landau-de Gennes theory to account for symmetries of these phases that do not exist in standard nematics \cite{brand2005tetrahedratic,pleiner2014low,LR2002,Radzihovsky_2001}. For example, it has been shown in \cite{LR2002} that an appropriate continuum theory in absence of positional order should depend on the first \emph{three} moments of an orientational probability distribution of V-shaped bent-core molecules. To this end, suppressing the dependence on $x$, recall that the probability density function can be expanded in terms of powers of $\mathbf{m}$ by using the Buckingham's formula \cite{Buckingham1967,Turzi} written as 
\begin{equation}
\label{eq:buckingham}
f(\mathbf{m}) = 
    \frac{1}{4\pi} \left(
    1 + \sum_{k=1}^\infty \frac{ (2k+1)!!}{ k!} {\left<\overbracket{\mathbf{m}^{\otimes k }} \right>}_f\cdot \mathbf{m}^{\otimes k} 
    \right),
\end{equation}
where 
\[
\mathbf{m}^{\otimes k} = \underbrace{\mathbf{m} \otimes \cdots \otimes \mathbf{m}}_k,
\]
the quantity $\overbracket{\mathcal{A}}$ is the symmetric traceless part of the tensor $\mathcal{A}$ and $``\cdot"$ denotes tensor contraction. By isolating the first three terms in \eqref{eq:buckingham} we obtain
\begin{equation}
    f(\mathbf{m})=\frac{1}{4\pi}\left(1+3\,\mathbf{p}\cdot\mathbf{m}+\frac{15}{2}\,Q\cdot(\mathbf{m}\otimes\mathbf{m})+\frac{35}{2}\,\mathcal{T}\cdot(\mathbf{m}\otimes\mathbf{m}\otimes\mathbf{m})+\ldots\right).
\end{equation}
In this expression, the first moment $\mathbf{p}$ is the \emph{polarization vector}, the second moment $Q$ is the $Q$-tensor defined in \eqref{eq:nematic} and the third moment describing tetrahedratic order is given by the third order tensor $\mathcal{T}$ with components
\[\mathcal{T}_{ijk}:=\left<m_im_jm_k-\frac{1}{5}\left(m_i\delta_{jk}+m_j\delta_{ik}+m_k\delta_{ij}\right)\right>_f.\]
Note that at fourth order, we retrieve a tensor that is useful in describing an ordering with cubic symmetry considered in \cite{Chemin2019,GMS}.

The appropriate Landau-de Gennes free energy functional can be constructed as a rotationally-invariant power series expansion around the isotropic state in the order parameters $\mathbf{p}$, $Q$ and $\mathcal{T}$ and their gradients. The contribution from the gradients of the order parameter fields is the elastic energy while the remaining terms that do not vanish in a spatially homogeneous material comprise the Landau-de Gennes potential. The coefficients of this potential, in general, are temperature-dependent and thus they control the structure of the minimal set of the potential, or the phase in which the material is observed at a given temperature. Given that the bent-core liquid crystals are described by three order parameters, there is a large number of possible phases that form via interactions between different material symmetries. For example, nematic order described by the standard second order tensor $Q$ may induce  tetrahedratic order described by the third order tensor $\mathcal{T}$ and vice versa via appropriate coupling terms \cite{LR2002}. If one were to neglect the contributions from lower order moments $\mathbf{p}$ and $Q$, the form of the Landau-de Gennes energy for bent-core liquid crystals for third order tensor fields \cite{LR2002} with the potential
\[
    W(\mathcal{T}) = \frac{\abs{\mathcal{T}}^4}{4} - \frac{\alpha}{2}\abs{\mathcal{T}}^2 + \frac{\beta}{4}\sum_{i, j = 1}^3 \langle \mathcal{T}_i, \mathcal{T}_j\rangle ^2.
\]
is a more complex version of the relaxed energy functional considered in this work. Because both functionals would require identical algebraic and analytical tools to obtain rigorous mathematical results, the present work can serve as a first step toward understanding of the Landau-de Gennes models for third-order tensors in higher dimensions. As a first step in this direction, inspired by \cite{LR2002}, in Appendix C we rigorously describe the minima of $W$. Note that recent analysis results \cite{MR2924435,MR3853605,MR3423208,doi:10.1137/16M1099789} have not considered phases of bent-core liquid crystals with tetrahedral symmetry.






\section{Symmetric, traceless 3rd order $\mathcal{Q}$-tensors}
\label{s:nD}


\subsection{Notation}

We will define a series of sub and affine spaces based on a set of vectors $\mathbf{u}^j \in \mathbb{R}^n$.  For a given vector $\mathbf{u}^j \in \mathbb{R}^n$, we can write it component-wise,  $\mathbf{u}^j = (u^j_1, \ldots, u^j_n)^T$.  Let $\mathbf{e}^1, \ldots, \mathbf{e}^n$ denote the canonical basis in $\R^n$.  Let $\mathbb{V}(n,\mathcal{R})$ be the set of all $n$-vectors with entries from a ring $\mathcal{R}$.  We will frequently drop $n$ when the dimension of the vector is clear.  In particular, $\mathbb{V}(n,\mathbb{R}) = \mathbb{R}^n$.

Let $\mathbb M(m, n,\mathcal{R})$ be the set of all $m\times n$ matrices with entries from a ring $\mathcal{R}$.  In particular, $\mathbb{M}(m,1,\mathcal{R}) = \mathbb{V}(m,\mathcal{R})$.  For a square $n\times n$ matrix with entries in $\mathcal{R}$, we write $\mathbb{M}(n,\mathcal{R})$.  We denote elements  $A  \in \mathbb{M}(m,n,\mathbb{R})$ with capital letters.  One particularly important class of matrices for us are projections $P^\ell \in \mathbb{M}(n,n,\mathbb{R})$ for a unit vector  $\mathbf{u}^\ell \in \mathbb{R}^n$ with $P^\ell_{j k} = u^\ell_j u^\ell_k = (\mathbf{u}^\ell \otimes \mathbf{u}^\ell)_{j k}$.   We also denote the $n\times n$ identity matrix, ${I}(n)$.

We next define generic $k$-th order tensors
\[
\mathbb{H}(n,k)
= \underbrace{\mathbb{R}^n \otimes \cdots \otimes \mathbb{R}^n}_{k \hbox{ times}}
\]
with elements $\mathcal{A} \in \mathbb{H}(n,k)$ that have indices $\mathcal{A}_{i_1 i_2 \ldots i_k}$ for $i_j \in\{1,\ldots,n\}$ with script letters.  We finally define symmetric $k$-order tensors as
\[
\Hsym(n,k) = \{ \mathcal{A} \in \mathbb{H}(n,k) \hbox{ such that } \mathcal{A}_{\sigma(i_1 \ldots i_k)} = \mathcal{A}_{i_1 \ldots i_k} \hbox{ for all } \sigma \in S_k\},
\]
where $S_k$ is the group of permutations of $k$-length words.
Finally, we define set of traceless, symmetric $k$-order tensors
\begin{equation} \label{e:tracelesstens}
\Htrace (n,k) = \{ \mathcal{A} \in \Hsym(n,k) \hbox{ such that } \sum_{j=1}^n \mathcal{A}_{a_1 a_2 \ldots a_{k-2}  j j } = 0 \hbox{ for } a_\ell \in \{1,\ldots, n\} \}.
\end{equation}
Given this notation, we now describe the specific class of  third order tensors that we will study.

\begin{remark}\label{rm:tensors_3_4_9_2}
We define an especially useful bijection $Q_{\mathcal{Q}}: \mathbb{H}(n,4) \to \mathbb{H}(n^2,2)$, where for any choice $i,j,k,\ell \in \{1,\ldots, n\}$,
\begin{equation} \label{e:tensor2matrix}
(Q_{\mathcal{Q}})_{(i-1)n+k,(j-1)n+\ell} = \mathcal{Q}_{ijk\ell}.
\end{equation}
This bijection identifies a canonical element of   $\mathbb{H}(n^2,2)=\mathbb{M}(n^2,n^2)$ with an element of $\mathbb{H}(n,4)$, and vice versa.  Given this bijection, we will frequently refer to elements in $\mathbb{H}(n,4)$ and $\mathbb{H}(n^2,2)$ interchangeably.

Likewise, we define the bijection $Q_{\mathcal{Q}}: \mathbb{H}(n,3) \to \mathbb{M}(n,n^2)$ by
\begin{equation}
\label{e:3tensor2matrix}
(Q_{\mathcal{Q}})_{i,(j-1)n+k} 
= \mathcal{Q}_{ijk},
\end{equation}
for $i,j,k \in \{1,\ldots, n\}.$
As with the prior definition, we will  refer to elements in $\mathbb{H}(n, 3)$ and $\mathbb{M}(n,n^2)$ interchangeably.
\end{remark}

\subsection{Elements of $\Htrace(n,3)$ generated by a frame with $n+1$-hedral symmetry}
\label{ss:nDresults}

We say a  collection of vectors  $\{\mathbf{u}^\ell \}_{\ell=1}^{n+1} \in \mathbb{S}^{n-1} \subset \mathbb{R}^n$ has "$n+1$-hedral symmetry" if the following condition holds: 
\begin{align} \label{e:innerprodcondition}
 \langle \mathbf{u}^j , \mathbf{u}^k \rangle & = -\frac{1}{n} + \frac{n+1}{n} \delta_{jk} \ \hbox{ for all } j,k \in\{1,\ldots, n+1\}.
\end{align}
Such collections satisfy the following result.
\begin{lemma}\label{e:propsvecsangles}
Suppose we have $n+1$ vectors $\{\mathbf{u}^j\}$ satisfying \eqref{e:innerprodcondition}, then
\begin{equation}
\label{e:spanRn} \mbox{any}\,\,\, n \,\,\, \mbox{of the vectors}\,\,\, \mathbf{u}^j, \,j=1, ..., n+1 \,\,\, \mbox{are linearly independent.}
\end{equation}
In particular, the vectors $\{\mathbf{u}^j\}_{j=1}^{n+1}$ span $\mathbb{R}^n$.  Furthermore
\begin{align}
\label{e:sumvecszero}
    \sum_{\ell=1}^{n+1} \mathbf{u}^\ell & = 0, \\
\label{e:sumprojections}
\frac{1}{n+1} \sum_ {\ell=1}^{n+1} P^\ell & = \frac{1}{n} {I}(n) 
\end{align}
where $P^k$ denotes the projection matrix generated by $\mathbf{u}^k$.
\end{lemma}
\begin{proof}
To prove \ref{e:spanRn} we consider for example $\mathbf{u}^j$, $j=1, ..., n$ and consider scalars $\alpha_j \in \mathbb{R}$ such that
$$
\alpha_1 \mathbf{u}^1 + ... + \alpha_n \mathbf{u}^n = 0.
$$
Taking the dot product of this with $\mathbf{u}^j$, $j=1,..., n$, gives
\begin{align*}
0 &= \alpha_1 - \frac{1}{n}\alpha_2 - ...- \frac{1}{n}\alpha_n, \\
&\vdots \\
0 &= -\frac{1}{n} \alpha_1 - ... - \frac{1}{n}\alpha_{n-1} + \alpha_n.
\end{align*}
This can be written in matrix form as follows:
$$
0 = \left (  \frac{n+1}{n}I(n) - \frac{1}{n}\boldsymbol{1}_n \boldsymbol{1}_n^T \right )  \left (  \begin{array}{c} \alpha_1 \\ \vdots \\ \alpha_n \end{array}\right ),
$$
where we use the notation $\boldsymbol{1}_n = \left (  \begin{array}{c} 1 \\ \vdots \\ 1 \end{array}\right ) \in \R^{n}$.  It is easy to check that $\frac{1}{n}\boldsymbol{1}_n \boldsymbol{1}_n^T$ is a rank-1, orthogonal projection matrix.  Hence, the matrix $ \frac{n+1}{n}I(n) - \frac{1}{n}\boldsymbol{1}_n \boldsymbol{1}_n^T$ is invertible.  It follows that $\alpha_1 = ... = \alpha_n = 0$, so $\mathbf{u}^j$, $j=1, ..., n$ are linearly independent.

Next, since $\{\mathbf{u}^\ell\}$ span $\R^n$ then $\mathbf{e}^k = \sum_{j=1}^{n+1} \alpha_{jk} \mathbf{u}^j$ for some constants $\alpha_{jk}$. This implies 
\begin{align*}
   \left\langle   \sum_{\ell =1 }^{n+1}\mathbf{u}^\ell , \mathbf{e}^k  \right\rangle
& = \left\langle \sum_{\ell=1}^{n+1} \mathbf{u}^\ell , (\sum_{j=1}^{n+1}\alpha_{jk} \mathbf{u}^j) \right\rangle = \sum_{j=1}^{n+1} \alpha_{jk} \left( \sum_{\ell=1}^{n+1} \left\langle \mathbf{u}^j , \mathbf{u}^\ell  \right\rangle \right) \stackrel{\eqref{e:innerprodcondition}}{=} 0.
\end{align*}

Finally, we choose $\mathbf{w} \in \R^n$.  Since the frame $\{\mathbf{u}^\ell\}$ spans $\R^n$ then \eqref{e:sumvecszero} implies we can write $\mathbf{w} = \sum_{j=1}^n a_j \mathbf{u}^j$ for unique constants $\{a_j\}$. Therefore, 
\begin{align*}
    \sum_{\ell=1}^{n+1} P^\ell \mathbf{w} & = 
    \sum_{\ell=1}^{n} \left\langle (\sum_{j=1}^n a_j \mathbf{u}^j) , \mathbf{u}^\ell \right\rangle \mathbf{u}^\ell
    - \left\langle (\sum_{j=1}^n a_j \mathbf{u}^j) ,    \mathbf{u}^{n+1}\right\rangle \sum_{i=1}^n  \mathbf{u}^i \\
    & \stackrel{\eqref{e:innerprodcondition}}{=} \sum_{\ell=1}^{n} \left( a_\ell-  {1\over n}\sum_{1 \leq j \leq n, j \neq \ell} a_j \right) \mathbf{u}^\ell
    + {1\over n} \left( \sum_{j=1}^n a_j  \right)  \sum_{\ell=1}^n  \mathbf{u}^\ell = {n+1 \over n} \mathbf{w}.
\end{align*}
\end{proof}

We associate with these vectors a third order tensor $\mathcal{Q} \in \Hsym(n, 3)$, defined as
\begin{equation} \label{e:def3tensor}
	\mathcal{Q}_{ijk} = \sum_{\ell=1}^{n+1} u^\ell_i (P^\ell)_{jk} = \sum_{\ell=1}^{n+1} u^\ell_i u^\ell_{j} u^\ell_k.
\end{equation}
For simplicity, throughout this section we write $\mathcal{Q}_{i} = \sum_{\ell=1}^{n+1} u^\ell_i P^\ell$ to denote the $i$-th $n \times n$ submatrix of $\mathcal{Q}$. 
This tensor $\mathcal{Q}$ generated by a set of $n+1$ vectors such that \eqref{e:innerprodcondition}-\eqref{e:spanRn} hold, satisfy the following identities: 
\begin{proposition} \label{p:3tensorlaws}
Given $\{\mathbf{u}^j\}_{j=1}^{n+1}$ satisfying \eqref{e:sumvecszero} and let $\mathcal{Q}$ be defined as in \eqref{e:def3tensor} then the following holds:
\begin{align}
\label{e:traceless}
\tr\mathcal{Q}_\ell &  = \mathcal{Q} I(n) = 0, \\
\label{e:suminteriorprod}
\mathcal{Q} \mathcal{Q}^T 
& =  {(n+1) (n^2 -1) \over n^3 } {I}(n), \\
\label{e:eigenvector}
\mathcal{Q} \mathbf{u}^k
      & = {n + 1 \over n} \left( P^k - {1 \over n } I(n) \right), \\
\label{e:eigentensor}
\mathcal{Q} P^k
      & = {n^2 -1 \over n^2} \mathbf{u}^k.
\end{align}
\end{proposition}

\begin{proof}
First we have from  \eqref{e:sumvecszero}, 
\begin{align*}
\tr\mathcal{Q}_i &  = \sum_{\ell = 1}^{n+1} u_i^\ell \tr  P^\ell 
 = \sum_{\ell = 1}^{n+1} u_i^\ell  = 0,
\end{align*}
which proves \eqref{e:traceless}.  
To prove \eqref{e:suminteriorprod} we write
\begin{align*}
\mathcal{Q}  \mathcal{Q}^T & 
= \sum_{\ell = 1}^n \mathcal{Q}^2_\ell 
 = \sum_{\ell = 1}^{n+1} \sum_{m=1}^{n+1} ( \sum_{k=1} u_k^\ell u_k^m ) P^\ell P^m  
 = \sum_{\ell = 1}^{n+1} \sum_{m=1}^{n+1} \left\langle \mathbf{u}^\ell , \mathbf{u}^m \right\rangle^2 ( \mathbf{u}^\ell \otimes \mathbf{u}^m ) \\
 & = \sum_{\ell = 1}^{n+1} P^\ell
 + {1\over n^2} \sum_{j < k} \mathbf{u}^j \otimes \mathbf{u}^k + \mathbf{u}^k \otimes \mathbf{u}^j, 
\end{align*} 
and since $0 \stackrel{\eqref{e:sumvecszero}}{=} ( \mathbf{u}^1+ \cdots + \mathbf{u}^{n+1}) \otimes ( \mathbf{u}^1+ \cdots + \mathbf{u}^{n+1}) = \sum_{\ell = 1}^{n+1} P^\ell + \sum_{j < k} \mathbf{u}^j \otimes \mathbf{u}^k + \mathbf{u}^k \otimes \mathbf{u}^j$  then
\[
\mathcal{Q}\mathcal{Q}^T = \left( 1 - {1\over n^2} \right)\sum_{\ell = 1}^{n+1} P^\ell \stackrel{\eqref{e:sumprojections}}{=} {(n+1)(n^2 - 1) \over n^2 } I(n),
\]
which implies \eqref{e:suminteriorprod}.
Next, we prove  identity \eqref{e:eigenvector},
\begin{align*}
    \mathcal{Q}\mathbf{u}^k
    & = \sum_{\ell = 1}^{n+1} \left\langle P^\ell \mathbf{u}^\ell , \mathbf{u}^k  \right\rangle
\stackrel{\eqref{e:innerprodcondition}}{=} P^k - {1\over n} \sum_{\ell \neq k} P^\ell \\
& \stackrel{\eqref{e:sumprojections}}{=} P^k - {1 \over n} \left(
{n+1 \over n} I(n) - P^k 
\right) = { n+1 \over n} P^k - { n+1 \over n^2} I(n).
\end{align*}
Finally, \eqref{e:eigentensor} follows directly from \eqref{e:eigenvector}.
\end{proof}

One simple consequence of \eqref{e:traceless} is that the tensors defined by \eqref{e:def3tensor} live in $\Htrace(n,3)$.

\subsection{Linear algebraic results for  $\Hsym(n,3)$ and $\Htrace(n,3)$}


Our first results provides a minimal representation for the unknowns in our space of symmetric, trace-free 3-tensors in $n$ dimensions.  This will be used in later sections for numerically computing tetrahedral frame fields.  

\begin{lemma} \label{l:uniqueelements}
Let $n$ denote the number of variables then the number of 
unique monomials in $\Htrace(n,3)$ 
is 
\begin{equation} \label{eq:number3monomials}
{n (n+4) (n-1) \over 6}.
\end{equation}
\end{lemma}
\begin{proof}
Note that the number of monomials of degree 3 of $n$ variables is ${3 + n -1 \choose n-1}$.
There are additionally $n$ constraints, since $\tr \mathcal{Q}_j = 0$, so 
the total number of unique elements is 
\[
{
    3 + n - 1 \choose n-1} - n = {n (n+4) (n-1) \over 6}.
\]
\end{proof}




We conclude with a few results on vector / matrix pairings that our tensors satisfy.  We will discuss these identities in the context of eigentensor-eigenvector pairings in Subsection~\ref{ss:3Dresults}.  The following result was established by Qi \cite{qi2017transposes} in the case of $\Hsym(3,3)$.  We generalize the result to arbitrary dimensions here.

\begin{proposition} \label{p:etevpairs}
Let $\mathcal{Q}\in \Hsym(n, 3)$ with $n$ nonzero singular values.  There are matrices $B^k \in \Hsym(n, 2)$, $k=1, ..., n$, unit vectors $\mathbf{f}^k \in \R^{n}$, $k=1, ..., n$, and scalars $\lambda_k \in \R{}$, $k=1, ..., n$, such that
$$
\langle B^i, B^j\rangle = \langle \mathbf{f}^i, \mathbf{f}^j\rangle = \delta_{ij},
$$
$$
\sum_{k=1}^n \langle B^k, \mathcal{Q}_k \rangle \mathbf{e}^k = \lambda_k \mathbf{f}^k,
$$
$$
\sum_{k=1}^n \langle \mathbf{f}^k, \mathbf{e}^k\rangle \mathcal{Q}_k = \lambda_k B^k,
$$
and
$$
\mathcal{Q} = \sum_{k=1}^n \lambda_k \mathbf{f}^k \left ({\mathbf{X}}_{B^k}\right)^T.
$$

\end{proposition}

\begin{proof}
Let us first notice that $\mathcal{Q}\mathcal{Q}^T \in \Hsym(n, 2)$, so there is $R\in O(n)$ such that
$$
\mathcal{Q}\mathcal{Q}^T = RDR^T,
$$
where $D$ is a diagonal matrix.  Furthermore, $\mathcal{Q}\mathcal{Q}^T$ is non-negative definite.  Hence we can call the elements of its diagonal
$$
(\lambda_1)^2 \geq ... \geq (\lambda_n)^2 > 0,
$$
due to the assumption on the singular values.
Let now 
$$
\mathcal{R}_R = \left ( \begin{array}{ccc} R_{11}R & \hdots & R_{1n}R \\ \vdots & \ddots & \vdots \\ R_{n1}R & \hdots & R_{nn}R \end{array}\right ),
$$
and define 
$$
\mathcal{S} = R^T \mathcal{Q}\mathcal{R}_R.
$$
By construction, $\mathcal{S}\in \Hsym(n, 3)$, and if we write $\mathcal{S} = (\mathcal{S}^1 \ ... \ \mathcal{S}^n)$, then $\mathcal{S}^k \in \Hsym(n, 2)$, and
$$
\langle \mathcal{S}^i, \mathcal{S}^j\rangle = (\lambda_i)^2 \delta_{ij}.
$$
By the first claim of Proposition 8.1 of the main draft, we have
$$
\mathcal{S}= \sum_{k=1}^n \mathbf{e}^k (\mathbf{X}_{\mathcal{S}^k})^T,
$$
so
$$
\mathcal{Q}= \sum_{k=1}^n R\mathbf{e}^k (\mathbf{X}_{\mathcal{S}^k})^T\mathcal{R}_R^T.
$$
To finish the proof we remember that equation 8.7 of the main draft tells us that
$$
\mathcal{R}_R \mathbf{X}_{\mathcal{S}^k} = \mathbf{X}_{R\mathcal{S}^k R^T}.
$$
Defining $\mathbf{f}^k = R\mathbf{e}^k$, and $B^k = R\mathcal{S}^k R^T$, we obtain the conclusion of the proposition.
\end{proof}

If we consider three tensors generated by  \eqref{e:def3tensor} via $n+1$ vectors satisfying \eqref{e:sumvecszero}, then we can get explicit eigenvector-eigentensor pairs.

\begin{proposition} \label{p:etevpairs2}
For every $n \geq 2$ there exists an $n\times (n+1)$ matrix, $C^n$ such that the following holds. For any set of $n+1$ unit vectors in $\R^n$,  $\{\mathbf{u}^j\}_{j=1}^{n+1}$, that satisfy the inner product condition \eqref{e:innerprodcondition}, and its associated 3-tensor $\mathcal{Q} \in \Htrace(n, 3)$ generated by these $\mathbf{u}^j$'s via \eqref{e:def3tensor}.  Let $\mathcal{Q} = (\mathcal{Q}_1 | ... | \mathcal{Q}_n)$ where the $\mathcal{Q}_j \in \Htrace(n, 2)$ are defined by
$$
\mathcal{Q}_j = \sum_{k=1}^n \langle \mathbf{u}^k, \mathbf{e}^j\rangle \mathbf{u}^k(\mathbf{u}^k)^T.
$$
If we denote $A_n$ the $n\times (n+1)$ matrix whose columns are the vectors $\{\mathbf{u}^j\}_{j=1}^{n+1}$.  Then, the matrix
$$
R = \frac{n}{n+1}A_nC_n^T \in O(n)
$$
is orthogonal, and $A_n = RC_n$.  Letting $\mathbf{f}^k = R\mathbf{e}^k$, $k=1, ..., n$, and
$$
B^k = \frac{1}{\lambda_n}\sum_{j=1}^n\langle \mathbf{f}^k, \mathbf{e}^j\rangle \mathcal{Q}_j,
$$
where
$$
\lambda_n = \left (\frac{(n^2-1)(n+1)}{n^3}\right )^{\frac{1}{2}},
$$
then, the vectors $\mathbf{f}^k$ along with the tensors $B^k$, $k=1, ..., n$ are eigenvector-eigentensor pairs for $\mathcal{Q}$ in the sense of Qi (see Remark~\ref{r:xiconditions}).  In particular, 
$$
\langle \mathbf{f}^j, \mathbf{f}^k \rangle = \langle B^j, B^k \rangle = \delta_{jk},
$$
$$
\sum_{i=1}^n \langle B^k, \mathcal{Q}_i\rangle \mathbf{e}^i = \lambda_n \mathbf{f}^k, \,\,\,\mbox{and}\,\,\, \sum_{j=1}^n\langle \mathbf{f}^k, \mathbf{e}^j\rangle \mathcal{Q}_j = \lambda_n B^k.
$$
Finally, one can recover the $3$-tensor $\mathcal{Q}$ from the eigenvector-eigentensor pairs, in the sense that
$$
\mathcal{Q}_i = \sum_{j=1}^n \langle \mathbf{f}^j, \mathbf{e}^i\rangle B^j,
$$
for $i=1, ..., n$.

\end{proposition}

The proof of this proposition can be found in Appendix A. 

\begin{remark}\label{r:xiconditions}
When $n=3$, the vectors $\{\mathbf{r}^j\}_{j=1}^3$ and tensors $\{V^j\}_{j=1}^3$ form  eigenvector-eigentensor pairs that are identical to those of Qi:
\begin{equation}
\begin{split}
    \mathcal{Q}\mathbf{r}^j & = \sigma V^j \\
    \mathcal{Q} V^j & = \sigma \mathbf{r}^j \\
    \left\langle \mathbf{r}^j , \mathbf{r}^k \right\rangle & = \left\langle V^j , V^k \right\rangle = \delta_{jk},
    \end{split}
\end{equation}
where the eigenvalue $\sigma = -\sqrt{32 \over27}$. 
\end{remark}

\subsection{Rotations and tetrahedral frames}
To conclude this section, we introduce two canonical  tetrahedral frames that will be utilized below: $\{\mathbf{u}_0^j\}_{j=1}^4$ and $\{\mathbf{v}_0^j\}_{j=1}^4$. The first set of vectors includes a vector aligned with $\mathbf{e}^3$: 
\begin{align*}
	\mathbf{u}^1_0 \equiv \begin{pmatrix} 0 \\ 0 \\ 1\end{pmatrix}, \ 
	\mathbf{u}^2_0 \equiv \begin{pmatrix} \sqrt{8 \over 9} \\ 0 \\ - {1\over 3} \end{pmatrix}, \
	\mathbf{u}^ 3_0 \equiv \begin{pmatrix}   -\sqrt{2 \over 9} \\ \sqrt{2\over 3} \\ - {1\over 3} \end{pmatrix},  \
	\mathbf{u}^4_0 \equiv \begin{pmatrix}   -\sqrt{2 \over 9} \\ - \sqrt{2\over 3} \\ - {1\over 3} \end{pmatrix} .
\end{align*} 
There is an additional set of canonical  vectors typically associated with four vertices on a cube:
\begin{equation} \label{e:stdTet}
\mathbf{v}^1_0\equiv \frac1{\sqrt{3}}\begin{pmatrix} 1\\ 1\\ 1\end{pmatrix},
\mathbf{v}^2_0 \equiv \frac1{\sqrt{3}}\begin{pmatrix} 1\\ -1\\ -1\end{pmatrix},
\mathbf{v}^3_0 \equiv \frac1{\sqrt{3}}\begin{pmatrix} -1\\ 1\\ -1\end{pmatrix},
\mathbf{v}^4_0 \equiv \frac1{\sqrt{3}}\begin{pmatrix} -1\\ -1\\ 1\end{pmatrix},
\end{equation}
In both cases, $\left\langle \mathbf{u}^j_0 , \mathbf{u}^k_0 \right\rangle = \left\langle \mathbf{v}^j_0 , \mathbf{v}^k_0 \right\rangle = {4 \over 3} \delta_{jk}-{1 \over 3}$, and so the corresponding three tensors $\mathcal{Q}(\mathbf{u}^j_0)$ or $\mathcal{Q}(\mathbf{v}^j_0)$ enjoy the results of Proposition~\ref{p:3tensorlaws}.
One can rotate the set of vectors $\{\mathbf{v}^j_0\}$ into $\{\mathbf{u}^j_0\}$ with the rotation matrix,
\[
R_0 \equiv {1\over \sqrt{24}}
 \begin{bmatrix}
4 & - 2 &  - 2 \\
0 & 2 \sqrt{3} & - 2 \sqrt{3} \\
\sqrt{8} & \sqrt{ 8} & \sqrt{8}
\end{bmatrix}.
\]\
More generally, we can rotate into any tetrahedral configuration on $\mathbb{S}^2$ from either canonical set of tetrahedral vectors. In particular, 
$\mathbf{u}^{\sigma(j) }= R\mathbf{v}^{j}_0$ for a rotation matrix, $R = R(\mathbf{u}^j)$, and permutation operator on four elements, $\sigma$.  The corresponding three-tensor satisfies
\[\mathcal{Q}_{ijk}(\mathbf{u}^\ell)  = \sum_{\ell=1}^4 u^\ell_i u^\ell_j u^\ell_k = \sum_{\ell=1}^4 \sum_{p,q,r} (R_{i p}v^\ell_{0,p})(R_{j q}v^\ell_{0,q})
(R_{k r}v^\ell_{0,r})  
 = \sum_{p,q,r=1}^3 \mathcal{W}_{ijkpqr} \mathcal{Q}_{pqr}(\mathbf{v}_0).
\]
where 
\[ 
\mathcal{W}_{ijkpqr} = R_{ip}R_{jq}R_{kr}
\]
is an element of $\mathbb{H}(3,6)$.   We will use this rotational perspective in the proof of the recovery.

\section{Recovery of $n+1$-hedral frame in $n = 2,3$ dimensions}
\label{sec:rec}

In this section we show the converse of results in Subsection~\ref{ss:nDresults}, namely that elements in $\Htrace(n,3)$ with a specific nonlinear constraint produce unique $n+1$-hedral frame fields. 

\subsection{Tensors in $\Htrace(2,3)$ and MB frames}
Focusing on $n=2$, we show that the identities developed in Subsection~\ref{ss:nDresults} are necessary and sufficient to uniquely describe the associated $2+1$-frame.  These frames are characterized by three planar vectors with equal ${2\pi\over 3}$-angles between. Given the shape, they are commonly referred to as \emph{Mercedes-Benz frames}, though we will refer to them as \emph{MB frames}.  Given its structure, it is sufficient to provide a single angle in $[0,{2\pi\over3})$ to fully characterize the frame.  Our result in this subsection is

\begin{theorem} \label{t:isomorphism2D}
The following diffeomorphism holds:
\begin{equation}  \label{e:isomercedes}
    SO(2)/D_3 \equiv  \Htrace(2,3) \cap \left\{ \mathcal{Q} \mathcal{Q}^T =  {9 \over 8} I(2) \right\},
\end{equation} 
where $D_3$ is the ${2 \pi \over 3}$-rotation group.
\end{theorem}

\begin{proof}

Consider three vectors that are ${2 \pi \over 3}$ rotations of each other.  If we let $$R_{2\pi \over 3} = \begin{pmatrix} -{1\over 2} & -{\sqrt{3} \over 2} \\ {\sqrt{3} \over 2} & - {1\over 2} \end{pmatrix}\hbox{ and }
R_{4\pi \over 3} = \begin{pmatrix} -{1\over 2} & {\sqrt{3} \over 2} \\ -{\sqrt{3} \over 2} & - {1\over 2} \end{pmatrix}$$
denote  ${2 \pi \over 3}$ and ${4 \pi \over 3}$ rotation, and if $\theta$ denotes the angle off the $x$-axis, then an MB frame can be described by three vectors 
\begin{equation} \label{e:threevecs}
	\mathbf{u}^1  = \begin{pmatrix} \cos(\theta) \\ \sin(\theta) \end{pmatrix}, \  
	\mathbf{u}^2  = R_{2\pi \over 3} \begin{pmatrix} \cos(\theta) \\ \sin(\theta) \end{pmatrix},  \ 
	\mathbf{u}^3 = R_{4\pi \over3} \begin{pmatrix} \cos(\theta) \\ \sin(\theta) \end{pmatrix}.
\end{equation}
Consequently, $\mathbf{u}^i \cdot \mathbf{u}^j  = -{1\over 2} + {3\over 2} \delta_{ij}$.  Recalling our definition $\mathcal{Q}_{ijk} = \sum_{\ell=1}^3 {u}^\ell_i{u}^\ell_j{u}^\ell_k,$ then Proposition~\ref{p:3tensorlaws} and Lemma~\ref{l:uniqueelements} imply $\mathcal{Q}$ is a symmetric, traceless 3-tensor with two unique elements, which implies it can be written as
\begin{equation} \label{e:2Dframeform}
	\mathcal{Q}\equiv \begin{pmatrix}
		q_1 & q_2 & q_2 & - q_1 \\
		q_2 & - q_1 & - q_1 & - q_2
	\end{pmatrix},
\end{equation}
where $\mathcal{Q}\mathcal{Q}^T = {9 \over 8} I(2)$. 
Since $ (\cos(\theta) , \sin(\theta))^T$ is part of the frame, then \eqref{e:eigentensor} implies
$$ \begin{pmatrix}
    \cos(2\theta) & \sin(2 \theta) \\
    -\sin(2\theta) & \cos(2 \theta)
\end{pmatrix} \begin{pmatrix} q_1 \\q_2 \end{pmatrix} = {3 \over 4} \begin{pmatrix} \cos(\theta) \\ \sin(\theta) \end{pmatrix}. 
$$
As a consequence, $q_1 = {3\over 4} \cos(3 \theta)$, $q_2 = {3 \over 4} \sin(3 \theta)$, and so
\begin{equation}\label{e:mercedesexplicit}
\mathcal{Q} \equiv \begin{pmatrix}
{3 \over 4} \cos(3 \theta) & {3 \over 4} \sin(3 \theta) & {3 \over 4} \sin(3 \theta) & -{3 \over 4} \cos(3 \theta) \\
{3 \over 4} \sin(3 \theta) & -{3 \over 4} \cos(3 \theta) & -{3 \over 4} \cos(3 \theta) & - {3 \over 4} \sin(3 \theta) 
\end{pmatrix}.
\end{equation}
Tensor \eqref{e:mercedesexplicit} recovers the expected three-fold symmetry of the MB frame in two dimensions, and it also provides an explicit representation for boundary alignment of a frame field, as will be discussed later. 

We now consider the converse. Suppose $\mathcal{Q} \in \Htrace(2,3)$ with 
 $\mathcal{Q}\mathcal{Q}^T = {9 \over 8} I(2)$ then the representation \eqref{e:2Dframeform} and its nonlinear constraint implies
$2 q^2_1 + 2 q^2_2 = 9/8$.
Consequently,  
\begin{equation} \label{e:2Dthreefold}
q_1 = {3\over4}\cos(\phi) \hbox{ and } q_2 = {3\over 4}\sin(\phi)
\end{equation}
for some angle $\phi$.  Now, assuming that some vector $(\cos(\theta),\sin(\theta))^T$ is part of the MB frame, and if 
 $\arg: \R^2\backslash\{0\} \mapsto [0,2\pi)$ returns the unique angle in $[0,2\pi)$ associated to the ordered pair off the $x$-axis, then \eqref{e:mercedesexplicit} implies
\begin{equation} \label{e:onethirdratio}
\theta = {1\over 3} \arg(q_1,q_2).
\end{equation}
Therefore, our tensor $\mathcal{Q}$ is determined by a unique angle $\theta \in [0,{2 \pi \over 3})$, and since that angle retrieves the other two vectors by ${2\pi \over 3}$ and ${4\pi \over 3}$ rotations of the vector associated to \eqref{e:onethirdratio}, we retrieve the full MB frame.
\end{proof}

\subsection{Tensors in $\Htrace(3,3)$ and tetrahedral frames}
\label{ss:3Dresults}

We turn to $n=3$ and show our the identities are necessary and sufficient to uniquely describe an tetrahedral frame.  Our result is
\begin{theorem}\label{t:recovery_3_d}
We have the following diffeomorphism:
\begin{equation}
\label{e:isotetra}
    SO(3)/T \equiv  \Htrace(3,3) \cap\left\{ \mathcal{Q} \mathcal{Q}^T =  {32 \over 27} I(3)   \right\},
\end{equation} 
where $T$ is the tetrahedral group.  

More explicitly, for every $\mathcal{Q}\in \Htrace(3, 3) \cap\left\{ \mathcal{Q} \mathcal{Q}^T =  {32 \over 27} I(3)   \right\}$ 
there are vectors $\mathbf{b}^j \in \mathbb{S}^2$, $j=1, ..., 4$, such that
$$
\langle \mathbf{b}^i, \mathbf{b}^j \rangle = \frac{1}{3} \left ( 4\delta_{ij} - 1 \right ),
$$
and such that
$$
\mathcal{Q} = \sum_{j=1}^4 \mathbf{b}^j \left (  \mathbf{X}_{\mathbf{b}^j \left ( \mathbf{b}^j \right )^T} \right )^T.
$$
The four vectors $\mathbf{b}^1,\ldots, \mathbf{b}^4$ are the four unique maximizers of 
\begin{equation}\label{e:def_mu}
\mu_{\mathcal{Q}}(\mathbf{a}) = {\rm det} \left( \sum_{j=1}^3 \langle \mathbf{e}^j, \mathbf{a} \rangle \mathcal{Q}_j\right).
\end{equation}
\end{theorem}

\begin{proof}
Given four tetrahedral vectors, the corresponding tensor $\mathcal{Q}$ lives in $\Htrace(3, 3) \cap\left\{ \mathcal{Q} \mathcal{Q}^T =  {32 \over 27} I(3)   \right\}$ due to our results in Section~\ref{s:nD}.  
The converse is much more involved, and  its proof can be found in Appendix B.  
\end{proof}

\begin{remark}Note that in \cite{GaetaVirga2019} eigenvalues and eigenvectors in the sense of \cite{qi2017transposes} were obtained for third order symmetric traceless tensors by maximizing the so-called octupolar potential
\[\Phi(\mathbf{x})=\mathcal{Q}_{ijk}x_ix_jx_k.\]
The potential $\Phi$ is different from $\mu_{\mathcal{Q}}(\mathbf{a})$ in Theorem \ref{t:recovery_3_d}. Both $\Phi$ and $\mu_{\mathcal{Q}}(\mathbf{a})$ have exactly the same maximizing set when $\mathcal{Q}\in SO(3)/T$, however the maximizing sets no longer coincide for $\mathcal{Q}\notin SO(3)/T$.
\end{remark}

\section{Ginzburg-Landau relaxation to the appropriate variety in 2 and 3 dimensions}
\label{sec:relax}

Since MB and tetrahedral frame fields can be identified by nonlinear sets in $\Htrace(n,3)$ for $n=\{2,3\}$, we propose a Ginzburg-Landau relaxation towards these constraints.  This procedure leads to a direct method for generating these frame fields on Lipschitz domains.


\subsection{MB Frames: $SO(2)/D_3 \equiv \Htrace(2,3) \cap \{ \mathcal{Q}\mathcal{Q}^T = {9\over 8} I(2)\}$}

MB-frame fields are an example of an $n$-direction field, in which each point in domain or tangent to a surface (see Section~\ref{sec:poi-hop}) is assigned $n$ evenly spaced vectors, see the review article \cite{Vaxman} for  background on this topic.  We now show how our framework generates an MB-frame on a two dimensional domain, outside of small number of singular sets and reduces to methods similar to those found in \cite{BeaufortPoincare,ViertelOsting} for 2-cross fields. 

By Theorem~\ref{t:isomorphism2D}, we can uniquely represent our MB frame field by an element of $\Htrace(2,3) \cap \{ \mathcal{Q} \mathcal{Q}^T = {9 \over 8} I(2) \}$.  However, not all maps with boundary data  in $\{ \mathcal{Q} \mathcal{Q}^T = {9 \over 8} I(2) \}$ extend smoothly into the interior.  


In particular from Theorem~\ref{t:isomorphism2D} we let
\begin{align*}
	\mathcal{Q}\equiv \begin{pmatrix}
		q_1 & q_2 & q_2 & - q_1 \\
		q_2 & - q_1 & - q_1 & - q_2
	\end{pmatrix},
\end{align*}
and if $\mathcal{Q}\mathcal{Q}^T = {9 \over 8} I(2)$, then
$(\cos(3\theta),\sin(3\theta)) = {4 \over 3} (q_1,q_2) $
          where $\theta  = {1\over 3} \arg(q_1,q_2)$.
Therefore, our MB frame can be generated by  determining  angle $\theta \in [0,{2 \pi \over 3})$ and computing  the other two vectors by ${2\pi \over 3}$ and ${4\pi \over 3}$ rotations of the vector associated to $\theta$.

 \begin{definition}\label{def:degree_first_def}
 Let $A:U\mapsto TU/\sim$ be a map with $A(x)=[(x,v)]$ with $v\neq 0$ except at
 isolated points,
 and $(x,v) \sim (x,e^{2\pi i/n} v)$ for some $n\in \mathbb{N}$.
 For a simple closed curve $\gamma:[0,1]\to U$ not meeting any of the zeroes, 
 we define the index of $A$ on $\gamma$ by finding a continuous 
 lifting via the universal cover, $\theta:[0,1]\to \R$,
with $\frac{v}{|v|}=e^{i\theta}$ and setting 
\[
\#(A,\gamma)=\frac1{2\pi} (\theta(1)-\theta(0)).
\]
 The index of an isolated zero $a$ is defined as that of $A$ 
on any closed, Jordan curve $\gamma:[0,1] \to U$ surrounding $a$ (in a counter clockwise manner) and no other zeroes. 
 We denote the index about this $a$ as 
\[
i(a) = \#(A,\gamma).
\]
\end{definition}
It is standard that the  index just defined takes values in 
$\frac 1n \mathbb{Z}$, compare \cite[p. 108]{Hopf} for the case $n=2$. 
We often use the words \emph{index}, \emph{degree} and \emph{winding number} interchangeably.
 

For this reason, instead of looking for smooth extensions of boundary data, we relax towards $SO(2)/D_3$ using the variety associated to this quotient.
In particular we look for $\mathcal{Q} (x) \in H^1(\Omega;\Htrace(2,3))$ such that $\mathcal{Q}\mathcal{Q}^T = {9\over 8} I(2)$ is satisfied
on the boundary, and the normal is a part of an MB frame.  
We can, therefore, define the space
\[
H^1_\nu(\Omega;\Htrace(2,3))
= \left\{  
\mathcal{A}  \in H^1(\Omega;\Htrace(2,3)) \hbox{ such that } 
\mathcal{A} (\nu \otimes \nu) = {3 \over 4} \nu
\hbox{ on } \p \Omega
 \right\}, 
\]
i.e. the normal, $\nu$, is part of the MB frame.
If the normal $\nu = (\cos(\theta(x)), \sin(\theta(x)))^T$ is smoothly defined on the boundary, it induces nontrivial topology in $\left.\mathcal{Q}\right|_{\p \Omega}$, due to \eqref{e:mercedesexplicit}, 
and this can preclude globally defined MB frame fields. 
In particular, by the constraint we can set $(q_1, q_2) = {3 \over 4}(u_1,u_2) = {3 \over 4 }\mathbf{u}$ with $\mathbf{u} \in \mathbb{S}^1$ and  $\mathbf{u} = (\cos(3\theta),\sin(3\theta)) \equiv g$ on $\p \Omega$.  Setting
$$H^1_{g}(\Omega;\mathbb{S}^1) = \{ \mathbf{v} \in H^1(\Omega; \mathbb{S}^1) \hbox{ such that } \left. \mathbf{v} \right|_{\p \Omega} = g \},$$
then we can try to minimally extend the boundary data into the interior subject to
\eqref{e:mercedesexplicit}, which entails minimizing
\[
\min_{\substack{\mathcal{Q} \in H^1(\Omega;\Htrace(2,3)) \\
\mathcal{Q}\mathcal{Q}^T = {9\over 8}I(2) \\
\left. \nu \right|_{\p \Omega} \in \mathcal{Q}}} {1\over2} \int_{\Omega} \left| \nabla  \mathcal{Q} \right|^2 dx
=  \min_{\mathbf{u} \in H^1_g (\Omega;\mathbb{S}^1)} {9 \over 8}\int_{\Omega} \left| \nabla  \mathbf{u} \right|^2 dx .
\]
However, by classical  arguments $H^1_g (\Omega;\mathbb{S}^1) = \varnothing$ for any domain  topologically equivalent to a disk.

To avoid such problems, we relax towards the manifold $\Htrace(2,3)\cap \{ \mathcal{Q}\mathcal{Q}^T = {9 \over 8} I(2) \}$.
In particular we take a sequence
 $\mathcal{Q}_\e \in  H^1(\mathbb{R}^2, \Htrace(3,2))$
 and penalize the distance from the variety $\{\mathcal{Q}\mathcal{Q}^T = {9 \over 8} I(2)\}$ by looking for minimizers of the associated Ginzburg-Landau functional, 
\begin{equation} \label{e:GL2D}
	\mathcal{E}^{2d}_\e(\mathcal{A}) \equiv {1\over2} \int_\Omega \left|\nabla \mathcal{A}\right|^2 + {1\over \e^2}\left| \mathcal{A} \mathcal{A}^T - {9\over 8} I(2)\right|^2 dx .
\end{equation}
We then consider a sequence 
\[
\mathcal{Q}_\e = \arg \min_{\mathcal{A} \in H^1_\nu(\Omega; \Htrace(2,3)) } \mathcal{E}^{2d}_\e(\mathcal{A})
\]
subject to the boundary conditions that have normal as a part of the MB frame which, in turn, induces the $3\theta$ dependency.

The condition that competitors on the boundary are MB frames containing the normal to $\partial\Omega$ can also be enforced in a weak sense by introducing the surface energy term that penalizes deviations from this condition. With the help of \eqref{e:eigentensor}, the surface energy can be taken in the form
\begin{equation}
    \label{eq:surfMB}
    \int_{\partial\Omega}\left\{\frac{1}{\delta_1^2}{\left|\mathcal{A}\left(\nu\otimes\nu\right)-\frac{3}{4}\nu\right|}^2+\frac{1}{\delta_2^2}{\left| \mathcal{A} \mathcal{A}^T - {9\over 8} I(2)\right|}^2\right\}\,dS,
\end{equation}
so that the energy functional is
\begin{multline}
     \label{eq:GL2DW}
	\mathcal{E}^{2d}_{\e,\delta_1,\delta_2}(\mathcal{A}) \equiv {1\over2} \int_\Omega \left|\nabla \mathcal{A}\right|^2 + {1\over \e^2}\left| \mathcal{A} \mathcal{A}^T - {9\over 8} I(2)\right|^2 dx\\ +\int_{\partial\Omega}\left\{\frac{1}{\delta_1^2}{\left|\mathcal{A}\left(\nu\otimes\nu\right)-\frac{3}{4}\nu\right|}^2+\frac{1}{\delta_2^2}{\left| \mathcal{A} \mathcal{A}^T - {9\over 8} I(2)\right|}^2\right\}\,dS.
\end{multline}
The interplay between the parameters $\delta_1,\ \delta_2$ and $\varepsilon$ will be discussed in Section \ref{sec:bcrelax} in the case of tetrahedral frames.
 
\subsection{Tetrahedral Frames: $SO(3)/T \equiv \Htrace(3,3) \cap \{ \mathcal{Q}\mathcal{Q}^T = {32\over 27} I(3)\}$}
\label{s:3Dsection}

We now turn our attention to foliating a Lipschitz domain with tetrahedral frame fields.  Using Theorem~\ref{t:recovery_3_d}, this is equivalent to looking for harmonic maps in $H^1 (\Omega ; \Htrace(3,3)\cap \{ \mathcal{Q}\mathcal{Q}^T = {32 \over 27} I(3)\})$.  However, as in the 2D problem, this manifold typically generates singularities due to boundary conditions, and so we will use a harmonic map relaxation with prescribed boundary conditions.   

Suppose that $\mathcal{Q} \in \Htrace(3,3)$ with $\mathcal{Q} \mathcal{Q}^T = {32 \over 27} I(3) $, we define the operator
$\mu_\mathcal{Q}: \Hsym(3,3) \mapsto \mathbb{R}$ via
\begin{equation} \label{e:mumaxQ}
\mu_\mathcal{Q}(\mathbf{b}) := \det \left(
\sum_{j=1}^3\mathcal{Q}_{j} \left< \mathbf{e}^j, \mathbf{b} \right>\right)
\end{equation}
then by \eqref{e:def_mu} in the proof of Theorem~\ref{t:recovery_3_d}, the four vectors which maximize \eqref{e:mumaxQ} define a unique tetrahedral frame in $SO(3)/T$.
Therefore, we can generate a tetrahedral frame field by filling our domain with harmonic maps in $\Htrace(3,3)\cap \{\mathcal{Q} \mathcal{Q}^T = {32 \over 27} I(3) \}$ and  generate the tetrahedral vectors at each point in the domain using maximizers of $\mu_\mathcal{Q}(\mathbf{b})$.  
However, as in the MB frame field situation, if we look to fill out our Lipschitz domain $\Omega \subset \mathbb{R}^3$ with tetrahedral frame fields that adhere to the boundary, we find topological obstructions.  
The challenge, again, is the nonlinear constraint, and so we again relax towards the nonlinear constraint using a different  nonlinearity.


We now describe this procedure in detail.  
From \eqref{eq:number3monomials} a general $\mathcal{Q} \in \Htrace(3,3)$ has
the $\# \{ \hbox{unique monomials} \} = 7$.
Letting
$\mathbf{q} = (q_1, \ldots, q_7)^T \in \mathbb{R}^7$ allows us to express  
\[
\mathcal{Q}(\mathbf{q}) = 
\begin{bmatrix}
	\begin{pmatrix}
		q_1 & q_2 & q_3 \\
		q_2 & q_4 & q_5 \\
		q_3 & q_5  & ( - q_1 - q_4)
	\end{pmatrix} & 
\begin{pmatrix}
	q_2 & q_4 & q_5 \\
	q_4 & q_6 & q_7 \\
	q_5 & q_7  & ( - q_2 - q_6)
\end{pmatrix} &
\begin{pmatrix}
	q_3 & q_5 & (-q_1-q_4) \\
	q_5 & q_7 & (-q_2-q_6) \\
	(-q_1-q_4) & (-q_2-q_6)  & ( - q_3 - q_7)
\end{pmatrix} 
\end{bmatrix}, 
\]
and our potential
\begin{equation} \label{e:Wdef}
\begin{split}
W(\mathcal{Q}(\mathbf{q})) & = 
 \left| \mathcal{Q}(\mathbf{q}) \mathcal{Q}^T(\mathbf{q}) - {32 \over 27} I(3) \right|^2 \\
&	= \left((q_1 + q_4)^2 + q_1^2 + 2q_2^2 + 2q_3^2 + q_4^2 + 2q_5^2 - {32\over 27} \right)^2 \\
	& \quad  + \left((q_2 + q_6)^2 + q_2^2 + 2q_4^2 + 2q_5^2 + q_6^2 + 2q_7^2 - {32\over 27} \right)^2 \\
	& \quad + \left(2(q_1 + q_4)^2 + 2(q_2 + q_6)^2 + (q_3 + q_7)^2 + q_3^2 + 2q_5^2 + q_7^2 - {32\over 27}\right)^2 \\
	& \quad + 2\left(2q_2q_3 - 2q_1q_5 - q_2q_7 + q_3q_6\right)^2 + 2(q_3q_4 - q_1q_7 - 2q_4q_7 + 2q_5q_6)^2 \\
	& \quad + 2\left(2q_1q_2 + 3q_2q_4 + q_1q_6 + 2q_3q_5 + 2q_4q_6 + 2q_5q_7\right)^2 .
	\end{split}
\end{equation}

Before generating the associated Ginzburg-Landau energy, we need to provide suitable boundary conditions that ensure the normal vector is included in the tetrahedral frame. 

\subsection{Boundary conditions and reduction to the MB frame}

In order to prescribe natural boundary conditions of the tetrahedral frame, we impose that the normal on the boundary $\nu$ comprises one of the four vectors of the frame.  Consequently, we arrive at the following conditions on the frame at the boundary, due to \eqref{e:eigenvector} and \eqref{e:eigentensor}:
\begin{lemma}\label{bdry_conds}
If $\nu$ is a normal on the boundary and  an element of the tetrahedral frame $\mathcal{Q} = \mathcal{Q}(\mathbf{q})$, then
\begin{align*}
\mathcal{Q}_{ijk}\nu_k & = {4\over 3} (\nu \otimes \nu)_{ij} - {4 \over 9} I(3), \\
\mathcal{Q}_{ijk}\nu_j \nu_k
& = {8 \over 9} \nu_i.
\end{align*}
Equivalently, if $\nu$ is the outer normal on the boundary and  an element of the tetrahedral frame $\mathcal{Q} = \mathcal{Q}(\mathbf{q})$, then
\begin{equation}\label{bdry_potential}
V(\mathcal{Q}, \nu) := \frac{1}{2}\abs{\mathcal{Q} \mathbf{X}_{\nu\nu^T} - \mu_3 \nu}^2 = 0.
\end{equation}
Here, for an $3\times 3$ matrix $A$, $\mathbf{X}_A$ is the vector in $\R^9$ that contains the columns of $A$ vertically, in order, and $\mu_3 = \frac{3^2-1}{3^2} = \frac{8}{9}$.

If this holds for  $\nu = (\nu_1, \nu_2, \nu_3)^T$, then we can solve for $\mathbf{q} = (q_1,\ldots, q_7)$ using the underdetermined system of equations
\begin{equation} \label{e:3Dboundarycondition}
    \begin{pmatrix}
    \nu_1 & \nu_2 & \nu_3 & 0 & 0 & 0 & 0 \\
    0 & \nu_1 & 0 & \nu_2 & \nu_3  & 0 & 0 \\
    -\nu_3 & 0 & \nu_1 & - \nu_3 & \nu_2 & 0 & 0 \\
    0 & 0 & 0 & \nu_1 & 0 & \nu_2  & \nu_3 \\
    0 & - \nu_3 & 0 & 0 & \nu_1 & - \nu_3 & \nu_2 \\
    - \nu_1 & - \nu_2 & - \nu_3 & -\nu_1 & 0 & -\nu_2 & - \nu_3 
    \end{pmatrix}
    \begin{pmatrix}
    q_1 \\ q_2 \\ q_3 \\ q_4 \\ q_5 \\ q_6 \\ q_7
    \end{pmatrix}
    = \begin{pmatrix}
    {4 \over 3} \nu_1^2 - {4 \over 9} \\  {4 \over 3} \nu_1 \nu_2 \\  {4 \over 3} \nu_1 \nu_3 \\  {4 \over 3} \nu_2^2 - {4 \over 9} \\  {4 \over 3} \nu_2 \nu_3 \\  {4 \over 3}  \nu_3^2 - {4 \over 9}
    \end{pmatrix}.
\end{equation}
In particular, the matrix on the left has rank 5.
\end{lemma}

If we assume the boundary is locally $\mathbb{R}^2$ then $\nu = (0,0,1)^T$ is an element of the frame.  We can then locally parametrize the orientation by planar rotation of $\theta$.  The corresponding $\mathcal{Q}$ becomes:
\[
Q = 
\begin{bmatrix}
	\begin{pmatrix}
		a & b & -{4 \over 9} \\
		b & -a & 0 \\
		-{4 \over 9} & 0  & 0
	\end{pmatrix} & 
\begin{pmatrix}
	b & -a & 0 \\
	-a & -b & -{4 \over 9} \\
	0 & -{4 \over 9}  & 0
\end{pmatrix} &
\begin{pmatrix}
	-{4 \over 9} & 0 & 0 \\
	0 & -{4 \over 9} & 0 \\
	0 & 0  & {8 \over 9}
\end{pmatrix} 
\end{bmatrix}.
\]
Furthermore, condition \eqref{e:suminteriorprod} implies $a^2 + b^2  = {32 \over 81}$ or 
\[
a = {4\sqrt{2} \over 9}\cos(\phi) \hbox{ and } b = {4\sqrt{2} \over 9}\sin(\phi).
\]
Given the three-fold symmetry of the remaining three vectors of the tetrahedral frame implies a direct analogue of the MB frame result: 
\begin{align*}
\mathcal{Q} = \begin{bmatrix}    
{4 \sqrt{2} \over 9} \cos(3 \theta) &  -{4 \sqrt{2} \over 9} \sin(3 \theta) &  -{4 \over   9} & -{4 \sqrt{2} \over 9} \sin(3 \theta) &  -{4 \sqrt{2} \over 9} \cos(3 \theta) &     0 & -{4 \over   9} &    0 &    0 \\
-{4 \sqrt{2} \over 9} \sin(3 \theta) &  -{4 \sqrt{2} \over 9} \cos(3 \theta) &     0 &  -{4 \sqrt{2} \over 9} \cos(3 \theta) &   {4 \sqrt{2} \over 9} \sin(3 \theta) &  -{4 \over   9} &    0 &  -{4 \over   9} &   0 \\
                  -{4 \over   9} &                       0 &     0 &                        0 &                     -{4 \over   9} &    0 &     0 &     0 &   {8 \over   9}
\end{bmatrix} .
\end{align*}

\subsection{Relaxation}
\label{sec:bcrelax}
We can now consider a harmonic map relaxation to the tetrahedral frame field using a Ginzburg-Landau formalism.
We first set 
\[
H^1_\nu(\Omega; \Htrace(3,3)) = \left\{ \mathcal{A} \in  H^1(\Omega; \Htrace(3,3)) \hbox{ such that } \mathcal{A} = \mathcal{A}(\mathbf{q}) \hbox{ satisfies } \eqref{e:3Dboundarycondition} \right\}.
\]
For elements of $H^1_\nu(\Omega; \Htrace(3,3))$, we define a relaxed energy, 
\begin{equation}
\label{eq:grafstrong}
\mathcal{E}^{3d}_\e(\mathcal{A}) \equiv {1\over 2} \int_{\Omega} \left| \nabla \mathcal{A} \right|^2 + {1\over \e^2} \left| \mathcal{A} \mathcal{A}^T - {32 \over 27} I(3) \right|^2 dx
\end{equation}
and consider
\[
\mathcal{Q}_\e = \arg\min_{\mathcal{A} \in H^1(\Omega;\Htrace(3,3))} \mathcal{E}^{3d}_\e (\mathcal{A}).
\]
As pointed out in Lemma \ref{bdry_conds}, the condition that the normal to the boundary $\nu$ is part of the tetrahedral frame $\mathcal{Q}$ can be conveniently imposed by assuming that
$$
V(\mathcal{Q}, \nu) = 0.
$$
This condition could also be imposed in a weak form through a boundary integral.  This leads us to the alternative energy
\begin{equation}
    \label{eq:graf}
    \mathcal{E}^{3d}_{\e, \delta_1, \delta_2}(\mathcal{A}) \equiv {1\over 2} \int_{\Omega} \left| \nabla \mathcal{A} \right|^2 + {1\over \e^2} W(\mathcal{A}) dx + \int_{\partial \Omega} \frac{V(\mathcal{A}, \nu)}{\delta_1^2} + \frac{W(\mathcal{A})}{\delta_2^2} \, dS,
\end{equation}
where
$$
W(\mathcal{A}) =  \frac{1}{2}\left| \mathcal{A} \mathcal{A}^T - {32 \over 27} I(3) \right|^2.
$$
For our next result, we will use the notation
$$
\Pi : \mathbb{H}(3, 3) \to \Htrace(3, 3)
$$
for the orthogonal projection from the set of rank-3 tensors $\mathbb{H}(3, 3)$ onto the our relaxation space $\Htrace(3, 3)$.
\begin{proposition}
Let $\gamma = \frac{\delta_2^2}{\delta_1^2}$.  The critical points of the energy $\mathcal{E}^{3d}_{\e, \delta_1, \delta_2}$ satisfy
$$
\abs{\mathcal{A}} \leq \max \left \{  \frac{8}{3}, \sqrt[3]{\frac{8\gamma}{3}} \right \},
$$
independent of $\e > 0$.
\end{proposition}

\begin{proof}
First we observe that the Euler-Lagrange equation along with boundary conditions for the energy $\mathcal{E}^{3d}_{\e, \delta_1, \delta_2}$ are
\begin{align*}
-\Delta \mathcal{A} + \frac{1}{\e^2}\Pi\left ( (\nabla_{\mathcal{A}} W)(\mathcal{A}) \right ) &= 0 \,\,\,\mbox{in}\,\,\,\Omega, \\
\nabla \mathcal{A} \cdot \nu + \Pi( \frac{1}{\delta_1^2}(\nabla_{\mathcal{A}} V)(\mathcal{A}, \nu)+ \frac{1}{\delta_2^2}(\nabla_{\mathcal{A}} W)(\mathcal{A})) &= 0 \,\,\,\mbox{on}\,\,\,\partial \Omega.
\end{align*}
Here,
$$
(\nabla_{\mathcal{A}} W)(\mathcal{A}) = 2(\mathcal{A}\mathcal{A}^T - \lambda_3 I(3))\mathcal{A},
$$
and
$$
(\nabla_{\mathcal{A}} V)(\mathcal{A}, p) = \mathcal{A} \mathbf{X}_{pp^T} \mathbf{X}_{pp^T}^T - \mu_n pX_{pp^T}^T = (\mathcal{A} \mathbf{X}_{pp^T} - \mu_3 p) \mathbf{X}_{pp^T}^T
$$
denote the gradients of the potentials with respect to $\mathcal{A}$, and $\lambda_3 = \frac{32}{27}$, $\mu_3 = \frac{8}{9}$.

Now, since $\mathcal{A}\in \Htrace(3, 3)$, we have $\Pi(\mathcal{A}) = \mathcal{A}$.  We also have $\abs{\mathcal{A}\mathcal{A}^T}^2 \geq \frac{\abs{\mathcal{A}}^4}{3}$, so we can estimate
\begin{equation}\label{eq:grad_W_dot_A}
\langle \mathcal{A}, \Pi \left ( \nabla_{\mathcal{A}} W)(\mathcal{A}) \right ) \rangle = 2(\abs{\mathcal{A}\mathcal{A}^T}^2 - \lambda_3 \abs{\mathcal{A}}^2 ) \geq \frac{2\abs{\mathcal{A}}^2}{3}(\abs{\mathcal{A}}^2 -3\lambda_3).
\end{equation}

Next, for $\abs{p}=1$, and using the facts that $\Pi(\mathcal{A}) = \mathcal{A}$, $\langle \mathcal{A}, \mathcal{A} \mathbf{X}_{pp^T} \mathbf{X}_{pp^T}^T \rangle \geq 0$, $\abs{p \mathbf{X}_{pp^T}^T} = 1$, we obtain
\begin{equation}\label{eq:grad_V_dot_A}
\langle \mathcal{A}, \Pi \left ( \nabla_{\mathcal{A}} V)(\mathcal{A}, p) \right )\rangle = \langle \mathcal{A},  \mathcal{A} \mathbf{X}_{pp^T}\mathbf{X}_{pp^T}^T - \mu_3 p \mathbf{X}_{pp^T}^T  \rangle \geq -\mu_3 \langle \mathcal{A},  p \mathbf{X}_{pp^T}^T \rangle \geq -\mu_3 \abs{\mathcal{A}}.
\end{equation}

Now, we take the inner product of the equation satisfied by $\mathcal{A}$ with $\mathcal{A}$, and use \eqref{eq:grad_W_dot_A}, to obtain
$$
\Delta \left (  \frac{\abs{\mathcal{A}}^2}{2} \right ) \geq \abs{\nabla \mathcal{A}}^2 + \frac{2\abs{\mathcal{A}}^2}{n\e^2}(\abs{\mathcal{A}}^2-3\lambda_3).
$$

On the other hand, using \ref{eq:grad_W_dot_A} and \ref{eq:grad_V_dot_A}, and writing $\gamma = \frac{\delta_2^2}{\delta_1^2}$, we obtain
$$
\langle \mathcal{A}, \Pi \left (  \frac{1}{\delta_1^2} (\nabla_{\mathcal{A}} V)(\mathcal{A}, \nu) + \frac{1}{\delta_2^2} (\nabla_{\mathcal{A}} W)(\mathcal{A})  \right )  \rangle \geq \frac{\abs{\mathcal{A}}}{\delta_2^2} \left (  \frac{2\abs{\mathcal{A}}^3}{3}-\gamma \mu_3 - 2\lambda_3 \abs{\mathcal{A}} \right ).
$$
From here we obtain
$$
\langle \mathcal{A}, \Pi \left (  \frac{1}{\delta_1^2} (\nabla_{\mathcal{A}} V)(\mathcal{A}, \nu) + \frac{1}{\delta_2^2} (\nabla_{\mathcal{A}} W)(\mathcal{A})  \right )  \rangle \geq 0 \,\,\,  \mbox{for}\,\,\,  \abs{\mathcal{A}} \geq \max\{ \sqrt{6\lambda_3}, \sqrt[3]{3\gamma \mu_3}\}.
$$
From the equation satisfied by $\abs{\mathcal{A}}^2$, and Hopf Lemma, we conclude that $\abs{\mathcal{A}} \leq \max\{ \sqrt{3\lambda_3}, \sqrt[3]{3\gamma \mu_3}\}$.  This is the conclusion of the Proposition.

\end{proof}

Before we use this formalism to generate tetrahedral fields for different Lipschitz domains in two and three dimensions, we first discuss the topology of $SO(3)/T$ and its effect on the frame field.

\medskip
\medskip
\medskip
\medskip

\section{Tetrahedral frames and quaternions}
\label{sec:qua}

In this section we discuss how 
tetrahedral frames can be described 
using quaternions as $S^3/2T$,
where $S^3$ are the unit quaternions
and $2T$ is a specific finite
subgroup. We show that 
a natural map from quaternions
to symmetric traceless tensors
induces an isometric
embedding of the space of tetrahedra.

We compute the fundamental 
group of the space of tetrahedra,
following \cite{MRV2021}. As 
the group is non-abelian, the free 
homotopy classes are characterized
by the conjugacy classes of the 
fundamental group, compare
also \cite{Mermin}, \cite{Trebin1984}
for similar topological considerations
in theoretical physics.


    
    
    
    

\subsection{Quaternions, rotations and tetrahedra}
It is a well known result that the group of unit quaternions $S^3$ can be used to describe rotations in $\R^3$. The following lemma
is standard:
\begin{lemma}
Set $\mathbf{q}=a+b\mathbf{i}+c\mathbf{j}+d \mathbf{k}\in S^3$ for $a,b,c,d \in \mathbb{R}$, then 
$\mathbf{q}\mapsto R_\mathbf{q}$
\[
R_\mathbf{q}=\left(\begin{array}{ccc} a^2+b^2-c^2-d^2 & 2bc-2ad & 2ac+2bd\\ 2ad+2bc & a^2+c^2-b^2-d^2 & 2cd-2ab\\ 2bd-2ac & 2ab+2cd & a^2+d^2-b^2-c^2 \end{array}\right)
\]
with $a^2 + b^2 + c^2 + d^2 = 1$
is a group homomorphism $S^3\to SO(3)$ with kernel $\{\pm 1\}$.
\end{lemma}
\begin{proof}
Note that $R_\mathbf{q}$ is the matrix representation of the map $\mathbf{p}\mapsto \mathbf{q}\mathbf{p}\mathbf{q}^{-1}$ for 
$\mathbf{p}=p_1\mathbf{i}+p_2\mathbf{j}+p_3\mathbf{k}$ a pure quaternion identified with a vector in $\R^3$. It is an easy computation 
that $\mathbf{q}\mathbf{p}\mathbf{q}^{-1}$ is a pure quaternion
and the matrix is orthogonal.
\end{proof}

We now define some useful groups:
The \textbf{tetrahedral group} $T$ is the subgroup of $R\in SO(3)$ that map the standard tetrahedron defined in \eqref{e:stdTet} to itself,
i.e. $R\mathbf{v}^0_j = \mathbf{v}^0_{\sigma(j)}$ where $\sigma$ is a permutation. It is well known that $T$ is isomorphic to the 
\textbf{alternating group} $A_4$: rotations preserve the orientation so cannot generate any transpositions. On the other hand, any $3$-cycle in $A_4$ can
be generated as a rotation around the leftover vector. 
Finally, we have the \textbf{binary tetrahedral group} $2T\subset S^3$, given by 
\[
2T= \{ \pm 1, \pm \textbf{i}, \pm \textbf{j}, \pm\textbf{k}, \frac12(\pm 1 \pm \textbf{i}\pm \textbf{j} \pm \textbf{k} )\},
\]
where each of the $\pm$ represents an independent choice so there are $24$ elements in total. It is generated by $\mathbf{s}=\frac12(1+\mathbf{i}+
\mathbf{j}+\mathbf{k})$ and $\mathbf{t}=\frac12(1+\mathbf{i}+
\mathbf{j}-\mathbf{k})$, which satisfy $\mathbf{s}^3=\mathbf{t}^3=-1$.

\begin{lemma}
The map $\mathbf{q}\mapsto R_\mathbf{q}$ induces a map $2T\to T$ that acts as follows on the standard tetrahedron:
\[
R_\mathbf{q}\mathbf{v}^0_j = \mathbf{v}^0_{\sigma(j)}  
\]
so there is an induced map $2T\to A_4$ that has the following representation:
\begin{center}
    \begin{tabular}{|c|c|}
        \hline 
    $2T$    & $A_4$ \\
\hline
$\pm 1$ &  $(1)$ \\
$\pm\mathbf{i}=\pm\mathbf{ts}$ & $(12)(34)$ \\
$\pm\mathbf{j}=\pm\mathbf{st}$ & $(13)(24)$ \\
$\pm\mathbf{k}=\pm\mathbf{s}\mathbf{t}^2\mathbf{s}$ & $(14)(23)$ \\
$\pm\frac12(1+\mathbf{i}+
\mathbf{j}+\mathbf{k}) =\pm \mathbf{s}$ & $(234)$ \\
$\pm\frac12(1-\mathbf{i}-
\mathbf{j}-\mathbf{k}) =\pm \mathbf{s}^{-1}$ & $(243)$ \\
$\pm\frac12(1+\mathbf{i}+
\mathbf{j}-\mathbf{k}) =\pm \mathbf{t}$ & $(123)$ \\
$\pm\frac12(1-\mathbf{i}-
\mathbf{j}+\mathbf{k}) =\pm \mathbf{t}^{-1}$ & $(132)$ \\
$\pm\frac12(1+\mathbf{i}-
\mathbf{j}+\mathbf{k}) =\pm \mathbf{st}^{-1}$ & $(142)$ \\
$\pm\frac12(1-\mathbf{i}+
\mathbf{j}-\mathbf{k}) =\pm \mathbf{ts}^{-1}$ & $(124)$ \\
$\pm\frac12(1-\mathbf{i}+
\mathbf{j}+\mathbf{k}) =\pm \mathbf{s}^{-1}\mathbf{t}$ & $(143)$ \\
$\pm\frac12(1+\mathbf{i}-
\mathbf{j}-\mathbf{k}) =\pm \mathbf{t}^{-1}\mathbf{s}$ & $(134)$ \\
\hline
    \end{tabular}
\end{center}
\end{lemma}
\begin{proof}
This is a straightforward computation.
\end{proof}
We see that the tetrahedral group is composed of
$\pm\frac{2\pi}{3}$ rotations around of the vectors of the
tetrahedron (these correspond to the $3$-cycles) and 
$\pi$ rotations around the axis going through the midpoints of 
two opposite edges (these correspond to the products of two
transpositions). The binary tetrahedral group has twice 
as many elements, and we have the following characterisation of 
its conjugacy classes:
\begin{lemma}
There are seven conjugacy classes of elements of $2T$:
\begin{center}
\begin{tabular}{|c|c|c|}
\hline
    Elements & Geodesic distance from $1$ & description \\
\hline
    $1$ & $0$ & identity \\
$\mathbf{s},\mathbf{t}^{-1},\mathbf{t}^{-1}\mathbf{s},\mathbf{s}\mathbf{t}^{-1}$
& $\frac\pi3$ & 1/3 rotation around one of the tetrahedral vectors \\
$\mathbf{s}^{-1}, \mathbf{t}, \mathbf{t}\mathbf{s}^{-1}, \mathbf{s}^{-1}\mathbf{t}$ &
$\frac\pi3$ & -1/3 rotation around one of the tetrahedral vectors \\
$\mathbf{s}^2, \mathbf{t}^{-2},\mathbf{t}^{-1}\mathbf{s}\mathbf{t}^{-1}\mathbf{s},
\mathbf{s}\mathbf{t}^{-1}\mathbf{s}\mathbf{t}^{-1}
$ & $\frac{2\pi}3 $ & -1/3 rotation around one of the tetrahedral vectors \\
$\mathbf{s}^{-2}, \mathbf{t}^2, \mathbf{t}\mathbf{s}^{-1}\mathbf{t}\mathbf{s}^{-1}, \mathbf{s}^{-1}\mathbf{t}\mathbf{s}^{-1}\mathbf{t}$ & $\frac{2\pi}3 $ & 1/3 rotation around one of the tetrahedral vectors \\
    $-1$ & $\pi$ & full rotation \\
    $\pm\mathbf{i},\pm\mathbf{j},\pm\mathbf{k}$ & $\frac\pi 2$ & 
    rotation interchanging two pairs of vectors \\
\hline
\end{tabular}    
\end{center}
\end{lemma}
\begin{proof}
We note that powers of $\mathbf{s}$ are conjugate to each other if 
and only if they are the same, which gives that the classes corresponding 
to $\mathbf{s}^0=1,\mathbf{s}^1, \mathbf{s}^2,\mathbf{s}^3=-1,\mathbf{s}^4=\mathbf{s}^{-2}, \mathbf{s}^5=\mathbf{s}^{-1}$.
are all separate. Conjugating each of these elements with $\mathbf{i}$, $\mathbf{j}$, $\mathbf{k}$ yields the rest of the conjugacy class. The elements 
$\pm \mathbf{i}, \pm \mathbf{j}, \pm \mathbf{k}$ form their own conjugacy class 
because the geodesic distance to $\pm 1$ is invariant under conjugation 
and we can compute $-\mathbf{iji}=\mathbf{j}$ and $\mathbf{t}^{-1}\mathbf{it}=
\mathbf{j}$ etc. 
\end{proof}

    


The configuration space of regular tetrahedra with vertices
on the unit sphere can be understood as follows. Let 
\[
\mathcal{M}_0=\left\{ (\mathbf{v}^1,\mathbf{v}^2,\mathbf{v}^3,\mathbf{v}^4) \in(\R^3)^4 : \left<\mathbf{v}^i,\mathbf{v}^j\right> = \frac43 \delta_{ij}-\frac13, \det(\mathbf{v}^1|\mathbf{v}^2|\mathbf{v}^3)>0  \right\}
\]
and $\mathcal{M}=\mathcal{M}_0/A_4$, where $A_4$ acts by permuting the
indices. Then $\mathcal{M}_0$ contains all collections
of oriented tetrahedra with indexed vertices and 
$\mathcal{M}$ the corresponding collection without numbered vertices.
\begin{proposition}
We can identify $\mathcal{M}=SO(3)/T=SU(2)/2T=S^2/2T$.
\end{proposition}
\begin{proof}
We can map $SO(3)$ into $\mathcal{M}$ by considering the 
action on a fixed standard tetrahedron. The image of two rotations
in $\mathcal{M}$ is clearly the same iff they differ by an element of $T$.
To see that this map is surjective, note that
for any $(\mathbf{v}^1,\mathbf{v}^2,\mathbf{v}^3,\mathbf{v}^4)\in \mathcal{M}_0$, we can choose $R\in SO(3)$ such that 
$\mathbf{v}^j=R\mathbf{v}^j_0$ for $j=1,\dots,4$, by first rotating $\mathbf{v}^1_0$ into $\mathbf{v}^1$ and then rotating around this axis 
to align the other vectors. 

That $SO(3)/T=SU(2)/2T$ is almost the definition of $2T$, the preimage
of $T$ under the double covering $SU(2)\to SO(3)$. 
\end{proof}



 %
\subsection{Embedding of tetrahedra into tensor spaces}

Our main result in this subsection establishes an isometry between $SO(3)/T$ and our tensor space $\Htrace(3,3) \cap \{ \mathcal{Q}\mathcal{Q}^T = {32\over 27} I(3)\}$. 

\begin{theorem}\label{thm55}
Taking the vectors of a standard tetrahedron, $\{\mathbf{v}^j_0\}_{j=1}^4$, given in \eqref{e:stdTet}, we can generate the $\mathbf{q}$-rotated 
tetrahedron 
\begin{equation}
\mathcal{T}(\mathbf{q}) = \sum_{\ell=1}^4 R_\mathbf{q} \mathbf{v}^\ell_0\otimes R_\mathbf{q} \mathbf{v}^\ell_0\otimes R_\mathbf{q} \mathbf{v}^\ell_0
\end{equation}
The map $\mathcal{T}: S^3\to \Htrace(3,3)$ given by
$q\mapsto \mathcal{T}(\mathbf{q})$ 
is up to scaling  a local isometry in the sense that
it satisfies $|d\mathcal{T}(\mathbf{q})(\mathbf{v})|^2
=\alpha_0 |\mathbf{v}|^2$ for all $\mathbf{q}\in S^3$ and 
$\mathbf{v}\in T_\mathbf{q} S^3$, the tangent space, where $\alpha_0=\frac{512}{9}$.

The induced map 
$\mathcal{T}:S^3/2T \to \Htrace(3,3)$ is well-defined and injective, 
and up to scaling we find an isometry 
$S^3/2T \equiv SO(3)/T \equiv \Htrace(3,3) \cap \{ \mathcal{Q}\mathcal{Q}^T = {32\over 27} I(3)\}$.
\end{theorem}
\begin{proof}
The heart of the matter for proving the local (scaled) isometry
character is to use that $S^3$ is a group, so it 
suffices to show this for $\mathbf{q}=1$. 
Writing a geodesic through $1$ as 
$\exp(t\mathbf{a})=\cos(t|\mathbf{a}|)+\frac{\mathbf{a}}{|\mathbf{a}|}\sin(t|\mathbf{a}|)$ 
for $\mathbf{a}=a_1 \mathbf{i} +a_2 \mathbf{j} +a_3 \mathbf{k}$, 
we have $\frac{d}{dt}\big|_{t=0} \exp(t\mathbf{a})=\mathbf{a}$ and
$\frac{d}{dt}\big|_{t=0} R_{\exp(t\mathbf{a})}\mathbf{v}=\mathbf{a}\mathbf{v}-\mathbf{v}\mathbf{a}=2\mathbf{a}\times\mathbf{v}$.
Let $\mathbf{u}^j$, $j=1,\dots,4$ be the four unit vector of
a tetrahedron. We compute
\[
A=A_1+A_2= |d\mathcal{T}(1)(\mathbf{a})|^2 = 12\sum_{j,k} \langle \mathbf{a}\times \mathbf{u}^j,
\mathbf{a}\times\mathbf{u}^k\rangle \langle\mathbf{u}^j,\mathbf{u}^k\rangle^2
+ 
24\sum_{j,k} \langle \mathbf{a}\times \mathbf{u}^j,
\mathbf{u}^k\rangle\langle\mathbf{u}^j,\mathbf{a}\times \mathbf{u}^k\rangle \langle\mathbf{u}^j,\mathbf{u}^k\rangle.
\]
Using $\langle\mathbf{u}^j,\mathbf{u}^k\rangle=\frac{4}{3}\delta_{jk}-\frac{1}{3}$ and $\langle\mathbf{u}^j,\mathbf{u}^k\rangle^2 = \frac{8}{9}\delta_{jk}+\frac{1}{9}$ yields
\[
A_1=\frac{32}{3}\sum_j |\mathbf{a}\times \mathbf{u}^j|^2.
\]
For $A_2$, we note that
\[
A_2= -24\sum_{j,k} \langle \mathbf{a}\times \mathbf{u}^j,
\mathbf{u}^k\rangle^2\left ( \frac{4}{3}\delta_{jk}-\frac{1}{3}\right ) = 8\sum_{j, k} \langle \mathbf{a}\times \mathbf{u}^j,
\mathbf{u}^k\rangle^2 = \frac{32}{3} \sum_{j}\abs{\mathbf{a}\times \mathbf{u}^j}^2,
\]
because $\langle \mathbf{a}\times \mathbf{u}^j, \mathbf{u}^j\rangle = 0$, and $\sum_j \mathbf{u}^j(\mathbf{u}^j)^T = \frac{4}{3} I(3)$.  So far then we have
\[
A = A_1 + A_2 = \frac{64}{3} \sum_{j}\abs{\mathbf{a}\times \mathbf{u}^j}^2.
\]
However, again using the fact that $\sum_j \mathbf{u}^j(\mathbf{u}^j)^T = \frac{4}{3} I(3)$, we obtain
\[
\sum_{j}\abs{\mathbf{a}\times \mathbf{u}^j}^2 = \frac{8\abs{\mathbf{a}}^2}{3}.
\]
Putting everything together, we then conclude that
\[
A = \frac{512}{9}\abs{\mathbf{a}}^2.
\]
This is a multiple of $|\mathbf{a}|^2$. The parallelogram identity
then implies that $\mathcal{T}$ is (up to scaling) a local isometry.

If $\mathcal{T}(\mathbf{q}_1)=\mathcal{T}(\mathbf{q}_2)$ then by Theorem~\ref{t:recovery_3_d} we have that 
$\{R_{\mathbf{q}_1}\mathbf{v}_0^j : j=1,..,4 \} = \{R_{\mathbf{q}_2}\mathbf{v}_0^j : j=1,..,4 \} $ so $R_{\mathbf{q}_1\mathbf{q}_2^{-1}} \in T$ 
and $\mathbf{q}_1\mathbf{q}_2^{-1} \in 2T$, so $\mathcal{T}$ is well-defined and injective on $S^3/2T$. It is surjective by 
Theorem~\ref{t:recovery_3_d}.


\end{proof}

\subsection{Homotopy of the set of tetrahedral frames}
With our identification of the space of tetrahedra as $S^3/2T$, we 
can now determine its fundamental group:
\begin{proposition}
The fundamental group of the space of tetrahedra is $\pi_1(\mathcal{M})
=\pi_1(SO(3)/T) = \pi_1(S^3/2T) = 2T$.
\end{proposition}
\begin{proof}
As $S^3$ is simply connected and locally path connected and $2T$ acts properly discontinuously
on $S^3$, this is a consequence of standard results on 
covering spaces, see e.g. \cite[Corollary III.7.3]{Bredon} 
or \cite[Proposition 9.1]{MRV2021}.
\end{proof}
The definition of $\pi_1$ means that its elements are 
represented by homotopy classes of (continuous) loops
where the homotopy keeps the start/end point $1$ fixed. 
The following result considers
also \emph{free} homotopies
where the start/end point is 
\emph{not} fixed throughout 
the homotopy.
\begin{proposition}
If $\gamma_1,\gamma_2:S^1\to S^3/2T$ are loops in $S^3/2T$ starting and
ending at $1$, then $\gamma_1$ and $\gamma_2$ are (freely) homotopic 
to each other if and only if their homotopy classes $[\gamma_1], [\gamma_2]\in \pi_1(S^3/2T)$ are conjugate,
i.e. if there exists $h\in 2T$ with
\[
[\gamma_1] = h [\gamma_2] h^{-1}.
\]
\end{proposition}
\begin{proof}
This is a standard result in elementary algebraic topology, see 
e.g. \cite[Proposition III.2.4]{Bredon}.
\end{proof}
As the identity is the only elements in its conjugacy class, 
it follows that for simply connected $\Omega$,
a  $g\in C^0(\partial\Omega,S^3/2T)$ has an 
extension $h\in C^0(\overline{\Omega},S^3/2T)$ if and only $[g]=1$.





For boundary conditions 
that cannot be resolved by a globally continuous map,
it is possible to find "topological resolutions". 
The following is a special case of the 
treatment in \cite[Section 2.1]{MRV2021}. 
\begin{definition}
Let $\Omega$ be a simply connected
sufficiently smooth domain in $\R^2$ and 
let $g\in C^0(\partial\Omega, S^3/2T)$. 
A collection of maps $\gamma_j\in C^0(S^1,S^3/2T)$
is called a \emph{topological resolution}
of $g$ if there exist distinct points $a_j\in\Omega$
and $\rho\in(0,\min(\{\frac12\dist(a_i,a_j): i\neq j\}
\cup \{ \dist(a_i, \partial\Omega)\})$ 
such that there is $h \in C^0(\overline{\Omega}
\setminus \bigcup_i B_\rho(a_i), S^3/2T)$
with $h=g$ on $\partial\Omega$ and $h(a_i+\rho z)
=\gamma_i(z)$ on each copy of $S^1$.
\end{definition}
\begin{proposition}
Let $\Omega$ be a simply connected
sufficiently smooth domain in $\R^2$ and 
let $g\in C^0(\partial\Omega, S^3/2T)$. 
A collection of maps $\gamma_j\in C^0(S^1,S^3/2T)$
is a {topological resolution}
of $g$ if and only if there exist 
$q_0,q_j \in 2T$ such that the homotopy classes $[g]$
of $g$ and $[\gamma_j]$ of $\gamma_j$ satisfy
\[
q_0^{-1}[g]q_0 = q_1^{-1}[\gamma_1]q_1 \dots q_k^{-1}
[\gamma_k]q_k,
\]
i.e. iff a conjugate of the homotopy class of the outer
boundary map $g$
can be written as a product of conjugates of the 
homotopy classes of the inner boundary maps $\gamma_j$.
\end{proposition}
\begin{proof}
This is a slight reformulation of the 
simplest case of \cite[Proposition 2.4]{MRV2021},
adapted to our special case.
We refer to that article for a discussion 
about the independence of this result from 
the order used in the product. 
\end{proof}

\begin{remark} \label{rm:topres}
There are several possible topological 
resolutions of the identity with different numbers
of homotopy classes at geodesic distance $\frac{\pi}{3}$:
We have a length zero resolution
 $1=1$, a length $2$ resolution
 $1=\mathbf{s} \cdot \mathbf{s}^{-1}$ 
 and a length $3$ resolution 
 $1=\mathbf{st}^{-1}\cdot \mathbf{t}\cdot \mathbf{s}^{-1}$.
 From these we can construct arbitrary longer resolutions. Note that $\mathbf{s}^6=1$ gives 
 another resolution. For this reason, in numerical simulations we observe local minimizers with higher number of singularities than what is expected for a global minimizer. The same phenomenon has been observed in the Landau-de Gennes context for $Q$-tensors describing nematic liquid crystals, where local minimizers with singularities were rigorously shown to exist \cite{ignat2020symmetry} for topologically trivial boundary data.
 \end{remark}

\subsection{Generating data in free homotopy classes of tetrahedral frames}
\label{sec:gendata}
Armed with the results of the previous subsection,
we can now try to interpret tensor-valued maps
with target into $\Htrace(3,3) \cap \{ \mathcal{Q}\mathcal{Q}^T = {32\over 27} I(3)\}$
except for a finite number of point singularities
as topological resolutions of their boundary data.

To construct a map with given homotopy types, we can 
use the following recipe:





\begin{definition}
For $\sigma=w+x \mathbf{i} +y \mathbf{j}+z \mathbf{k}\in S^3$, we set 
$s=\arg(w+i\sqrt{1-w^2})\in(-\pi,\pi]$ and let
\[
G_\sigma(t) = \cos(st) +\frac{(x \mathbf{i} +y \mathbf{j}+z \mathbf{k}) \sin(st)}{\sqrt{1-w^2}}  
\]
with $G_1(t)=1$ and $G_{-1}(t)=\cos(\pi t)+ i \sin(\pi t)$.
\end{definition}
\begin{proposition}
The $G_\sigma$ satisfies $G_\sigma:[0,1]\to S^3$ with $G_\sigma(0)=1$,
$G_\sigma(1)=\sigma$. $G_\sigma$ is a smooth geodesic.

For  $\sigma\in 2T$ the formula reduces to  the following cases:

If $\sigma=\frac12 (1\pm \mathbf{i} \pm \mathbf{j} \pm \mathbf{k})$
then
\[
G_\sigma(t) = \cos \frac{\pi t}{3} 
+ \frac1{\sqrt{3}} \sin\frac{\pi t}{3} (\pm \mathbf{i} \pm \mathbf{j} \pm \mathbf{k})
\]
For $\sigma=\pm \mathbf{i}$ ($\pm \mathbf{j}$ and $\pm \mathbf{k}$ are analogous):
\[
{G_\sigma(t) = \cos \frac{\pi t}{2} \pm \mathbf{i} \sin \frac{\pi t}{2}.}
\]
If  $\sigma=\frac12(-1\pm \mathbf{i} + \mathbf{j} + \mathbf{k})$ then
\[
G_\sigma(t) = \cos \frac{2\pi t}{3} 
+ \frac1{\sqrt{3}} \sin\frac{2\pi t}{3} (\pm \mathbf{i} \pm \mathbf{j} \pm \mathbf{k}).
\]
\end{proposition}
\begin{proof}
This is a straightforward computation.
\end{proof}
\begin{lemma}
Let $\alpha,\beta\in 2T$, $0<\rho<1$.
Then the map $F_{\alpha,\beta}: \overline{B_1}\setminus B_\rho
\to S^3$,
\[
F_{\alpha,\beta} \left(re^{i\theta}\right) = G_\beta\left(\frac{1-r}{1-\rho}\right)
G_\alpha\left(2\pi \theta\right)
\]
induces after taking the quotient modulo $2T$ a
topological resolution 
of its boundary map. The homotopy class 
of the outer boundary map is $\alpha$ and 
the homotopy class of the inner boundary is 
$\beta^{-1}\alpha\beta$.
\end{lemma}
\begin{proof}
We need to show that $F_{\alpha,\beta}$ induces
a continuous map. This follows after taking the 
quotient from the fact that 
$G_\alpha(1)=\alpha = \alpha G_\alpha(0)$,
with $\alpha \in 2T$.

\end{proof}
\begin{proposition}
Let $\alpha_1,\dots,\alpha_k, \beta_1, \dots, \beta_k\in 2T$ and $a_1,\dots, a_k\in B_1$ distinct points.
Let 
$\mu_a(z)=\frac{z-a}{1-\overline{a}z}$
be a M\"obius transformation
mapping $a$ to $0$.
Then 
\[
\mathbf{q}(z)= F_{\alpha_1,\beta_1}(\mu_{a_1}(z))
\dots F_{\alpha_k,\beta_k}(\mu_{a_k}(z))
\]
induces a topological resolution of its
boundary data of homotopy type
$\alpha_1\dots\alpha_k$.

Post-composing this with the map $\mathcal{T}(\mathbf{q})$
of Theorem~\ref{thm55} leads to a 
corresponding resolution in the space of tensors.
\end{proposition}
\begin{proof}
This is a direct consequence of the preceding 
computations.
\end{proof}

\section{Poincare-Hopf for MB frame-valued maps on generic surfaces}
\label{sec:poi-hop}

In this section we recall a version of the Poincare-Hopf theorem adapted to our situation.  This theorem can be easily adapted from page 112 of Heinz Hopf's book \cite{Hopf}. Similar results have been shown to be valid for cross-field-valued maps, \cite{BeaufortPoincare, Beben, FoggPoincare,ray:NSDF:2006}. Theorem~\ref{thm:poincarehopf} below  will be useful in interpreting our numerical simulations in Section~\ref{sec:numerics}.




Let $\Sigma \subset \R^{3}$ be a closed, smooth, orientable surface with normal $\nu$, and fix an integer $m\in \mathbb{N}$, $m \geq 2$. Consider the set
$$
\overline{M} = \{(x, \omega) \in \Sigma \times \mathbb{S}^2 : \nu(x) \cdot \omega = 0\}.
$$
For $(x_j, \omega_j)\in \overline{M}$, $j=1, 2$, define the equivalence relation
$$
(x_1, \omega_1) \sim (x_2, \omega_2) \,\,\iff \,\,  x_1 = x_2 \,\,\, \mbox{and}\,\,\, \omega_1\cdot \omega_2 = \cos\left ( \frac{2k\pi}{m}\right ) \,\,\, \mbox{for some}\,\,\,  k \in \mathbb{N}.
$$
Define then
$$
M = \overline{M}/\sim.
$$
In other words, we consider the unit tangent bundle of $\Sigma$, and identify tangent vectors that are related to each other by a rotation of an integer multiple of $\frac{2\pi}{m}$ about the normal.  Let also $P_1 : M \to \Sigma$ be the projection onto the first coordinate.
\begin{definition}
Let $\Lambda \subset \Sigma$ be a finite set.  An $m$-gon valued field on $\Sigma$ is a continuous map $Q : \Sigma\setminus \Lambda \to M$ such that $P_1(Q(x)) = x$ for every $x\in \Sigma$.
\end{definition}
In the above definition, if $x\in \Lambda$ and the field $Q:\Sigma\setminus \Lambda\to M$ cannot be extended to $x$ by continuity, we say that $x$ is a singularity of $Q$.
Next, if $x\in \Lambda \subset \Sigma$ is a singular point of $Q : \Sigma\setminus \Lambda \to M$, we can define its index.  

\begin{definition}\label{def:degree_second_def}
Let $x\in \Lambda \subset \Sigma$  be a singular point of $Q : \Sigma\setminus \Lambda \to M$.  Consider a closed, continuous curve $\gamma \subset \Sigma \setminus \Lambda$ surrounding $x$, small enough to be contained in a single coordinate patch, and such that $x$ is the only singularity surrounded by $\gamma$.  
Consider a continuous lifting of $Q$ along $\gamma$ through a unit tangent vector, in the sense that, at every point in $\gamma$, $Q$ can be obtained rotating the unit vector by $\frac{2\pi}{m}$, $m-1$ times.  Compute the angle between this unit vector and one of the coordinate tangents.  The total change of this angle as we travel through $\gamma$ once anti-clockwise, divided by $2\pi$, will be called the index of the singularity $x\in \Lambda$ and denoted $i(x)$.
\end{definition}
Note when $\Sigma \subset \mathbb{R}^2\times\{0\}$, this definition agrees with  Definition~\ref{def:degree_first_def}.
In particular the index of a singularity in this case is of the form $\frac{ k}{m}$ for some integer $k \in \mathbb{Z}$.  Furthermore, as pointed out in Theorem 1.3, page 108 of \cite{Hopf}, the degree does not depend on the curve $\gamma$ nor on the coordinate patch used to define it.  With this terminology we can now state the following theorem.  Its proof can be found in page 112 of \cite{Hopf}, for the case $m=2$, or in the appendix of \cite{ray:NSDF:2006}, for any $m \geq 2$.  Hence, we omit it.

\begin{theorem}[\cite{Hopf,ray:NSDF:2006}] \label{thm:poincarehopf}
Let $\Sigma \subset \R^3$ be a closed, smooth, orientable surface, $A\subset \Sigma$ a finite set, and $Q:\Sigma \setminus A \to M$ an $m$-gon-valued field in $\Sigma$.  We assume further that every $x\in A$ is a singularity of $Q$.  Then
$$
\int_\Sigma K\,dS = 2\pi \sum_{x\in A} i(x).
$$
Here $K$ is the Gauss curvature of $\Sigma$, and $i(x)$ denotes the index of the singularity $x\in A$.
\end{theorem}
\begin{remark}
\label{rmk:PH}
By the Gauss-Bonnet theorem, 
$$
\int_\Sigma K\,dS = 2\pi \xi(\Sigma) = 2\pi (2-2g(\Sigma)),
$$
where $g(\Sigma)$ and $\xi(\Sigma)$ are the genus and Euler characteristic of $\Sigma$, respectively.  Hence, we conclude that
\begin{equation}\label{e:indexformulaPH}
\sum_{x\in A} i(x) = 2-2g.
\end{equation}
\end{remark}



\section{Examples and numerical experiments} 
\label{sec:numerics}

In this section we discuss nontrivial examples of tetrahedron-valued maps, either constructed analytically or obtained via numerical simulations. In the latter case, the goal is to understand behavior of local minimizers of \eqref{e:GL2D} and  \eqref{eq:graf}, respectively by simulating gradient flow for each energy using the finite element software COMSOL \cite{comsol}. Note that the two analytical examples below only provide competitors and not minimizers of the corresponding variational problems. 

\subsection{Example of a map from $B_1 \subset \R^2$ into $SO(3)/T$}
Here we show that, similar to what is known for $Q$-tensors, normal boundary alignment of tetrahedron-valued maps in two-dimensional domains does not require singularities. Indeed, we can construct a map that is nonsingular because it "escapes into the third dimension" in the interior of the domain. 

On the unit disk with polar coordinates define
\begin{align*}
\mathbf{f}^1(r, \theta) & = \cos(\theta) \left ( \cos\left ( \frac{r\pi}{2}\right ) \mathbf{e}^r - \sin\left ( \frac{r\pi}{2}\right ) \mathbf{e}^3 \right ) - \sin(\theta) \mathbf{e}^\theta, \\
\mathbf{f}^2(r, \theta) & = \sin(\theta) \left ( \cos\left ( \frac{r\pi}{2}\right ) \mathbf{e}^r - \sin\left ( \frac{r\pi}{2}\right ) \mathbf{e}^3 \right ) + \cos(\theta) \mathbf{e}^\theta, \\
\mathbf{f}^3(r, \theta) & =  \sin\left ( \frac{r\pi}{2}\right ) \mathbf{e}^r + \cos\left ( \frac{r\pi}{2}\right ) \mathbf{e}^3 .
\end{align*}
Now set
$$
\mathbf{b}^2 = \mathbf{f}^2, \,\,\, \mathbf{b}^3 = -\frac{1}{2} \mathbf{f}^2 + \frac{\sqrt{3}}{2} \mathbf{f}^3, \,\,\,  \mathbf{b}^4 = -\frac{1}{2} \mathbf{f}^2 - \frac{\sqrt{3}}{2} \mathbf{f}^3,
$$
and define
$$
\mathbf{a}^1 = \mathbf{f}^1\quad\mbox{and}\quad\mathbf{a}^j = -\frac{1}{3} \mathbf{f}^1 + \frac{2\sqrt{2}}{3} \mathbf{b}^j\,\,\,\mbox{for}\,\,\, j=2, 3, 4.
$$
This gives
\begin{equation}
\label{eq:trivial}
\begin{aligned}
\mathbf{a}^1 & = \mathbf{f}^1,\\
\mathbf{a}^2 & = -\frac{1}{3}\mathbf{f}^1 + \frac{2\sqrt{2}}{3} \mathbf{f}^2, \\
\mathbf{a}^3 & = -\frac{1}{3} \mathbf{f}^1 - \frac{\sqrt{2}}{3} \mathbf{f}^2 + \frac{\sqrt{6}}{3} \mathbf{f}^3 \\
\mathbf{a}^4 & = -\frac{1}{3} \mathbf{f}^1 - \frac{\sqrt{2}}{3} \mathbf{f}^2 - \frac{\sqrt{6}}{3} \mathbf{f}^3.
\end{aligned}
\end{equation}
It is easy to check that the map that sends a point in the unit disk to $\{ \mathbf{a}^1, \mathbf{a}^2, \mathbf{a}^3, \mathbf{a}^4\}$ is tetrahedron-valued, nonsingular and the vector $\mathbf{a}^1$ coincides with the normal on the boundary of the disk.

Note that a quaternion representation of the rotation inherent in this map (starting from the tetrahedron containing $\mathbf{e}^3$)
is given by 
\[
\mathbf{q}(r,\theta)=\cos \left( \frac{\pi r}{4} \right) - \sin (\theta )
\sin \left( \frac{\pi r}{4 } \right) \mathbf{i} + \cos (\theta )
\sin \left( \frac{\pi r}{4 } \right) \mathbf{j}. 
\]
As this is a smooth map into $S^3$, it belongs
to the trivial homotopy class.

\subsection{Example of a map from $B_1(0)\subset \R^3$ into $SO(3)/T$}

The next example is that of a tetrahedron-valued map on the unit ball in $\R^3$ that has exactly one point singularity on the boundary of the ball (at the north pole) and no other interior point or line singularities. This map also has a finite Dirichlet integral.  The example is a straightforward adaptation of a similar computation for orthonormal frame-valued map in \cite{GMS}.

Let $\mathbf{x}= (x_1,x_2,x_3)^T$, then   define the vector fields
\begin{align*}
\mathbf{f}^1(\mathbf{x}) & = \left (  \frac{2x_1(1-x_3)}{x_1^2+x_2^2+(1-x_3)^2}, \frac{2x_2(1-x_3)}{x_1^2+x_2^2+(1-x_3)^2}, \frac{x_1^2+x_2^2-(1-x_3)^2}{x_1^2+x_2^2+(1-x_3)^2} \right )^T,\\
\mathbf{f}^2(\mathbf{x}) & = \left (  \frac{-x_1^2+x_2^2+(1-x_3)^2}{x_1^2+x_2^2+(1-x_3)^2}, \frac{-2x_1x_2}{x_1^2+x_2^2+(1-x_3)^2}, \frac{2x_1(1-x_3)}{x_1^2+x_2^2+(1-x_3)^2} \right )^T, \\
\mathbf{f}^3(\mathbf{x}) & = \left (  \frac{2x_1x_2}{x_1^2+x_2^2+(1-x_3)^2}, \frac{-(x_1^2-x_2^2+(1-x_3)^2)}{x_1^2+x_2^2+(1-x_3)^2}, \frac{-2x_2(1-x_3)}{x_1^2+x_2^2+(1-x_3)^2} \right )^T,
\end{align*}
for $\mathbf{x}\in B_1(0) \subset \R^3$, $\mathbf{x}\neq (0, 0, 1)^T$.  Note that $\mathbf{f}^i \cdot \mathbf{f}^j = \delta_{ij}$, $i,j=1,\ldots,3$ and $\mathbf{f}^1(\mathbf{x}) = \mathbf{x}$ whenever $\abs{\mathbf{x}}^2 = 1$.  It follows that $\{\mathbf{f}^2(\mathbf{x}), \mathbf{f}^3(\mathbf{x})\}$ is an orthonormal frame for the tangent plane at the boundary whenever $\mathbf{x}\in \partial B_1(0) \setminus \{(0, 0, 1)\}$.

Next set
$$
\mathbf{b}^2 =  \mathbf{f}^2, \,\,\, \mathbf{b}^3 = -\frac{1}{2} \mathbf{f}^2 + \frac{\sqrt{3}}{2} \mathbf{f}^3, \,\,\,  \mathbf{b}^4 = -\frac{1}{2} \mathbf{f}^2 - \frac{\sqrt{3}}{2} \mathbf{f}^3,
$$
then $\left\langle \mathbf{f}^1 ,  \mathbf{b}^j \right\rangle=0$ for $j=2, 3, 4$.  Further,
$
\left\langle \mathbf{b}^i ,  \mathbf{b}^j \right\rangle = \frac{3}{2}\delta_{ij}-\frac{1}{2} \,\,\, \mbox{for}\,\,\,  i, j = 2, 3, 4,
$
that is the map that sends $\mathbf{x} \in B_1(0)$ to $\{\mathbf{b}^2, \mathbf{b}^3, \mathbf{b}^4\}$ is MB-valued in $B_1(0)\setminus \{(0, 0, 1)\}$.  By construction, the two-dimensional MB frame is contained in the tangent plane at any $\mathbf{x}\in \partial B_1(0) \setminus \{(0, 0, 1)\}$.

Similar to the previous example, now let
$$
\mathbf{a}^1 = \mathbf{f}^1\quad\mbox{and}\quad \mathbf{a}^j = -\frac{1}{3} \mathbf{f}^1 + \frac{2\sqrt{2}}{3} \mathbf{b}^j\,\,\,\mbox{for}\,\,\, j=2, 3, 4,
$$
so that
\begin{align*}
\mathbf{a}^1 & = \mathbf{f}^1, \\
\mathbf{a}^2 & = -\frac{1}{3} \mathbf{f}^1 + \frac{2\sqrt{2}}{3} \mathbf{f}^2, \\
\mathbf{a}^3 & = -\frac{1}{3} \mathbf{f}^1 - \frac{\sqrt{2}}{3} \mathbf{f}^2 + \frac{\sqrt{6}}{3} \mathbf{f}^3, \\
\mathbf{a}^4 & = -\frac{1}{3} \mathbf{f}^1 - \frac{\sqrt{2}}{3} \mathbf{f}^2 - \frac{\sqrt{6}}{3} \mathbf{f}^3.
\end{align*}
A straightforward computation shows that, for $i, j\in \{1, 2, 3, 4\}$, we have
$$
\mathbf{a}^i \cdot \mathbf{a}^j = \frac{4}{3}\delta_{ij}- \frac{1}{3},
$$
Thus the vectors $\mathbf{a}^1, \mathbf{a}^2, \mathbf{a}^3, \mathbf{a}^4$ give the 4 vertices of a tetrahedron and we can consider the map that sends $\mathbf{x}\in B_1(0)$ to the tetrahedron defined by $\{\mathbf{a}^1, \mathbf{a}^2, \mathbf{a}^3, \mathbf{a}^4\}$.  Since $\mathbf{a}^1 = \mathbf{f}^1$ and $\mathbf{f}^1(\mathbf{x}) = \mathbf{x}$ for $\mathbf{x}\in \partial B_1(0) \setminus \{(0, 0, 1)$, one of the vectors of this tetrahedron coincides with the normal on the boundary of the unit ball.  

\subsection{An energy-minimizing map from $B_1\subset\mathbb{R}^2$ into $SO(2)/D_3$}

Here we examine an MB-frames-valued map generated by a gradient descent of the energy \eqref{e:GL2D} subject to the Dirichlet condition that the normal on the boundary is aligned with the MB frame. When simulating in a domain that has a shape of an equilateral triangle,  the gradient flow converges to a constant state with the three vectors of the MB frame perpendicular to the respective sides of the triangle. Clearly this state is also the global minimizer of \eqref{e:GL2D}. 

In order to observe a state with vortices, we excised a disk from the equilateral triangle. The restriction of the solution to the boundary of the triangle is still a constant MB frame, while the requirement that the normal to the boundary of the disk is aligned with the MB frame produces winding of the normal by the angle $2\pi$ when the boundary of the disk is traversed once in the positive direction. Because the quantum of winding for the MB frame is $2\pi/3$ (as the degree takes values in $\frac13\mathbb{Z}$), we thus expect three vortices with opposite of that winding to form in the interior of the domain. Indeed, in Figure~\ref{fig:triang2d} one sees that an MB frame aligns with the exterior boundary away from the excised disk, while the disk induces three vortices with winding of $-{2\pi \over 3}$ (Fig.~\ref{fig:test2}).

\begin{figure}[H]
        \centering
           \subfloat[]{%
              \includegraphics[height=3in]{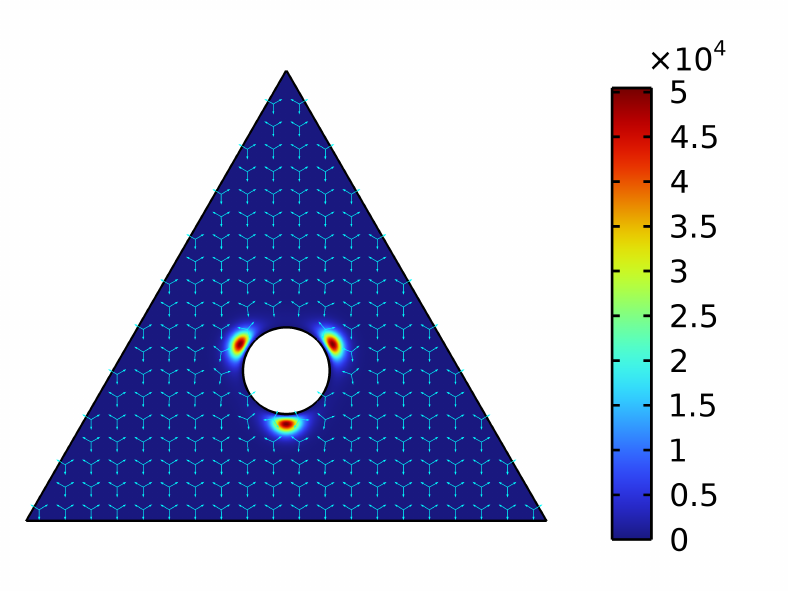}%
              \label{fig:triang2d}%
           }\qquad\raisebox{0.5in}
           {\subfloat[]{%
              \includegraphics[height=1.5in]{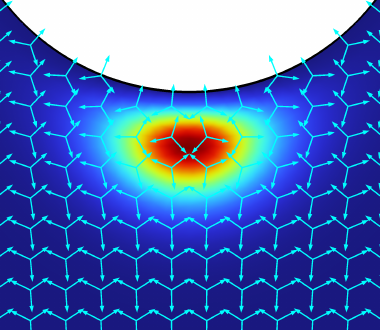}%
              \label{fig:test2}%
           }}
           \caption{A critical point of \eqref{eq:GL2DW} in an equilateral triangular domain with an excised disk. The vectors of the MB frame field are shown with the colorbar representing the values of the potential $\left|\mathcal{A} \mathcal{A}^T - {9\over 8} I(2)\right|^2$. Here $\varepsilon=0.04$. (a) A minimal MB frame shows three vortices with winding $-{2\pi \over 3}$; (b) Magnification about one of the singularities identifies the local rotation of the MB frame.}
           \label{fig:MB}
\end{figure}

\subsection{Computational examples of maps from $B_1\subset\mathbb{R}^2$ into $SO(3)/T$}

We now explore critical points of \eqref{eq:grafstrong} for two different choices of Dirichlet boundary conditions. 

{\bf (a.)} First, consider a tetrahedral frame field $\mathcal{Q}_a$ defined by the vectors given in \eqref{eq:trivial} and consider critical points of the energy $\mathcal{E}^{3d}_\e$ from \eqref{eq:grafstrong}, that satisfy the same Dirichlet data as $\mathcal{Q}_a$. We first run the gradient flow simulation for this setup assuming the the initial condition is also given by $\mathcal{Q}_a$. Because $\mathcal{Q}_a$ is smooth in $B_1$ and the energy of the gradient flow solution is bounded by its initial value, the critical points obtained starting from $\mathcal{Q}_a$ have energies that are uniformly bounded in $\varepsilon$ so that the critical point is nonsingular (Fig.~\ref{fig:trivial0}).

\begin{figure}[H]
        \centering
           \subfloat[]{%
           \includegraphics[height=2in]{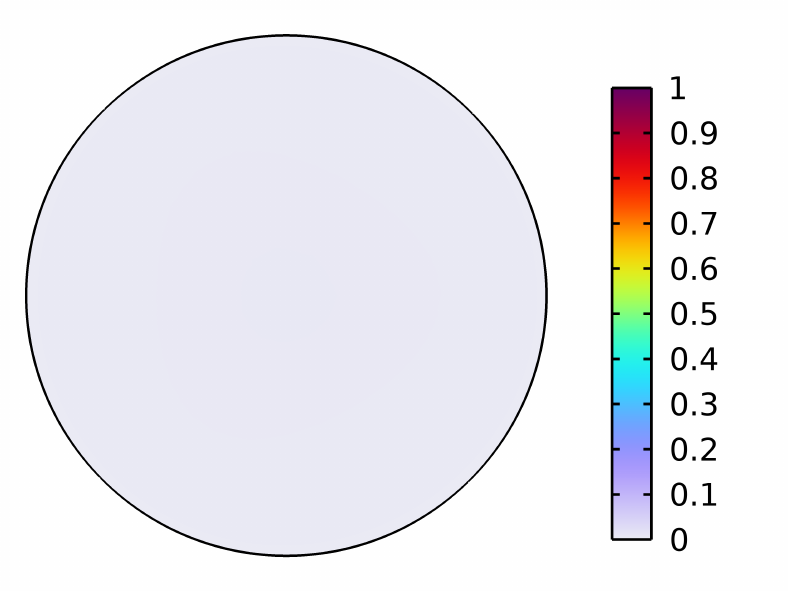}%
              \label{fig:trivial0-1}%
           }\qquad
           \subfloat[]{%
           \includegraphics[height=2in]{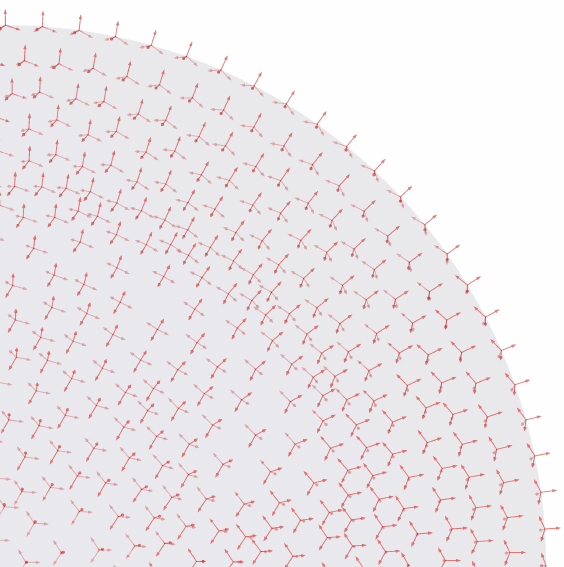}%
              \label{fig:trivial0-2}%
           }
               \caption{Gradient flow solution subject to the Dirichlet boundary data given by \eqref{eq:trivial} starting from the initial condition also given by \eqref{eq:trivial}. (a) Plot of $W(\mathcal{Q})$ shows no singularities in the interior of the domain; (b) Plot of the energy-minimizing tetrahedral frame field in the top right quadrant. Here $\varepsilon=0.05$.}
           \label{fig:trivial0}
\end{figure}
On the other hand, the gradient flow simulation that starts from the trivial initial condition $\mathcal{Q}\equiv0$ results in an entirely different stable critical point shown in Fig.~\ref{fig:trivial3}. This local minimizer has three equidistant point singularities such that the frame field over any curve surrounding one singularity is in a homotopy class conjugate to either $\mathbf{s}$ or $\mathbf{s}^{-1}$, i.e., one vector of the frame does not change while a curve is traversed, Fig.~\ref{fig:trivial3}(b)-(d).

The different critical points, given the same boundary data, is an example of distinct topological resolutions for that data, see Remark~\ref{rm:topres}. 
\begin{figure}[H]
        \centering
           \subfloat[]{%
           \includegraphics[height=2in]{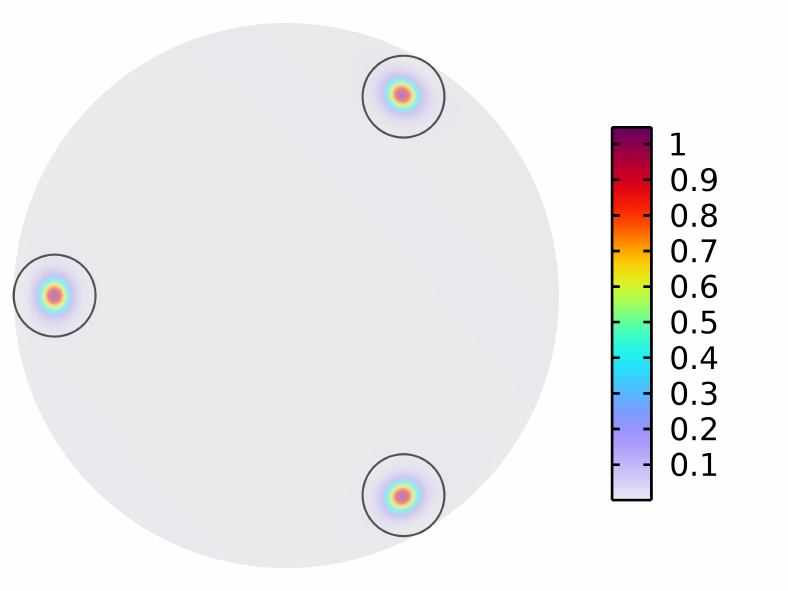}%
              \label{fig:trivial3-1}%
           }\qquad
           \subfloat[]{%
           \includegraphics[height=1.5in]{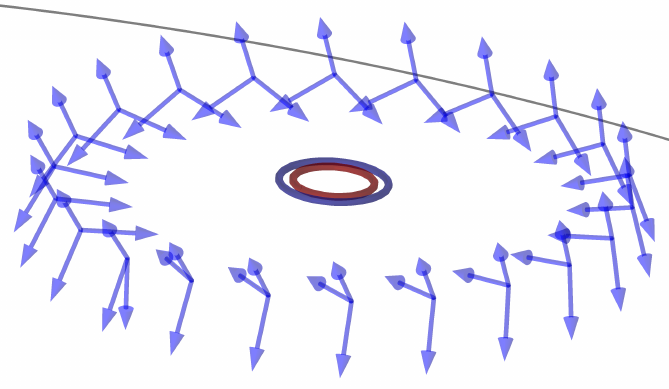}%
              \label{fig:trivial3-2}%
           } \\ 
           \subfloat[]{%
           \includegraphics[height=1.5in]{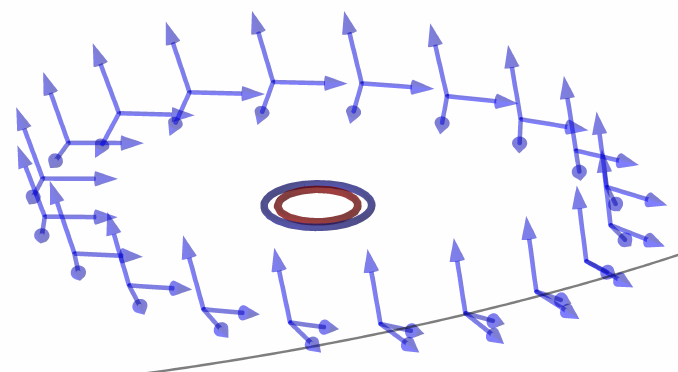}%
              \label{fig:trivial3-3}%
           }\qquad
           \subfloat[]{%
           \includegraphics[height=1.5in]{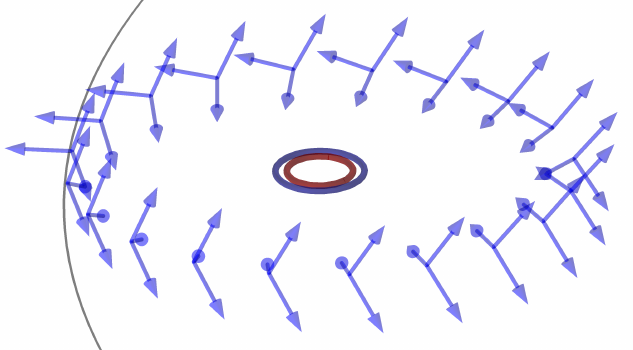}%
              \label{fig:trivial3-4}%
           }
               \caption{Gradient flow solution subject to the Dirichlet boundary data given by \eqref{eq:trivial} starting from the trivial initial condition $\mathcal{Q}\equiv0$. (a) Plot of $W(\mathcal{Q})$ shows three singularities in the interior of the domain; (b-d) Partial plots of the energy-minimizing tetrahedral frame field over the curves shown in black in (a). Here the frame field over all curves has a homotopy class conjugate to $\mathbf{s}$ or $\mathbf{s}^{-1}$ (one vector does not change while a curve is traversed). Here $\varepsilon=0.05$.}
           \label{fig:trivial3}
\end{figure}

{\bf (b.)} The next example shows a critical point of \eqref{eq:grafstrong} obtained via a gradient flow in a unit disk for maps satisfying Dirichlet boundary data that lies in a homotopy class conjugate to ${\bf i}$. Here both the boundary and the initial condition were obtained using the techniques described in Section \ref{sec:gendata}. The initial condition had a singularity at the center of the disk and the tetrahedral frame  map was in the same homotopy class as the boundary data for every circle surrounding the singularity. The simulation attains a local minimizer of \eqref{eq:grafstrong} shown in Fig.~\ref{fig:pivs2pi3}. This minimizer has two singularities in the interior of the disk and examination of winding of the frame vectors indicates that the frame field over the curves surrounding each singularity has a homotopy class conjugate to $\mathbf{s}$. Indeed, the state with two singularities corresponds to a shorter topological resolution than that for the boundary data and thus this state has a lower energy than a state with one singularity.
\begin{figure}[H]
        \centering
           \subfloat[]{%
           \includegraphics[height=2in]{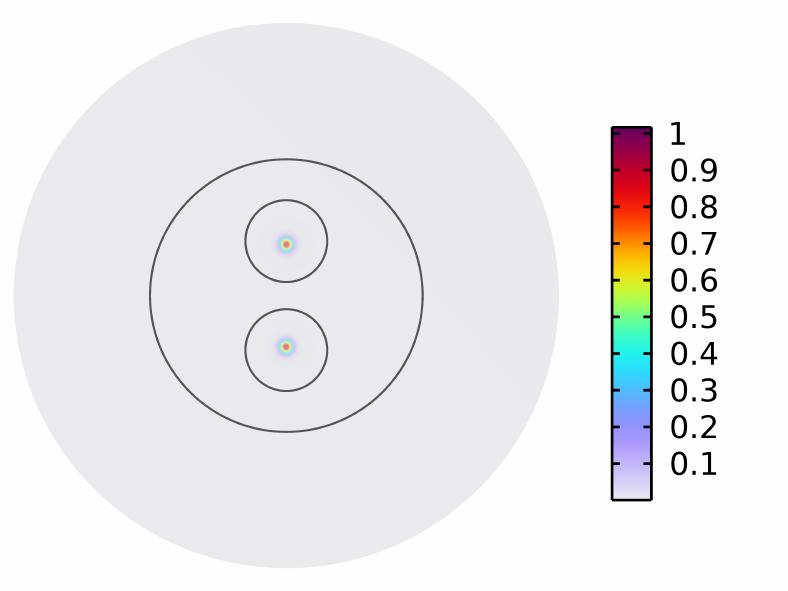}%
              \label{fig:pivs2pi3-1}%
           }\qquad
           \subfloat[]{%
           \includegraphics[height=1.5in]{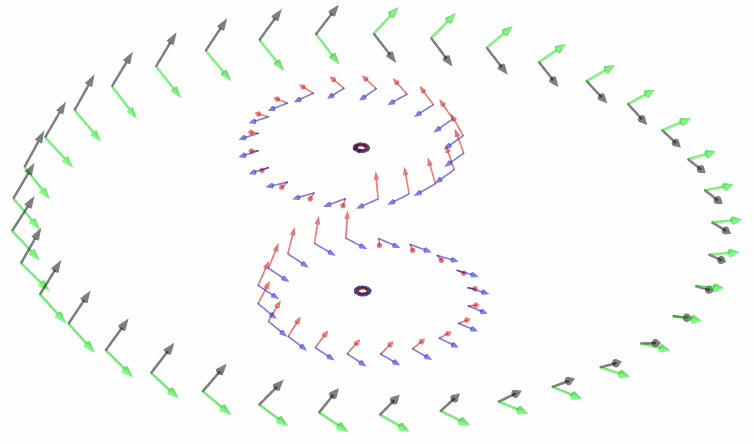}%
              \label{fig:pivs2pi3-2}%
           }
           \caption{Gradient flow solution subject to the Dirichlet boundary data homotopic to $\mathbf{i}$. (a) Plot of $W(\mathcal{Q})$ shows two singularities in the interior of the domain; (b) Partial plots of the energy-minimizing tetrahedral frame field over the curves shown in black in (a). Here the frame field over the larger curve has a homotopy class conjugate to $\mathbf{i}$ (both green and black vectors rotate by $\pi$ when the curved is traversed once) while the frame field over the smaller curves has a homotopy class conjugate to $\mathbf{s}$  (the blue vector does not change its orientation while red vector rotates by $2\pi/3$ when the respective curves are traversed once). Here $\varepsilon=0.05$.}
           \label{fig:pivs2pi3}
\end{figure}

The two remaining examples deal with the tetrahedron-valued maps obtained via a gradient flow for the energy \eqref{eq:graf} when $\delta_1=\delta_2=\delta\ll\varepsilon.$

\subsection{Computational examples of maps from $B_1(0)\subset\mathbb{R}^3$ into $SO(3)/T$}

{\bf (a.)} In Fig.~\ref{fig:sphere} we depict the singular set of a critical point of \eqref{eq:graf} in the ball of radius $1$, where the surface contribution to the energy forces the tetrahedron-valued map to contain the normal to the sphere $\partial B_1$ almost everywhere on $\partial B_1$. Then the trace of this map on $\partial B_1$ can be identified with an MB-valued map. From Remark \ref{rmk:PH}, it follows that any MB-valued map must have singularities on $\partial B_1$ with the degrees of these singularities adding up to $2$. Notice that the energy of an MB-frame field in \eqref{e:GL2D} is identical to the scaled Ginzburg-Landau energy after a short calculation. Recall that in Ginzburg-Landau theory, vortices of higher degree in equilibrium split into vortices of degree $1$ that repel each other, {\cite{BBH}}, and we expect analogous behavior from vortices of MB-valued maps. Because the degree of the MB-frame is measured in units of $1/3$, the singular set of total degree $2$ on the sphere should split into six equidistant vortices of degree $1/3$. 

Since we assumed that $\delta\ll\varepsilon$, the penalty associated with the surface energy is much stronger than that for the bulk energy and therefore the surface effects should dominate bulk effects. Indeed, the singular set in Fig.~\ref{fig:sphere}(a) intersects the surface of the sphere at six, approximately equidistant points, connected by line singularities in the interior of the domain. The bulk structure of the singular set is dictated by the topology of $SO(3)/T$ and the energy considerations. In particular, from Fig.~\ref{fig:sphere}(a) it seems to be clear that the total length of the singular set can be reduced by ``squeezing" it toward the center of the ball. This, in fact, is possibly what would happen as $\varepsilon,\delta\to0$ but such investigation is beyond the scope of the present paper. Rather, we are interested to understand the behavior of the singular set for {\it finite} $\varepsilon$ and $\delta$ by exploring the topological structure of a triple junction.
\begin{figure}[H]
        \centering
           \subfloat[Singular set]{%
              \includegraphics[height=2.5in]{sphere.pdf}%
              \label{fig:sphere-1}%
           }\qquad
           \subfloat[$z=0.25$]{%
              \includegraphics[height=2.5in]{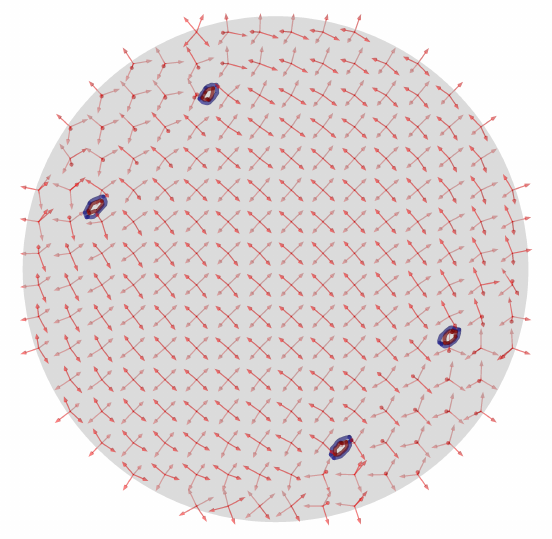}%
              \label{fig:sphere-2}%
           }\\
           \subfloat[$z=0.5$]{%
              \hspace{-0.5in}\includegraphics[height=2.5in]{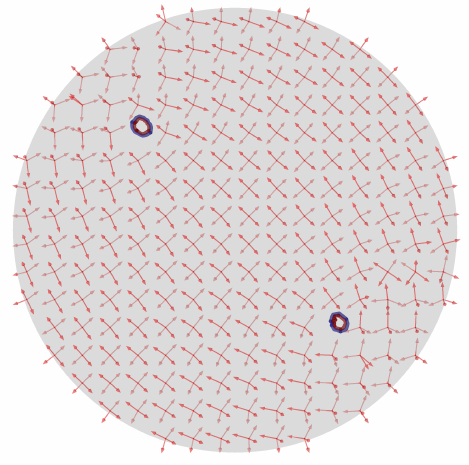}%
              \label{fig:sphere-3}%
           }
           \qquad\qquad\qquad\qquad\raisebox{0.5in}{\subfloat[$z=0.95$]{%
              \includegraphics[height=1.3in]{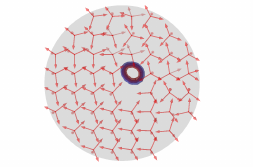}%
              \label{fig:sphere-4}%
           }}
           \caption{A locally minimizing tetrahedral frame field in a sphere of the radius $1$. (a) The singular set; (b-d) Cross-sections of the frame field and the line singularities by planes corresponding to different values of $z$. Here $\varepsilon=0.05$ and $\delta=0.01$.}
           \label{fig:sphere}
\end{figure}
In Fig.~\ref{fig:sphere-triple}(a), a small sphere surrounds the triple junction that we are interested in and in Fig.~\ref{fig:sphere-triple}(c) we plotted the vectors of the trace of the tetrahedron-valued frame field on the surface of that sphere. The circles on the same plot depict the intersection between the singular set and the surface of the sphere. 
\begin{figure}[H]
        \centering
           \subfloat[]{%
              \includegraphics[height=2.5in]{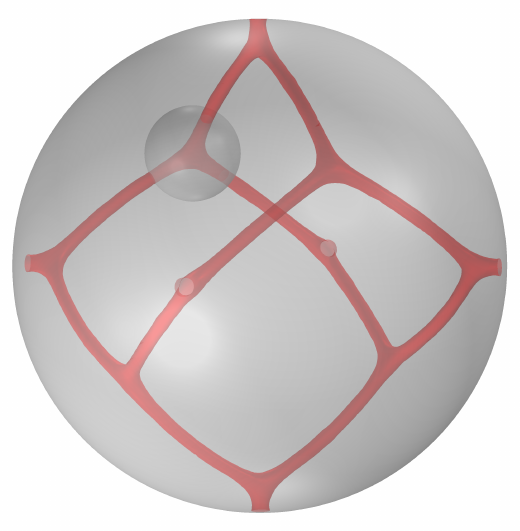}%
              \label{fig:triple-1}%
           }\qquad
           \subfloat[]{%
              \includegraphics[height=2.5in]{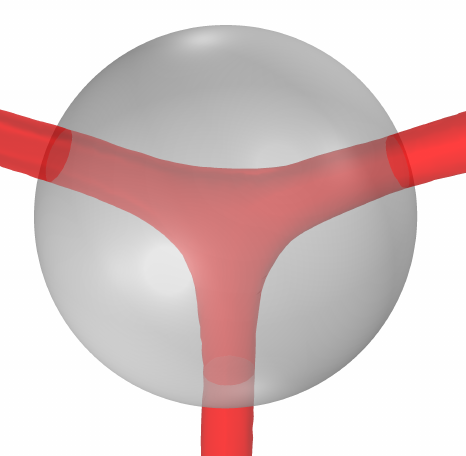}%
              \label{fig:triple-2}%
           }\\
           \subfloat[]{%
              \includegraphics[height=2.5in]{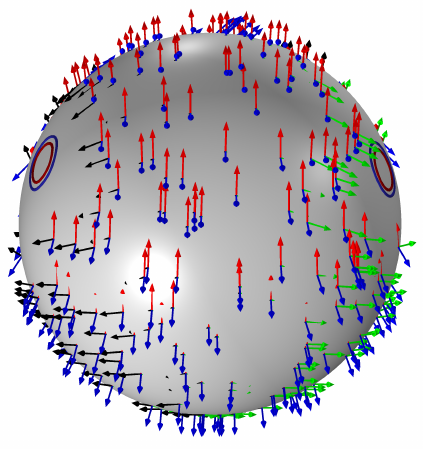}%
              \label{fig:triple-3}%
            }
           \caption{The structure of a triple junction. (a) The singular set with a small sphere surrounding a triple junction; (b) Zoom-in on the sphere with a triple junction; (c) Cross-section of the tetrahedral frame field by the same sphere. Circles correspond to cross-sections of the line singularities by the sphere. For the frame field near the singularity on the right, the black vector appears to be missing, while the green vector appears to be missing for the frame field near the singularity on the left. $\varepsilon=0.05$ and $\delta=0.01$.}
           \label{fig:sphere-triple}
\end{figure}

Fig.~\ref{fig:triple-13} shows the same small spherical cutout from the points of view of individual arms of the triple junction. Near each singularity on the sphere associated with an arm of the junction, a frame vector of one color points into the sphere toward the triple junction and is not visible, while the remaining three vectors form an MB frame field. This MB frame field rotates by $120^\circ$ counterclockwise when one travels around the singularity in the counterclockwise direction. The product of two $120^\circ$ counterclockwise rotations around any two adjacent edges of a tetrahedron indeed corresponds to an inverse of a counterclockwise rotation around a third edge and we expect that a similar structure is valid for the remaining triple junctions of the singular set in Fig.~\ref{fig:sphere}(a).
\begin{figure}[H]
        \centering
           \subfloat[]{%
              \includegraphics[height=2.3in]{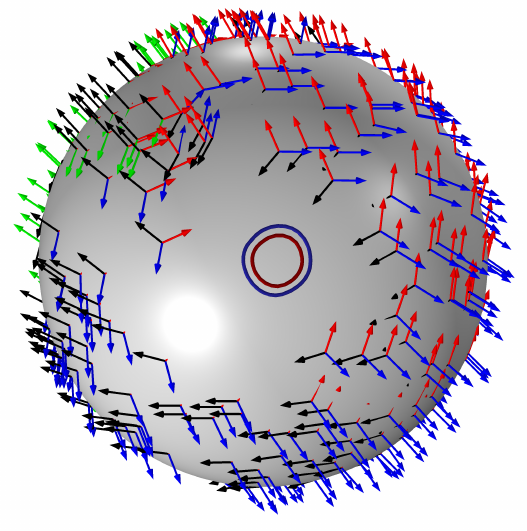}%
              \label{fig:triple-11}%
           }\qquad
           \subfloat[]{%
              \includegraphics[height=2.3in]{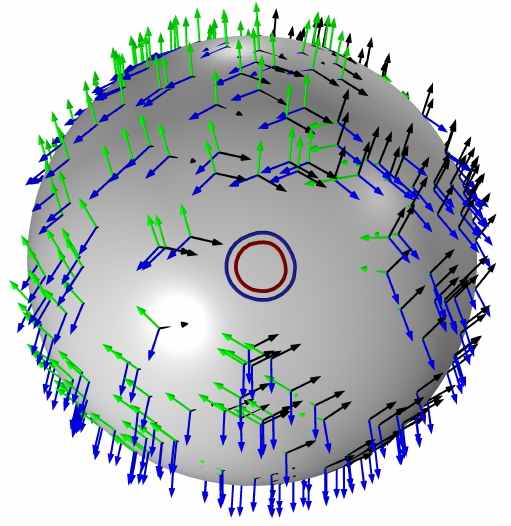}%
              \label{fig:triple-12}%
           }\\
           \subfloat[]{%
              \includegraphics[height=2.3in]{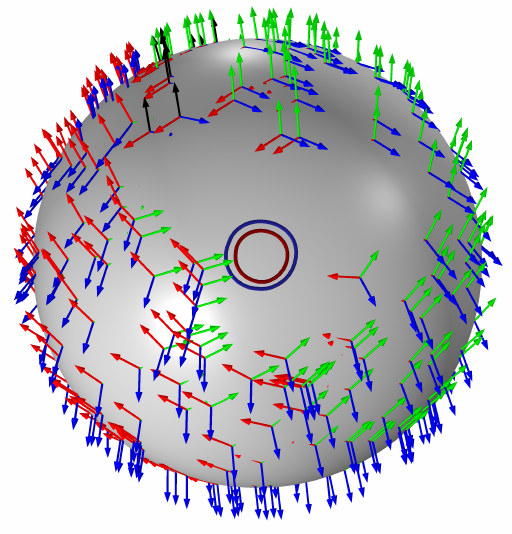}%
              \label{fig:triple-13}%
            }
           \caption{Cross-section of the tetrahedral frame field by the sphere in Fig.~\ref{fig:sphere-triple} viewed from the perspectives of three different arms of the triple junction. (a) Near the singularity, the green vector points into the sphere toward the triple junction and the MB frame consisting of the black, red and blue vectors rotates by $120^\circ$ counterclockwise when one travels around the singularity in the counterclockwise direction; (b) Near the singularity, the red vector points into the sphere toward the triple junction and the MB frame consisting of the blue, green and black vectors rotates by $120^\circ$ counterclockwise when one travels around the singularity in the counterclockwise direction; (c) Near the singularity, the black vector points into the sphere toward the triple junction and the MB frame consisting of the blue, green and red vectors rotates by $120^\circ$ counterclockwise when one travels around the singularity in the counterclockwise direction. Here $\varepsilon=0.05$ and $\delta=0.01$.}
           \label{fig:sphere-triple-1}
\end{figure}

{\bf (b.)} Finally, in Fig.~\ref{fig:sphere-out} we show the singular set of a critical point of \eqref{eq:graf} in a large domain in the {\it exterior} of the ball of radius $1$, where the surface contribution to the energy forces the tetrahedron-valued map to contain the normal to the sphere $\partial B_1$ almost everywhere on $\partial B_1$. The structure of this set is clearly similar to that in Fig.~\ref{fig:sphere} in that it has six equidistant vortices on the surface of the sphere and the same number of triple junctions but, this time, the total length of the line singularities composing the singular set can be reduced by compressing these singularities towards the sphere. This is indeed what is observed in Fig.~\ref{fig:sphere-out}.
\begin{figure}
    \centering
    \includegraphics[width=2in]{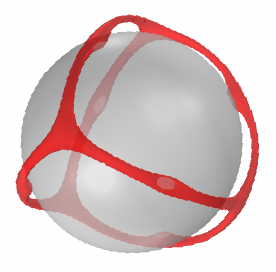}
    \caption{The line singularity of a locally minimizing tetrahedral frame field outside of a sphere of the radius $0.1$ and inside a large tetrahedron (not shown).  Here $\varepsilon=0.01$ and $\delta=0.01$.}
    \label{fig:sphere-out}
\end{figure}

\section{Discussion}
\label{sec:dis}

Tetrahedral symmetry arises in many contexts in nature, including in liquid crystals, condensed matter physics and materials science. The  framework presented here for generating such frame fields  entails identifying a bijection between a tetrahedral configuration and a specific variety on traceless, symmetric three-tensors.  When tetrahedral frame fields are constrained to align with the normal, singularities arise generically. To avoid this, the variety suggests a natural harmonic map (Ginzburg-Landau) relaxation  that helps in computation to identify a local minimizing configuration with isolated near-singular sets. We discuss some open questions and interesting lines of inquiry related to the work presented here.

In Section~\ref{s:nD} we showed that $n+1$ vectors satisfying \emph{$n+1$-hedral symmetry} on the  unit sphere in $\mathbb{R}^n$ defines a symmetric three tensor $\mathcal{Q}$ that satisfies $\mathcal{Q}\mathcal{Q}^T = {(n+1)(n^2 -1) \over n^3} I(n)$.  On the other hand, the recovery has only been shown for $n\in \{2,3\}$. In particular it would be natural to ask if for 
all $n \geq 4$ there is a bijection between $SO(n)/G_n$ and $\Htrace(n,3) \cap \{ \mathcal{Q}\mathcal{Q}^T = {(n+1)(n^2 -1) \over n^3} I(n)\}$, 
where $G_n$ represents the group of rotations leaving $n+1$ equipositioned points on the sphere $\mathbb{S}^{n-1}$ invariant.

There are many other interesting directions to explore. 
Recently, it was proved that the energy density of mappings from planar domains into relaxed vacuum manifold targets converge to a sum of delta functions with masses associated to the energy minimized over all nontrivial homotopy classes (so-called systolic geodesics), see \cite{ monteil2021ginzburg, MRV2021}.  
One interesting direction would be to perform a stability analysis of equivariant critical points that conform to different homotopy classes either on the boundary of a disk or at infinity, along the lines of \cite{ignat2020symmetry} for $Q$-tensor models.  This should provide additional information on the selection of particular topological resolutions found in our numerical experiments.   
Even more more important would be to  
rigorously establish the expected Gamma-limit of the relaxation energy, \eqref{eq:grafstrong}.  In particular, one should be able to identify a compact mapping, along the lines of the Jacobian in Ginzburg-Landau theory, that concentrates at singular points for two dimensional domains and co-dimension two, rectifiable sets for three dimensional domains. 
Results for the cross-field case are the subject of the work in 
 \cite{GMS2}, and we expect the Gamma-limit in the tetrahedral frame problem to behave similarly.

A more intriguing task would be to provide a rigorous rationale for the formation of triple junctions in the singular limit of tetrahedral frame fields for mappings  on three dimensional domains (versus quadruple junctions in the $SO(3)/O$ case), see Fig~\ref{fig:junctions}. Systolic geodesics should play a role here.

Even more generally, one can ask how this approach can be used to connect a quotient of a Lie group with a specific choice of symmetric $k$ tensor and a variety.  For example can a frame with icosahedral symmetry, $SO(3)/I$,  be generated by a suitable variety on the correct choice of $k$-symmetric tensors?  
How these $k$'s and choice of varieties arise remains mysterious.

\section{Appendix A: Proof of Proposition~\ref{p:etevpairs2}}

In this section we prove Proposition~\ref{p:etevpairs2}.  This requires the use of properties of families of vectors in $\R^n$ that we recorded in Lemma~\ref{e:propsvecsangles}, particularly the property stated in  \eqref{e:sumprojections}.  
For the convenience of the reader we first state the Proposition, and then proceed to its proof.

\begin{proposition}
For every $n \geq 2$ there exists an $n\times (n+1)$ matrix, $C^n$ such that the following holds. For any set of $n+1$ unit vectors in $\R^n$,  $\{\mathbf{u}^j\}_{j=1}^{n+1}$, that satisfy the inner product condition \eqref{e:innerprodcondition}, and its associated 3-tensor $\mathcal{Q} \in \Htrace(n, 3)$ generated by these $\mathbf{u}^j$'s via \eqref{e:def3tensor}.  Let $\mathcal{Q} = (\mathcal{Q}_1 | ... | \mathcal{Q}_n)$ where the $\mathcal{Q}^j \in \Htrace(n, 2)$ are defined by
$$
\mathcal{Q}_j = \sum_{k=1}^n \langle \mathbf{u}^k, \mathbf{e}^j\rangle \mathbf{u}^k(\mathbf{u}^k)^T.
$$
If we denote $A_n$ the $n\times (n+1)$ matrix whose columns are the vectors $\{\mathbf{u}^j\}_{j=1}^{n+1}$.  Then, the matrix
$$
R = \frac{n}{n+1}A_nC_n^T \in O(n)
$$
is orthogonal, and $A_n = RC_n$.  Letting $\mathbf{f}^j = R\mathbf{e}^k$, $k=1, ..., n$, and
$$
B^k = \frac{1}{\lambda_n}\sum_{j=1}^n\langle \mathbf{f}^k, \mathbf{e}^j\rangle \mathcal{Q}_j,
$$
where
$$
\lambda_n = \left (\frac{(n^2-1)(n+1)}{n^3}\right )^{\frac{1}{2}},
$$
then, the vectors $\mathbf{f}^k$ along with the tensors $B^k$, $k=1, ..., n$ are eigenvector-eigentensor pairs for $\mathcal{Q}$ in the sense of Remark~\ref{r:xiconditions}.  In particular, 
$$
\langle \mathbf{f}^j, \mathbf{f}^k \rangle = \langle B^j, B^k \rangle = \delta_{jk},
$$
$$
\sum_{i=1}^n \langle B_k, \mathcal{Q}_i\rangle \mathbf{e}^i = \lambda_n \mathbf{f}^k, \,\,\,\mbox{and}\,\,\, \sum_{j=1}^n\langle \mathbf{f}^k, \mathbf{e}^j\rangle \mathcal{Q}_j = \lambda_n B^k.
$$
Finally, one can recover the $3$-tensor $\mathcal{Q}$ from the eigenvector-eigentensor pairs, in the sense that
$$
\mathcal{Q}_i = \sum_{j=1}^n \langle \mathbf{f}^j, \mathbf{e}^i\rangle B_j,
$$
for $i=1, ..., n$.

\end{proposition}

\begin{proof}

For the proof we first construct the matrix $C^n$, and proving several of its properties.  The conclusions of the Proposition will then follow from here.

\noindent {\bf Step 1.}  In this step we will build an $n\times (n+1)$ matrix $C_n$, $n\geq 2$, such that $C_nC_n^T = \frac{n+1}{n}I(n)$, that
$$
C_n^TC_n = \frac{n+1}{n}I({n+1}) - \frac{1}{n} \boldsymbol{1}_{n+1}\boldsymbol{1}_{n+1}^T,
$$
and $C_n \boldsymbol{1}_{n+1}=0$, where we use the notation
$$
\boldsymbol{1}_n = \left ( \begin{array}{c} 1 \\ \vdots \\ 1 \end{array}\right ) \in \R^{n}.
$$

\medskip
\medskip

The construction of $C_n$ is inductive, starting with $n=2$.  Let
$$
C_2 = \left (  \begin{array}{ccc}  \frac{\sqrt{3}}{2} & -\frac{\sqrt{3}}{2} & 0 \\ -\frac{1}{2} & -\frac{1}{2} & 1 \end{array}\right ),
$$
and let $\mathbf{c}^j = C_n \mathbf{e}^j$ be the columns of $C_2$.  It is straight forward to verify that
$$
\mathbf{c}^i \cdot \mathbf{c}^j = \frac{3}{2}\delta_{ij}- \frac{1}{2}.
$$
It is also straight forward to check that $C_2 \boldsymbol{1}_3 = 0$.

The last two properties of $C_2$ that we need are the following:
$$
C_2C_2^T = \frac{3}{2} I(2), \,\,\,\mbox{and}\,\,\,  C_2^T C_2 = \frac{3}{2} I(3) - \frac{1}{2} \boldsymbol{1}_3\boldsymbol{1}_3^T.
$$
The first of these comes from the fact that the rows of $C_2$, thought of as vectors in $\R^3$, are orthogonal to each other, and have length $\frac{3}{2}$.  The second comes from the fact that the $ij$ entry of $C_2^T C_2$ is exactly $\mathbf{c}^i \cdot \mathbf{c}^j$.

Let us now suppose we have an $n\times (n+1)$ matrix $C_n$, $n\geq 2$, such that $C_nC_n^T = \frac{n+1}{n} I(n)$, that $C_n \boldsymbol{1}_{n+1}=0$, and that
$$
C_n^TC_n = \frac{n+1}{n}I(n+1) - \frac{1}{n} \boldsymbol{1}_{n+1}\boldsymbol{1}_{n+1}^T.
$$
All these conditions hold for $n=2$.  With all this let us define the $(n+1)\times (n+2)$ matrix
$$
C_{n+1} = \left (  \begin{array}{cc}  \frac{\sqrt{(n+1)^2-1}}{n+1}C_n & \begin{array}{c}0 \\ \vdots \\ 0 \end{array} \\ \begin{array}{ccc}-\frac{1}{n+1} & \hdots & -\frac{1}{n+1}\end{array} & 1 \end{array}\right ).
$$
Note that the last row of $C_{n+1}$, thought of as a vector in $\R^{n+2}$, is orthogonal to the others by the hypothesis $C_n \boldsymbol{1}_{n+1}$, and its length squared is $\frac{n+2}{n+1}$, whereas the length squared of the other rows is
$$
\frac{n+1}{n} \frac{n^2+2n}{(n+1)^2} = \frac{n+2}{n+1}.
$$
Since the other rows of $C_{n+1}$ are orthogonal among themselves by the hypotheses on $C_n$, this shows that $C_{n+1}C_{n+1}^T = \frac{n+2}{n+1} I(n+1)$.

It is also easy to see that each column of $C_{n+1}$ has length $1$, and that the dot product of the last column with any of the others is exactly $-\frac{1}{n+1}$.  Since the dot product between two different columns of $C_n$ is $-\frac{1}{n}$, we conclude that the dot product of any two different columns of $C_{n+1}$, not including the last, is
$$
-\frac{1}{n} \frac{n^2+2n}{(n+1)^2}+\frac{1}{(n+1)^2} = -\frac{1}{n+1}.
$$
This shows that
$$
C_{n+1}^TC_{n+1} = \frac{n+2}{n+1}I({n+2}) - \frac{1}{n+1} \boldsymbol{1}_{n+2}\boldsymbol{1}_{n+2}^T.
$$
Lastly, that $C_{n+1}\boldsymbol{1}_{n+2}=0$ follows easily from the form of the last row of $C_{n+1}$ and the hypotheses on $C_n$.  This concludes the proof of Step 1.

\medskip
\medskip

{\bf Step 2.}  In this step we now consider a family $\{\mathbf{u}^j\}_{j=1}^{n+1}\subset \R^n$ of unit vectors that satisfy the inner product conditions \eqref{e:innerprodcondition}, and its associated 3-tensor $\mathcal{Q} \in \Htrace(n, 3)$ generated by these $\mathbf{u}^j$'s via \eqref{e:def3tensor}.  Let $\mathcal{Q} = (\mathcal{Q}^1 | ... | \mathcal{Q}^n)$ where the $\mathcal{Q}^j \in \Htrace(n, 2)$ are defined by
$$
\mathcal{Q}^j = \sum_{k=1}^n \langle \mathbf{u}^k, \mathbf{e}^j\rangle \mathbf{u}^k(\mathbf{u}^k)^T.
$$
We will denote $A_n$ the $n\times (n+1)$ the matrix whose columns are the vectors $\{\mathbf{u}^j\}_{j=1}^{n+1}$.  It is easy to check that the inner product conditions \eqref{e:innerprodcondition} say exactly that
$$
A_n^T A_n = \frac{n+1}{n}I({n+1}) - \frac{1}{n} \boldsymbol{1}_{n+1}\boldsymbol{1}_{n+1}^T,
$$
whereas equations \ref{e:sumvecszero} and \ref{e:sumprojections} from Lemma \ref{e:propsvecsangles} say that
$$
A_n \boldsymbol{1}_{n+1}=0 \,\,\, \mbox{and}\,\,\, A_nA_n^T = \frac{n+1}{n}{I}(n).
$$
Our main claim in this step is the following:
$$
\mbox{The matrix}\,\,\,  R = \frac{n}{n+1}A_nC_n^T\,\,\, \mbox{is an orthogonal matrix, that is}\,\,\, RR^T = R^TR = {I}(n).
$$
From here it follows that $A_n =RC_n$, and $\mathbf{u}^j = A_n\mathbf{e}^j = RC_n \mathbf{e}^j$.

\medskip
\medskip

The proof is straight forward.  Indeed, we first observe that
$$
RR^T = \frac{n^2}{(n+1)^2}A_nC_n^TC_nA_n^T = \frac{n^2}{(n+1)^2}A_n \left (  \frac{n+1}{n}\mathbb{I}(n+1)-\frac{1}{n}\boldsymbol{1}_{n+1}\boldsymbol{1}_{n+1}^T\right )A_n^T = {I}(n),
$$
because $A_n\boldsymbol{1}_{n+1}=0$ and $A_nA_n^T = \frac{n+1}{n}{I}(n)$.

Next, we notice that
$$
RC_n = \frac{n}{n+1}A_nC_n^TC_n = \frac{n}{n+1}A_n\left (  \frac{n+1}{n}{I}(n+1)-\frac{1}{n}\boldsymbol{1}_{n+1}\boldsymbol{1}_{n+1}^T\right )=A_n,
$$
where again we used the fact that $A_n\boldsymbol{1}_{n+1}=0$.

\medskip
\medskip

{\bf Step 3.}  In this last step we build our eigenvector-eigentensor pairs from what we obtained in the previous step, and prove that the $3$-tensor $\mathcal{Q}$ we mentioned at the beginning of the previous step can be reconstructed from these pairs.  Our definitions are the following:
$$
\mbox{for the eigenvectors we set}\,\,\, \mathbf{f}^k = R\mathbf{e}^k, \,\,\, \mbox{and for the eigentensors}\,\,\, B^k = \frac{1}{\lambda_n}\sum_{j=1}^n\langle \mathbf{f}^k, \mathbf{e}^j\rangle \mathcal{Q}_j,
$$
where $\lambda_n = \left (\frac{(n^2-1)(n+1)}{n^3}\right )^{\frac{1}{2}}$.  We need to prove
$$
\mathbf{f}^j \cdot \mathbf{f}^k = \langle B^j, B^k \rangle = \delta_{jk},
$$
$$
\sum_{j=1}^n\langle \mathbf{f}^k, \mathbf{e}^j\rangle \mathcal{Q}_j = \lambda_n B^k \,\,\,\mbox{and}\,\,\, \sum_{j=1}^n \langle B^k, \mathcal{Q}^k \rangle \mathbf{e}^k = \lambda_n \mathbf{f}^k.
$$
$\mathbf{f}^j \cdot \mathbf{f}^k = \delta_{jk}$ is direct from their definition and the fact that $R\in O(n)$, whereas $\mathop{\sum}\limits_{j=1}^n\langle \mathbf{f}^k, \mathbf{e}^j\rangle \mathcal{Q}_j = \lambda_n B^k$ is by definition of $B^k$.  To prove the other two we first observe that $\langle \mathcal{Q}^j, \mathcal{Q}^k\rangle = \frac{(n+1)(n^2-1)}{n^3}\delta_{jk}=\lambda_n^2\delta_{jk}$, which is a direct consequence of equation \ref{e:suminteriorprod} from Proposition \ref{p:3tensorlaws}.  Then we compute
$$
\langle B^k, B^l \rangle = \frac{1}{\lambda_n^2}\sum_{i, j=1}^n \langle \mathbf{f}^k, \mathbf{e}^i\rangle \langle \mathbf{f}^l, \mathbf{e}^j\rangle \langle \mathcal{Q}_i, \mathcal{Q}_j \rangle = \mathbf{f}^k \cdot \mathbf{f}^l = \delta_{kl}.
$$
Next we compute
$$
\sum_{j=1}^n \langle B_k, \mathcal{Q}_j\rangle \mathbf{e}^j = \frac{1}{\lambda_n}\sum_{i, j=1}^n \langle \mathbf{f}^k, \mathbf{e}^j\rangle \langle \mathcal{Q}^j, \mathcal{Q}^i \rangle  \mathbf{e}^i = \lambda_n \sum_{j=1}^n \langle \mathbf{f}^k, \mathbf{e}^j\rangle \mathbf{e}^j = \lambda_n \mathbf{f}^k,
$$
where we again used the fact that $\langle \mathcal{Q}_j, \mathcal{Q}_k\rangle = \lambda_n^2\delta_{jk}$.  

Lastly we observe that
$$
\sum_{j=1}^n \langle \mathbf{f}^j, \mathbf{e}^i\rangle B_j = \sum_{j, k=1}^n \langle \mathbf{f}^j, \mathbf{e}^i\rangle \langle \mathbf{f}^j, \mathbf{e}^k \rangle \mathcal{Q}_k = \mathcal{Q}_i.
$$
This is the statement that the $3$-tensor $\mathcal{Q}$ can be recovered from the eigenvector-eigentensor pairs.

\end{proof}

\section{Appendix B: Proof of Theorem~\ref{t:recovery_3_d}}
\label{appendix:b}

In this section we will prove a useful property of the tensors in our algebraic variety.  Let us first recall that
\begin{align*}
\mathbb{H}(n,k)
&= \underbrace{\mathbb{R}^n \otimes \cdots \otimes \mathbb{R}^n}_{k \hbox{ times}}
\\
\Hsym(n,k) &= \{ \mathcal{A} \in \mathbb{H}(n,k) \hbox{ such that } \mathcal{A}_{\sigma(i_1 \ldots i_k)} = \mathcal{A}_{i_1 \ldots i_k} \hbox{ for all } \sigma \hbox{ permutations}\},
\end{align*}
and 
$$
\Htrace (n,k) = \{ \mathcal{A} \in \Hsym(n,k) \hbox{ such that } \sum_{j=1}^n \mathcal{A}_{i_1\ldots i_{k-2}j j} = 0 \hbox{ for all } i_\ell \in \{1, \ldots, n\} \}.
$$
Let also ${I}(n)$ denote the $n\times n$ identity matrix.

We define a contraction operator on symmetric tensors which is, in effect, the trace over the last two components.  
For $\mathcal{A} \in \mathbb{H}(n,k)$
we set the "block trace" operator, ${\rm BTr}:\mathbb{H}(n,k) \to \mathbb{H}(n,k-2)$, to be
\begin{equation} \label{e:blocktraceDef}
    {\rm BTr} (\mathcal{A})_{a_1 \ldots  a_{k-2}}
    = \sum_{\ell = 1}^n 
    \mathcal{A}_{a_1\ldots a_{k-2} \ell \ell} 
\end{equation}
for any $a_\ell \in \{1,\ldots, n\}$.
Consequently, if $\mathcal{Q} \in \Htrace(n, 3)$ then ${\rm BTr}(\mathcal{Q}) = 0$.  The block trace operator will be used extensively in the proof of Theorem~\ref{t:recovery_3_d}, as it was used in the recovery theorem in \cite{GMS}.
\medskip
\medskip

In this section we will think of the set $\mathbb{H}(3,2)$ as the set of $3\times 3$ matrices, and $\Hsym(3,2)$ as the set of symmetric $3\times 3$ ones.  For $A, B \in \mathbb{H}(3,2)$ we define the inner product
$$
\langle A, B\rangle = {\rm tr}(B^TA).
$$
We then define
$$
\mathbf{X}:\mathbb{H}(3,2) \to \R^9
$$
by the equation
$$
\mathbf{X}(A) :=  \mathbf{X}_A = \left ( \begin{array}{c} A \mathbf{e}^1 \\ A \mathbf{e}^2 \\ A \mathbf{e}^3 \end{array}\right ),
$$
where $\{\mathbf{e}^j\}_{j=1}^3$ denotes the canonical basis in $\R^3$.  
Notice that
$$
\langle A, B \rangle = \langle \mathbf{X}_A, \mathbf{X}_B\rangle,
$$
where the inner product in the left-hand side is the inner product in $\mathbb{H}(3,2)$, while the one in the right-hand side is the standard dot product in $\R^9$.  In other words, the map ${\bf X}:\mathbb{H}(3,2) \to \R^9$ is an isometry.

\medskip
\medskip

In this section we will also consider an element of the set $\mathcal{Q}\in \mathbb{H}(3,3)$ as a $3\times 9$ matrices with real entries.  For this, we think of vectors $\mathbf{q}\in \R^{m}$, in any Euclidean space, as columns.  In particular, for $\mathbf{p} \in \R^{m}$, $\mathbf{q} \in \R^{n}$, $\mathbf{p}\mathbf{q}^T$ is an $m\times n$ matrix.  Then, for $\mathbf{q}^1, \mathbf{q}^2, \mathbf{q}^3 \in \R^3$, we identify the tensor $\mathbf{q}^1 \otimes \mathbf{q}^2 \otimes \mathbf{q}^3$ with the $3\times 9$ matrix
$$
\mathbf{q}^1 \left (\mathbf{X}_{\mathbf{q}^2(\mathbf{q}^3)^T} \right )^T.
$$
With this notation the {\it permutation operators} can be expressed as follows: for $\mathbf{q}^1, \mathbf{q}^2, \mathbf{q}^3 \in \R^3$, so that $\mathbf{q}^2(\mathbf{q}^3)^T \in \mathbb{H}(3,2)$, and a permutation $\sigma\in S_3$, define $\mathcal{T}_\sigma : \mathbb{H}(3,3) \to \mathbb{H}(3,3)$ by
$$
\mathcal{T}_\sigma \left ( \mathbf{q}^1\left (\mathbf{X}_{\mathbf{q}^2(\mathbf{q}^3)^T} \right )^T \right ) = \mathbf{q}^{\sigma(1)}\left (\mathbf{X}_{\mathbf{q}^{\sigma(2)}(\mathbf{q}^{\sigma(3)})^T} \right )^T,
$$
and extend $T_\sigma$ to $\mathbb{H}(3,2)$ by linearity.

\medskip
\medskip
\medskip

Next recall the block-trace operator ${\rm BTr}:\Hsym(3,3)\to \mathbb{R}^3$
in \eqref{e:blocktraceDef}.  It is easy to check that this operator satisfies the condition
$$
{\rm BTr}\left (  \mathbf{q}^1\left (\mathbf{X}_{\mathbf{q}^2(\mathbf{q}^3)^T} \right )^T \right ) = \langle \mathbf{q}^1, \mathbf{q}^3 \rangle \mathbf{q}^2.
$$
This operator considers a $3\times 9$ matrix $\mathcal{Q}\in \mathbb{H}(3,3)$ as being built from $3$ blocks $\mathcal{Q}_j \in \mathbb{H}(3,2)$, and sends $\mathcal{Q}$ into a vector containing the trace of $\mathcal{Q}_j$ in its $j$-th component.

\medskip
\medskip

With all this notation in place we can restate the definition of $\Htrace (3, 3)$ as follows:
\begin{equation}
\Htrace (3, 3) = \{{\mathcal Q} \in \mathbb{H}(3,3): \,\,  {\rm BTr}(\mathcal{Q}) = 0, \,\,  \mathcal{T}_\sigma(\mathcal{Q}) = \mathcal{Q}\,\,\forall \,\, \sigma\in S_3\}.
\end{equation}
We now record some properties of tensors $\mathcal{Q}\in \Hsym(3, 3)$ that will be important to us.  Let us recall that the canonical orthonormal basis of $\R^3$ is denoted by $\{\mathbf{e}^j\}_{j=1}^3$.
\begin{proposition}
Let $\mathcal{Q}\in \Hsym(3, 3)$, and write $\mathcal{Q} = \left ( \mathcal{Q}_1 | \mathcal{Q}_2 | \mathcal{Q}_3 \right )$, where $\mathcal{Q}_j \in \Hsym(3, 2)$.  It holds
\begin{equation}\label{switch_blocks_rows}
\mathcal{Q} = \sum_{j=1}^3 \mathbf{e}^j \left ( \mathbf{X}_{\mathcal{Q}_j} \right )^T.
\end{equation}
Furthermore, we have
\begin{equation}\label{product_of_blocks}
\mathcal{Q}_i \mathcal{Q}_j = \sum_{k=1}^3 \mathcal{Q}_k \mathbf{e}^i (\mathbf{e}^j)^T \mathcal{Q}_k,
\end{equation}
for every $i, j = 1, 2, 3$.  Finally, we have
\begin{equation}\label{Q_i_e_j}
\mathcal{Q}_i \mathbf{e}^j = \mathcal{Q}_j \mathbf{e}^i, 
\end{equation}
also for every $i, j = 1, 2, 3$.
\end{proposition}
\begin{proof}
To prove identity \eqref{switch_blocks_rows}, for $\sigma_{\rm br}(1, 2, 3) = (3, 2, 1)$, let $\mathcal{T}_{\rm br} = \mathcal{T}_{\sigma_{\rm br}}$, and observe that the rows of the tensor $\mathcal{T}_{\rm br}(\mathcal{Q})$, as vectors in $\R^9$, correspond to the blocks of $\mathcal{Q}$.  Since $\mathcal{Q}\in \Hsym(3, 3)$, by definition we have $\mathcal{T}_{\rm br}(\mathcal{Q})=\mathcal{Q}$.  This shows \eqref{switch_blocks_rows}.

\medskip
\medskip

To prove \eqref{product_of_blocks}, observe that, owing to equation \eqref{switch_blocks_rows}, we have
$$
\mathcal{Q}^T\mathcal{Q} = \left ( \begin{array}{ccc} \mathcal{Q}_1\mathcal{Q}_1 & \mathcal{Q}_1\mathcal{Q}_2 & \mathcal{Q}_1\mathcal{Q}_3 \\ \mathcal{Q}_2\mathcal{Q}_1 & \mathcal{Q}_2\mathcal{Q}_2 & \mathcal{Q}_2\mathcal{Q}_3 \\ \mathcal{Q}_3\mathcal{Q}_1 & \mathcal{Q}_3\mathcal{Q}_2 & \mathcal{Q}_3\mathcal{Q}_3 \end{array} \right ) = \sum_{k=1}^3 \mathbf{X}_{\mathcal{Q}_k} \left (\mathbf{X}_{\mathcal{Q}_k}\right )^T.
$$
\eqref{product_of_blocks} follows from this last identity, and the definition of $\mathbf{X}:\mathbb{H}(3, 2) \to \R^9$.

\medskip
\medskip

Finally, to prove \eqref{Q_i_e_j}, recall that by definition
$$
\mathbf{X}_{\mathcal{Q}_j} = \left (  \begin{array}{c} \mathcal{Q}_j \mathbf{e}^1 \\ \mathcal{Q}_j \mathbf{e}^2 \\ \mathcal{Q}_j \mathbf{e}^3 \end{array} \right ).
$$
Because of this and \eqref{switch_blocks_rows}, we have
$$
\mathcal{Q} = \sum_{j=1}^3 \mathbf{e}^j \mathbf{X}_{\mathcal{Q}_j}^T = \sum_{j=1}^3 \left ( \mathbf{e}^j \left ( \mathbf{e}^1 \right )^T \mathcal{Q}_j | \mathbf{e}^j \left ( \mathbf{e}^2 \right )^T \mathcal{Q}_j | \mathbf{e}^j \left ( \mathbf{e}^3 \right )^T \mathcal{Q}_j \right ).
$$
From here we obtain
$$
\mathcal{Q}_i = \sum_{j=1}^3 \mathbf{e}^j \left ( \mathbf{e}^i \right )^T \mathcal{Q}_j = \sum_{j=1}^3 \mathcal{Q}_j \mathbf{e}^i \left ( \mathbf{e}^j \right )^T,
$$
where the last equation is because $\mathcal{Q}_i^T = \mathcal{Q}_i$.  \eqref{Q_i_e_j} follows directly from here.

\end{proof}

It will  be useful to have an explicit basis for the space $\Htrace(3, 3)$, in terms of the canonical orthonormal basis $\{\mathbf{e}^j\}_{j=1}^3$ of $\R^3$.  For this, let us define the tensors
$$
P^i = \mathbf{e}^i(\mathbf{e}^i)^T, \,\,\,i=1, 2, 3,
$$
associated to the canonical basis $\{\mathbf{e}^j\}_{j=1}^3$.
Note that the $P^j$'s are rank-$1$, orthogonal projection matrices with orthogonal images, and $\mathbb{I}(3) = \mathop{\sum}\limits_{i=1}^3 P^i$.  We emphasize that these $P^j$'s are different than those defined in Section~\ref{s:nD}. Define also
$$
S^{i, j} = \mathbf{e}^i (\mathbf{e}^j)^T + \mathbf{e}^j (\mathbf{e}^i)^T, \,\,\,  1 \leq i \neq j \leq 3.
$$
Observe that the $P^j$ together with the $S^{i,j}$ provide a basis for the set of symmetric, $3\times 3$ matrices with real entries.  With all these we now have the following.
\begin{proposition}\label{basis_Ten_3_3}
The following list provides a basis for $\Hsym(3,3)$:
\begin{align*}
{\mathcal A}^i & = \mathbf{e}^i \left ( \mathbf{X}_{P^i}\right )^T, i=1, 2, 3, \\
{\mathcal B}^{i, j} & = \mathbf{e}^i \left ( \mathbf{X}_{P^j} \right )^T+ \mathbf{e}^j  \left (\mathbf{X}_{S^{i, j}}\right )^T, \,\,\,\, 1 \leq i \neq j \leq 3, \\
{\mathcal C} & = \mathop{\Sigma}\limits_{\sigma \in S_3} T_\sigma \left (\mathbf{e}^1 \left ( \mathbf{X}_{\mathbf{e}^2 \left ( \mathbf{e}^3 \right )^T}\right )^T \right )= \mathbf{e}^1  \left (\mathbf{X}_{S^{2,3}}\right )^T + \mathbf{e}^2  \left (\mathbf{X}_{S^{1,3}}\right )^T + \mathbf{e}^3  \left (\mathbf{X}_{S^{1,2}}\right )^T.
\end{align*}
Given this notation, the following list is a basis for $\Htrace(3,3)$:
\begin{align*}
{\mathcal B}^{i, j} - {\mathcal A}^i & =  \mathbf{e}^i \left ( \mathbf{X}_{P^j} - \mathbf{X}_{P^i} \right )^T+ \mathbf{e}^j  \left (\mathbf{X}_{S^{i, j}}\right )^T, \,\,\,\, 1 \leq i \neq j \leq 3.\\
{\mathcal C} & = \mathbf{e}^1  \left (\mathbf{X}_{S^{2,3}}\right )^T + \mathbf{e}^2  \left (\mathbf{X}_{S^{1,3}}\right )^T + \mathbf{e}^3  \left (\mathbf{X}_{S^{1,2}}\right )^T.
\end{align*}
\end{proposition}
Note that this implies that $\Htrace(3,3)$ has dimension $7$.

Next, we will need to work in the set $\mathbb{H}(3, 4)$, and we recall that, by Remark \ref{rm:tensors_3_4_9_2}, we identify elements $\mathcal{Q}\in \mathbb{H}(3, 4)$ with $\mathcal{Q}\in \mathbb{H}(9, 2)$.  With this in mind, we point out that, for $\sigma \in S_4$ we can also define a permutation operator $\mathcal{T}_\sigma : \mathbb{H}(9, 2)\to \mathbb{H}(9, 2)$.  To define it, let  $\mathbf{a}^j\in \R^{3}$, $j=1, 2, 3, 4$, define $\mathcal{T}_\sigma$ by the condition
$$
\mathcal{T}_\sigma\left ( \mathbf{X}_{\mathbf{a}^1\left ( \mathbf{a}^2\right )^T} \left (\mathbf{X}_{\mathbf{a}^3\left ( \mathbf{a}^4\right )^T} \right )^T  \right ) = \mathbf{X}_{\mathbf{a}^{\sigma(1)}\left ( \mathbf{a}^{\sigma(2)}\right )^T} \left (\mathbf{X}_{\mathbf{a}^{\sigma(3)}\left ( \mathbf{a}^{\sigma(4)}\right )^T} \right )^T,
$$
and extend it to $\mathbb{H}(9, 2)$ by linearity.  We trust that the use of the same notation for the permutation operators in $\mathbb{H}(3, 4)$, $\mathbb{H}(9, 2)$ and $\mathbb{H}(3, 3)$ will not be a source of confusion.

There are two families of elements in $\mathbb{H}(9, 2)$ that will appear naturally in the proof of Lemma~\ref{long_unimaginative_lemma}.  Here we use the notation from the Appendix of \cite{GMS}.  For two $3\times 3$ matrices $A, B\in \mathbb{H}(3, 2)$ define
\begin{equation}\label{first_matrix_op}
\mathcal{N}_{A, B} = \mathbf{X}_A\left ( \mathbf{X}_B \right )^T.
\end{equation}

Observe that
$$
\mathcal{N}_{A, B} \mathbf{X}_C = \langle B, C\rangle \mathbf{X}_A.
$$
When $A=B$ we will simply write $\mathcal{N}_A$ in place of $\mathcal{N}_{A, A}$.

\medskip
\medskip

\noindent Next, again for $A, B\in \mathbb{H}(3, 2)$, let $\mathcal{M}_{A, B}\in \mathbb{H}(9, 2)$ be the $9\times 9$ matrix defined by the equation
\begin{equation}\label{second_matrix_op}
\mathcal{M}_{A, B} \mathbf{X}_C = \mathbf{X}_{ACB^T} \,\,\,  \forall \,\,C\in \mathbb{H}(3, 2).
\end{equation}
Again, when $A=B$ we will write $\mathcal{M}_A$ in place of $\mathcal{M}_{A, A}$.  From \cite{GMS}, we can give an expression for $\mathcal{M}_{A, B}$, $A, B \in \mathbb{H}(3, 2)$ as follows:  denoting $A_{i, j} = \langle A, \mathbf{e}^i \left ( \mathbf{e}^j \right )^T\rangle$, that is, $A_{i, j}$ is the $ij$ entry of $A$, then
\begin{equation}\label{matrix_conjugation}
\mathcal{M}_{A, B} = \left ( \begin{array}{ccc} B_{11}A & B_{12}A & B_{13}A \\ B_{21}A & B_{22}A & B_{23}A \\ B_{31}A & B_{32}A & B_{33}A \end{array} \right ).
\end{equation}

\medskip
\medskip
The next proposition appears in \cite{GMS}.  We include its proof for the reader's convenience.
\begin{proposition}\label{operator_T_*}
For the permutation $\sigma_*(1, 2, 3, 4) = (1,3,2,4)$, the operator $\mathcal{T}_* = \mathcal{T}_{\sigma_*} :\mathbb{H}(9, 2) \to \mathbb{H}(9, 2)$ satisfies
\begin{equation}\label{rel_matrix_ops}
\mathcal{T}_*(\mathcal{N}_{A, B}) = \mathcal{T}_{\sigma_*}(\mathcal{N}_{A, B}) = \mathcal{M}_{A, B}
\end{equation}
for every $A, B\in \mathbb{H}(3, 2)$.
\end{proposition}
\begin{proof}
A simple way to see that the operator $\mathcal{T}_*$ indeed has this property is to observe first that by definition
$$
\mathcal{T}_* \left ( \mathbf{X}_{\mathbf{a}^1\left ( \mathbf{a}^2\right )^T} \left (\mathbf{X}_{\mathbf{a}^3\left ( \mathbf{a}^4\right )^T} \right )^T \right ) = \mathbf{X}_{\mathbf{a}^1\left ( \mathbf{a}^3\right )^T} \left (\mathbf{X}_{\mathbf{a}^2\left ( \mathbf{a}^4\right )^T} \right )^T.
$$
Calling $A=\mathbf{a}^1\left ( \mathbf{a}^2\right )^T$, $B=\mathbf{a}^3\left ( \mathbf{a}^4\right )^T$, and $C=\mathbf{u}^1\left ( \mathbf{u}^2\right )^T$, then clearly
$$
\mathbf{X}_{\mathbf{a}^1\left ( \mathbf{a}^3\right )^T} \left (\mathbf{X}_{\mathbf{a}^2\left ( \mathbf{a}^4\right )^T} \right )^T \mathbf{X}_{C} = \langle \mathbf{a}^2\left ( \mathbf{a}^4 \right )^T, C \rangle \mathbf{X}_{\mathbf{a}^1 \left ( \mathbf{a}^3 \right )^T} = \langle \mathbf{a}^2, \mathbf{u}^1 \rangle \langle \mathbf{a}^4, \mathbf{u}^2 \rangle \mathbf{X}_{\mathbf{a}^1 \left ( \mathbf{a}^3 \right )^T} = \mathbf{X}_{ACB^T}.
$$
The result of the proposition follows by linearity, and the fact that every matrix in $\mathbb{H}(3, 2)$ is a linear combination of rank-1 matrices of the form $\mathbf{a}\left ( \mathbf{b} \right )^T$ with $\mathbf{a}, \mathbf{b} \in \R^3$.
\end{proof}

\begin{remark}\label{rm:change_basis}
For $\mathbf{q}^1, \mathbf{q}^2, \mathbf{q}^3 \in \R^3$, and $R\in O(3)$, the definition of $\mathcal{M}_R$ reads:
$$
\mathcal{M}_{R}\mathbf{X}_{\mathbf{q}^2 \left ( \mathbf{q}^3 \right )^T} = \mathbf{X}_{R\mathbf{q}^2 \left ( R\mathbf{q}^3 \right )^T}.
$$
Since every $\mathcal{Q}\in \Hsym(3, 3)$ is a linear combination of tensors of the form
$$
\sum_{\sigma \in S_3} \mathcal{T}_\sigma \left ( \mathbf{q}^1 \left(\mathbf{X}_{\mathbf{q}^2 \left ( \mathbf{q}^3 \right )^T} \right )^T \right ),
$$
it follows from here that
$$
R\,\mathcal{Q} \,\left (\mathcal{M}_R\right )^T = R\,\mathcal{Q} \,\mathcal{M}_{R^T} \in \Hsym(3, 3)
$$
for every $\mathcal{Q}\in \Hsym(3, 3)$.

\end{remark}


Next, we again recall the block-trace operator ${\rm BTr}: \Hsym(3,4) \to \Hsym(3,2)$ from \eqref{e:blocktraceDef}.  Under the identification of $\mathbb{H}(3, 4)$ with $\mathbb{H}(9, 2)$, this operator can be thought of as ${\rm BTr}: \Hsym(9,2) \to \Hsym(3,2)$, and then it can be defined by the condition
$$
{\rm BTr} \left (\mathbf{X}_{\mathbf{q}^1 \left ( \mathbf{q}^2 \right )^T}\left ( \mathbf{X}_{\mathbf{q}^3 \left ( \mathbf{q}^4\right )^T }\right )^T \right )= \langle \mathbf{q}^3, \mathbf{q}^4\rangle \,\, \mathbf{q}^1\left (\mathbf{q}^2\right )^T,
$$
Thinking of the elements $\mathcal{R}\in \Hsym(3, 4)$ as $9\times 9$ matrices, and we look at such an $\mathcal R$ as being built from $3\times 3$ blocks, then ${\rm BTr}(\mathcal{R})$ just adds the $3\times3$ blocks in the diagonal of $\mathcal{R}$.  For this reason we refer to ${\rm Btr}$ as the {\it block-trace} operator


\begin{remark}\label{rm:use_of_block_trace}
For $\mathcal{Q}\in \mathbb{H}(3, 3)$, a direct computation shows that ${\rm BTr}\left (\mathcal{Q}^T\mathcal{Q} \right ) = \mathcal{Q}\mathcal{Q}^T$.  This is a tensor analog of the fact that $\tr \left ( \mathbf{a}\left ( \mathbf{b}\right )^T \right ) = \mathbf{b}^T \mathbf{a} = \langle \mathbf{a}, \mathbf{b} \rangle$ for $\mathbf{a}, \mathbf{b}\in \R^3$.  However, to establish Lemma \ref{long_unimaginative_lemma} below, and its consequence, Corollary \ref{block_sum_property}, we need to analyze the full tensor $\mathcal{Q}^T\mathcal{Q}$, rather than $\mathcal{Q}\mathcal{Q}^T$.  This our main motivation to introduce this operator.
\end{remark}

Next, construction a basis for $\Hsym(3,4)$ in the spirit of \cite{GMS}.  We recall here that, by Remark \ref{rm:tensors_3_4_9_2}, we identify $\Hsym(3,4)$ with $\Hsym(9,2)$.
\begin{proposition}\label{basis_H_4_block_traces}
For the space 
$$
\Hsym(9, 2) = \{\mathcal{Q}\in \mathbb{H}(9, 2) : \,\, \,   T_\sigma(\mathcal{Q}) = \mathcal{Q} \,\,\,\,  \forall \,\, \sigma \in S_4\},
$$
the following list provides a basis:
\begin{align*}
{\mathcal P}^i &= \mathbf{X}_{P^i}\left ( \mathbf{X}_{P^i}\right )^T, \\
{\mathcal D}^{i, j} &= \mathbf{X}_{P^i}\left ( \mathbf{X}_{S^{i, j}}\right )^T + \mathbf{X}_{S^{i, j}}\left ( \mathbf{X}_{P^i}\right )^T, \\
{\mathcal F}^{i, j} &= \mathbf{X}_{P^i}\left ( \mathbf{X}_{P^j}\right )^T + \mathbf{X}_{S^{i, j}} \left (\mathbf{X}_{S^{i, j}}\right )^T + \mathbf{X}_{P^j} \left ( \mathbf{X}_{P^i}\right )^T, \\
{\mathcal G}^{i, j, k} &= \mathbf{X}_{P^i}\left ( \mathbf{X}_{S^{j, k}}\right )^T + \mathbf{X}_{S^{i, j}} \left (\mathbf{X}_{S^{i, k}}\right )^T + \mathbf{X}_{S^{i, k}} \left (\mathbf{X}_{S^{i, j}}\right )^T + \mathbf{X}_{S^{j, k}} \left (\mathbf{X}_{P^i}\right )^T, 
\end{align*}
for $1 \leq i, j, k \leq 3  \,\,\,  i\neq j, i \neq k, j \neq k$.
For their block-traces we have
\begin{align*}
{\rm BTr}({\mathcal P}^i) &= P^i, \\
{\rm BTr}({\mathcal D}^{i, j}) &= S^{i, j}, \\
{\rm BTr}({\mathcal F}^{i, j}) &= {\rm BTr}\left (\mathbf{X}_{S^{i, j}}\left (\mathbf{X}_{S^{i, j}}\right )^T \right ) = P^i + P^j, \,\,\,\mbox{and}\\
{\rm BTr}({\mathcal G}^{i, j, k}) &= {\rm BTr}\left (\mathbf{X}_{S^{i, j}}\left (  \mathbf{X}_{S^{i, k}}\right )^T+ \mathbf{X}_{S^{i, k}} \left (\mathbf{X}_{S^{i, j}}\right )^T \right ) = S^{j, k}.
\end{align*}
\end{proposition}

\begin{proof}
The proof of the basis is straightforward and follows from tensor combinations of the $\mathbf{e}^j$'s in $\mathbb{R}^3$, along with symmetry assumptions in the indices. 
The block-trace identities likewise follow from the tensor constructions and the contraction of the last two indices.
\end{proof}

\begin{remark}\label{rm:use_of_bases}
It is important to notice that in general ${\mathcal D}^{i, j} \neq {\mathcal D}^{j, i}$, whereas ${\mathcal F}^{i, j} = {\mathcal F}^{j, i}$.  On the other hand ${\mathcal G}^{i, j, k}$ depends on the order of the indices $i, j, k$; however, ${\mathcal G}^{i, j, k} = {\mathcal G}^{i, k, j}$.

The main use of the bases for $\Htrace(3,3)$ and $\Hsym(9, 2)$ will be to prove Lemma \ref{long_unimaginative_lemma}.  Indeed, using these bases we will compute $\mathcal{Q}^T\mathcal{Q}$, which satisfies $\mathcal{Q}^T\mathcal{Q} \in \mathbb{H}(9, 2)$ for $\mathcal{Q}\in \Htrace(3, 3)$, and write it as a sum of a term that belongs to $\Hsym(9, 2)$, and a second term to which we can apply Proposition \ref{operator_T_*}.  The analysis of these terms will give us the proof of the aforementioned Lemma.
\end{remark}

\medskip
\medskip

\subsection{Block-trace conditions on $\Htrace(3,3)$} \hfill \\

A result for permutation invariant $3$-tensors that have traceless blocks and satisfy a normalization condition.  
Our main result in this section is the following.
\begin{lemma}\label{long_unimaginative_lemma}
Let ${\mathcal Q} \in \Htrace (3, 3)$, and assume
$$
{\rm BTr}({\mathcal Q}^T{\mathcal Q}) = \alpha \,\,  {I}(3),
$$
for some $\alpha > 0$.  Then there is ${\mathcal S} \in \Hsym(9, 2)$ such that
$$
{\mathcal Q}^T {\mathcal Q} = {\mathcal S} -\frac{\alpha}{2} \mathbf{X}_{I(3)}\left ( \mathbf{X}_{I(3)} \right )^T  = \mathcal{S}-\frac{\alpha}{2}\mathcal{N}_{I(3)},
$$
where we are using the notation $\mathcal{N}_A$ defined in equation \eqref{first_matrix_op}.
\end{lemma}
Before giving the proof of this lemma we derive a corollary from it that we will need for our recovery argument.
\begin{corollary}\label{block_sum_property}
Let ${\mathcal Q} \in \Htrace(3, 3)$, and assume
$$
{\mathcal Q}{\mathcal Q}^T = \alpha \,\,  {I}(3),
$$
for some $\alpha > 0$.  Write $\mathcal{Q}= \left (  \mathcal{Q}_1 | \mathcal{Q}_2 |\mathcal{Q}_3 \right )$, where $\mathcal{Q}_j \in \Hsym(3, 2)$ and ${\rm tr}(\mathcal{Q}_j)=0$.  Then
$$
\sum_{j=1}^3 \mathcal{Q}_j \mathcal{Q}_i \mathcal{Q}_j = \frac{\alpha}{2}\mathcal{Q}_i
$$
for $i=1, 2, 3$.
\end{corollary}
\begin{proof}
First observe that $\mathcal{Q}\mathcal{Q}^T = \mathop{\sum}\limits_{j=1}^3 \mathcal{Q}_j^2$.  A direct computation shows that ${\rm BTr}({\mathcal Q}^T{\mathcal Q}) =  \mathop{\sum}\limits_{j=1}^3 \mathcal{Q}_j^2$.  We conclude that ${\rm BTr}({\mathcal Q}^T{\mathcal Q}) =  \alpha \,\, {I}(3)$, so $\mathcal{Q}$ satisfies the hypotheses of Lemma \ref{long_unimaginative_lemma}.

\medskip
\medskip

\noindent Next, recall that by Remark \ref{switch_blocks_rows}, we can write
$$
\mathcal{Q} = \mathop{\sum}\limits_{j=1}^3 \mathbf{e}^j \left (\mathbf{X}_{\mathcal{Q}_j} \right )^T.
$$
Because of this, we obtain 
$$
\mathcal{Q}\mathcal{Q}^T = \sum_{j=1}^3 \langle \mathcal{Q}_i, \mathcal{Q}_j\rangle \mathbf{e}^i \left ( \mathbf{e}^j \right)^T.
$$
In particular, the hypothesis ${\mathcal Q}{\mathcal Q}^T = \alpha \,\,  {I}(3)$ tells us that
\begin{equation}\label{inner_prods}
\langle \mathcal{Q}_i, \mathcal{Q}_j\rangle = \delta_{ij}\, \alpha.
\end{equation}
Also from Remark \ref{switch_blocks_rows} we obtain the expression
\begin{equation}\label{QQ^T_first}
\mathcal{Q}^T\mathcal{Q} = \mathop{\sum}\limits_{j=1}^3 \mathbf{X}_{\mathcal{Q}_j}  \left (\mathbf{X}_{\mathcal{Q}_j} \right )^T = \sum_{j=1}^3 \mathcal{N}_{\mathcal{Q}_j}.
\end{equation}
By Lemma \ref{long_unimaginative_lemma} we know that
\begin{equation}\label{QQ^T_second}
\mathcal{Q}^T \mathcal{Q} = \mathcal{S} - \frac{\alpha}{2} \mathcal{N}_{I(3)}.
\end{equation}
Now recall the operator $\mathcal{T}_* :\mathbb{H}(3, 4)\to \mathbb{H}(3, 4)$ defined in \eqref{rel_matrix_ops}.  Applying this operator to \eqref{QQ^T_first} we obtain
$$
\mathcal{T}_*(\mathcal{Q}^T\mathcal{Q}) = \sum_{j=1}^3 \mathcal{M}_{\mathcal{Q}_j}.
$$
We can also apply $\mathcal{T}_*$ to \eqref{QQ^T_second} to obtain
$$
\mathcal{T}_*(\mathcal{Q}^T\mathcal{Q}) = \mathcal{S} - \frac{\alpha}{2} \mathcal{M}_{I(3)} = \mathcal{S} - \frac{\alpha}{2} \mathcal{N}_{I(3)} + \frac{\alpha}{2} \mathcal{N}_{I(3)} - \frac{\alpha}{2} \mathcal{M}_{I(3)} = \mathcal{Q}^T\mathcal{Q} + \frac{\alpha}{2} \mathcal{N}_{I(3)} - \frac{\alpha}{2} \mathcal{M}_{I(3)},
$$
because $\mathcal{S}\in \Hsym(9, 2)$ and \eqref{QQ^T_second}.

\medskip
\medskip

To conclude, we observe that
$$
\mathcal{T}_*(\mathcal{Q}^T\mathcal{Q}) \mathbf{X}_{\mathcal{Q}_i} =  \sum_{j=1}^3 \mathcal{M}_{\mathcal{Q}_j} \mathbf{X}_{\mathcal{Q}_i} = \sum_{j=1}^3 \mathbf{X}_{\mathcal{Q}_j \mathcal{Q}_i \mathcal{Q}_j},
$$
by the definition of $\mathcal{M}_A$ given in \eqref{second_matrix_op}.  However, we also have
$$
\mathcal{T}_*(\mathcal{Q}^T\mathcal{Q}) \mathbf{X}_{\mathcal{Q}_i} = \left ( \mathcal{Q}^T\mathcal{Q} + \frac{\alpha}{2} \mathcal{N}_{I(3)} - \frac{\alpha}{2} \mathcal{M}_{I(3)} \right ) \mathbf{X}_{\mathcal{Q}_i}  = \frac{\alpha}{2}\mathbf{X}_{\mathcal{Q}_i},
$$
by the definitions \eqref{first_matrix_op} and \eqref{second_matrix_op} of $\mathcal{N}_A$ and $\mathcal{M}_A$ respectively, the fact that $\tr{\mathcal{Q}_i}=0$,  \eqref{inner_prods}, and the expression \eqref{QQ^T_first} for $\mathcal{Q}^T\mathcal{Q}$.  The last two equations give the claim of the Corollary.

\end{proof}

\begin{proof}[Proof of Lemma \ref{long_unimaginative_lemma}]
 The proof consists in computing ${\mathcal Q}^T {\mathcal Q}$ using structure of the bases for $\Htrace(3, 3)$ and $\Hsym(9, 2)$ 
 provided by Lemmas \ref{basis_Ten_3_3} and \ref{basis_H_4_block_traces} to conclude.

\medskip
\medskip

\noindent {\bf Step 1.}  By Proposition \ref{basis_Ten_3_3} we can write ${\mathcal Q}\in \Htrace(3, 3)$ as
\begin{align}
{\mathcal Q}  = & \beta_{1, 2}({\mathcal B}^{1, 2} - {\mathcal A}^1) + \beta_{1, 3}({\mathcal B}^{1, 3} - {\mathcal A}^1) \nonumber \\
&+ \beta_{2, 1}({\mathcal B}^{2, 1} - {\mathcal A}^2) + \beta_{2, 3}({\mathcal B}^{2, 3} - {\mathcal A}^2) \nonumber \\
&+ \beta_{3, 1}({\mathcal B}^{3, 1} - {\mathcal A}^3) + \beta_{3, 2}({\mathcal B}^{3, 2} - {\mathcal A}^3) \nonumber \\
&+ \gamma \, {\mathcal C} \nonumber \\
:=  & \mathcal{H}_\mathcal{Q}^1 + \mathcal{H}_\mathcal{Q}^2 + \mathcal{H}_\mathcal{Q}^3 + \mathcal{H}_\mathcal{Q}^4.
\end{align}

A direct computation shows that
\begin{equation}
{\mathcal Q}^T {\mathcal Q} = \sum_{i=1}^4 (\mathcal{H}_\mathcal{Q}^i)^T \mathcal{H}_\mathcal{Q}^i 
+\sum_{1\leq i < j \leq 3}( (\mathcal{H}_\mathcal{Q}^i)^T \mathcal{H}_\mathcal{Q}^j + (\mathcal{H}_\mathcal{Q}^j)^T \mathcal{H}_\mathcal{Q}^i) 
+ \sum_{i=1}^3 ( (\mathcal{H}_\mathcal{Q}^i)^T \mathcal{H}_\mathcal{Q}^4 + (\mathcal{H}_\mathcal{Q}^4)^T \mathcal{H}_\mathcal{Q}^i) .
\end{equation}
We will expand each of these sums.
\medskip
\medskip

\noindent {\bf Step 2.}  Computation of $(\mathcal{H}_\mathcal{Q}^i)^T \mathcal{H}_\mathcal{Q}^i$.  By Proposition \ref{basis_Ten_3_3} we have
$$
\mathcal{H}_\mathcal{Q}^1 = \beta_{1, 2}\left (\mathbf{e}^1 \left (\mathbf{X}_{P^2} - \mathbf{X}_{P^1}\right )^T + \mathbf{e}^2 \left ( \mathbf{X}_{S^{1, 2}}\right )^T \right) + \beta_{1,3} \left (\mathbf{e}^1 \left ( \mathbf{X}_{P^3}- \mathbf{X}_{P^1} \right )^T + \mathbf{e}^3 \left (\mathbf{X}_{S^{1, 3}}\right )^T \right ).
$$
Since $\langle \mathbf{e}^i, \mathbf{e}^j \rangle = \delta_{i, j}$, we obtain
\begin{align*}
(\mathcal{H}_\mathcal{Q}^1)^T \mathcal{H}_\mathcal{Q}^1 &= \beta_{1, 2}^2 \left ( \mathbf{X}_{P^2-P^1} \left ( \mathbf{X}_{P^2-P^1}\right )^T  + \mathbf{X}_{S^{1, 2}}\left (\mathbf{X}_{S^{1, 2}}\right )^T \right ) \\
&+ \beta_{1, 2}\beta_{1, 3} \left ( \mathbf{X}_{P^2-P^1} \left ( \mathbf{X}_{P^3-P^1} \right )^T + \mathbf{X}_{P^3-P^1} \left ( \mathbf{X}_{P^2-P^1}\right )^T \right ) \\
&+ \beta_{1, 3}^2 \left ( \mathbf{X}_{P^3-P^1} \left (\mathbf{X}_{P^3-P^1}\right )^T + \mathbf{X}_{S^{1, 3}}\left (\mathbf{X}_{S^{1, 3}}\right )^T \right).
\end{align*}
From here, by adding and subtracting terms of the form $\mathbf{X}_{S^{j, k}}\left ( \mathbf{X}_{S^{j, k}}\right )^T$ we obtain
\begin{align*}
(\mathcal{H}_\mathcal{Q}^1)^T \mathcal{H}_\mathcal{Q}^1 &= \beta_{1, 2}^2 ( {\mathcal P}^1 + {\mathcal P}^2 - {\mathcal F}^{1, 2}) \\
&+ \beta_{1, 2}\beta_{1, 3} ( {\mathcal F}^{2, 3} - {\mathcal F}^{1, 2} - {\mathcal F}^{1, 3} + 2{\mathcal P}^1 )\\
&+ \beta_{1,3}^2 ( {\mathcal P}^1 + {\mathcal P}^3 - {\mathcal F}^{1, 3} ) \\
&+2\beta_{1, 2}^2 \mathbf{X}_{S^{1, 2}}\left ( \mathbf{X}_{S^{1, 2}}\right )^T \\
&+ \beta_{1, 2}\beta_{1, 3} \left (\mathbf{X}_{S^{1, 2}}\left ( \mathbf{X}_{S^{1, 2}}\right )^T+\mathbf{X}_{S^{1, 3}}\left ( \mathbf{X}_{S^{1, 3}}\right )^T  - \mathbf{X}_{S^{2, 3}}\left ( \mathbf{X}_{S^{2, 3}}\right )^T \right )\\
&+ 2\beta_{1,3}^2 \mathbf{X}_{S^{1, 3}}\left ( \mathbf{X}_{S^{1, 3}}\right )^T.
\end{align*}
Observe that in this expression the first three terms belong to $\Hsym(3,4)$, and have Block-trace equal to zero.  In contrast, the last three terms are not permutation invariant, and each has Block-trace equal to a linear combination of the $P^j$'s.


The same argument gives us
\begin{align*}
(\mathcal{H}_\mathcal{Q}^2)^T \mathcal{H}_\mathcal{Q}^2 &= \beta_{2, 1}^2 ( {\mathcal P}^1 + {\mathcal P}^2 - {\mathcal F}^{1, 2}) \\
&+ \beta_{2, 1}\beta_{2, 3} ( {\mathcal F}^{1, 3} - {\mathcal F}^{1, 2} - {\mathcal F}^{2, 3} + 2{\mathcal P}^2 )\\
&+ \beta_{2, 3}^2 ( {\mathcal P}^2 + {\mathcal P}^3 - {\mathcal F}^{2, 3} ) \\
&+2\beta_{2, 1}^2 \mathbf{X}_{S^{1, 2}}\left ( \mathbf{X}_{S^{1, 2}}\right )^T \\
&+ \beta_{2, 1}\beta_{2, 3} \left (\mathbf{X}_{S^{1, 2}}\left ( \mathbf{X}_{S^{1, 2}}\right )^T+\mathbf{X}_{S^{2, 3}}\left ( \mathbf{X}_{S^{2, 3}}\right )^T  - \mathbf{X}_{S^{1, 3}}\left ( \mathbf{X}_{S^{1, 3}}\right )^T \right )\\
&+ 2\beta_{2,3}^2 \mathbf{X}_{S^{2, 3}}\left ( \mathbf{X}_{S^{2, 3}}\right )^T,
\end{align*}
and
\begin{align*}
(\mathcal{H}_\mathcal{Q}^3)^T \mathcal{H}_\mathcal{Q}^3 &= \beta_{3, 1}^2 ( {\mathcal P}^1 + {\mathcal P}^3 - {\mathcal F}^{1, 3}) \\
&+ \beta_{3, 1}\beta_{3, 2} ( {\mathcal F}^{1, 2} - {\mathcal F}^{1, 3} - {\mathcal F}^{2, 3} + 2{\mathcal P}^3 )\\
&+ \beta_{3, 2}^2 ( {\mathcal P}^2 + {\mathcal P}^3 - {\mathcal F}^{2, 3} ) \\
&+2\beta_{3, 1}^2 \mathbf{X}_{S^{1, 3}}\left ( \mathbf{X}_{S^{1, 3}}\right )^T \\
&+ \beta_{3, 1}\beta_{3, 2} \left (\mathbf{X}_{S^{1, 3}}\left ( \mathbf{X}_{S^{1, 3}}\right )^T+\mathbf{X}_{S^{2, 3}}\left ( \mathbf{X}_{S^{2, 3}}\right )^T  - \mathbf{X}_{S^{1, 2}}\left ( \mathbf{X}_{S^{1, 2}}\right )^T \right )\\
&+ 2\beta_{3, 2}^2 \mathbf{X}_{S^{2, 3}}\left ( \mathbf{X}_{S^{2, 3}}\right )^T.
\end{align*}
Finally, we also have
$$
(\mathcal{H}_\mathcal{Q}^4)^T \mathcal{H}_\mathcal{Q}^4 = \gamma^2 \left (\mathbf{X}_{S^{2, 3}}\left ( \mathbf{X}_{S^{2, 3}}\right )^T + \mathbf{X}_{S^{1, 3}}\left ( \mathbf{X}_{S^{1, 3}}\right )^T + \mathbf{X}_{S^{1, 2}}\left ( \mathbf{X}_{S^{1, 2}}\right )^T \right).
$$


\noindent {\bf Step 3.}  Computation of $(\mathcal{H}_\mathcal{Q}^i)^T \mathcal{H}_\mathcal{Q}^j + (\mathcal{H}_\mathcal{Q}^j)^T \mathcal{H}_\mathcal{Q}^i$, $1 \leq i < j \leq 3$.


First we observe that
\begin{align*}
(\mathcal{H}_\mathcal{Q}^1)^T \mathcal{H}_\mathcal{Q}^2 + (\mathcal{H}_\mathcal{Q}^2)^T \mathcal{H}_\mathcal{Q}^1 &= \beta_{1, 3}\beta_{2, 1} \left ( \mathbf{X}_{P^3 - P^1} \left (\mathbf{X}_{S^{1, 2}} \right )^T + \mathbf{X}_{S^{1, 2}} \left ( \mathbf{X}_{P^3-P^1} \right )^T \right) \\
&+ \beta_{1, 2}\beta_{2, 3}\left (\mathbf{X}_{P^3-P^2} \left (\mathbf{X}_{S^{1, 2} }\right )^T + \mathbf{X}_{S^{1, 2}}\left ( \mathbf{X}_{P^3-P^2}\right )^T \right ) \\
&+ \beta_{1, 3}\beta_{2, 3} \left ( \mathbf{X}_{S^{1, 3}}\left ( \mathbf{X}_{S^{2, 3}}\right )^T + \mathbf{X}_{S^{2, 3}}\left ( \mathbf{X}_{S^{1, 3}}\right )^T \right).
\end{align*}
We now add and subtract terms of the form $\mathbf{X}_{S^{1, 3}}\left ( \mathbf{X}_{S^{2, 3}}\right )^T + \mathbf{X}_{S^{2, 3}}\left ( \mathbf{X}_{S^{1, 3}}\right )^T$ to obtain
\begin{align*}
(\mathcal{H}_\mathcal{Q}^1)^T \mathcal{H}_\mathcal{Q}^2 + (\mathcal{H}_\mathcal{Q}^2)^T \mathcal{H}_\mathcal{Q}^1 &= \beta_{1, 3}\beta_{2, 1} (\mathcal{G}^{3, 1, 2} - {\mathcal D}^{1, 2}) + \beta_{1, 2}\beta_{2, 3}(\mathcal{G}^{3, 1, 2}- {\mathcal D}^{2, 1}) \\
&+ (\beta_{1, 3}\beta_{2, 3} - \beta_{1,3}\beta_{2, 1}-\beta_{1, 2}\beta_{2, 3})\left (\mathbf{X}_{S^{1, 3}}\left ( \mathbf{X}_{S^{2, 3}}\right )^T + \mathbf{X}_{S^{2, 3}}\left ( \mathbf{X}_{S^{1, 3}}\right )^T \right ).
\end{align*}
Observe again that the two terms on right hand side of the first line are Block-traceless, permutation invariant, whereas the term in the second line is not.
We also have
\begin{align*}
(\mathcal{H}_\mathcal{Q}^1)^T \mathcal{H}_\mathcal{Q}^3 + (\mathcal{H}_\mathcal{Q}^3)^T \mathcal{H}_\mathcal{Q}^1 &= \beta_{1, 2}\beta_{3, 1} ({\mathcal G}^{2, 1, 3} - {\mathcal D}^{1, 3}) + \beta_{1, 3}\beta_{3, 2}(\mathcal{G}^{2, 1, 3}- {\mathcal D}^{3, 1}) \\
&+ (\beta_{1, 2}\beta_{3, 2} - \beta_{1,2}\beta_{3, 1}-\beta_{1, 3}\beta_{3, 2})\left (\mathbf{X}_{S^{1, 2}}\left ( \mathbf{X}_{S^{2, 3}}\right )^T + \mathbf{X}_{S^{2, 3}}\left ( \mathbf{X}_{S^{1, 2}}\right )^T \right),
\end{align*}
and
\begin{align*}
(\mathcal{H}_\mathcal{Q}^2)^T \mathcal{H}_\mathcal{Q}^3 + (\mathcal{H}_\mathcal{Q}^3)^T \mathcal{H}_\mathcal{Q}^2 &= \beta_{2, 1}\beta_{3, 2} ({\mathcal G}^{1, 2, 3} - {\mathcal D}^{2, 3}) + \beta_{2, 3}\beta_{3, 1}(\mathcal{G}^{1,2, 3}- {\mathcal D}^{3, 2}) \\
&+ (\beta_{2, 1}\beta_{3, 1} - \beta_{2,1}\beta_{3, 2}-\beta_{2, 3}\beta_{3, 1})\left (\mathbf{X}_{S^{1, 2}}\left ( \mathbf{X}_{S^{1, 3}}\right )^T + \mathbf{X}_{S^{1, 3}}\left ( \mathbf{X}_{S^{1, 2}}\right )^T \right).
\end{align*}


\noindent {\bf Step 3.}  Computation of $\mathop{\sum}\limits_{i=1}^3 ( (\mathcal{H}_\mathcal{Q}^i)^T\mathcal{H}_\mathcal{Q}^4 + (\mathcal{H}_\mathcal{Q}^4)^T \mathcal{H}_\mathcal{Q}^i)$.  With the same logic we have used so far we have
\begin{align*}
(\mathcal{H}_\mathcal{Q}^1)^T\mathcal{H}_\mathcal{Q}^4 + (\mathcal{H}_\mathcal{Q}^4)^T \mathcal{H}_\mathcal{Q}^1 &= \gamma \beta_{1, 2}({\mathcal D}^{2, 3} - {\mathcal G}^{1, 2, 3}) + \gamma \beta_{1, 3}({\mathcal D}^{3, 2}- {\mathcal G}^{1, 2, 3}) \\
&+ 2\gamma(\beta_{1,2}+ \beta_{1, 3})\left (\mathbf{X}_{S^{1, 2}}\left ( \mathbf{X}_{S^{1, 3}}\right )^T+\mathbf{X}_{S^{1, 3}}\left ( \mathbf{X}_{S^{1, 2}}\right )^T \right),
\end{align*}
\begin{align*}
(\mathcal{H}_\mathcal{Q}^2)^T\mathcal{H}_\mathcal{Q}^4 + (\mathcal{H}_\mathcal{Q}^4)^T \mathcal{H}_\mathcal{Q}^2 &= \gamma \beta_{2, 1}({\mathcal D}^{1, 3} - {\mathcal G}^{2, 1, 3}) + \gamma \beta_{2, 3}({\mathcal D}^{3, 1}- {\mathcal G}^{2, 1, 3}) \\
&+ 2\gamma(\beta_{2,1}+ \beta_{2, 3}) \left (\mathbf{X}_{S^{1, 2}}\left ( \mathbf{X}_{S^{2, 3}}\right )^T+\mathbf{X}_{S^{2, 3}}\left ( \mathbf{X}_{S^{1, 2}}\right )^T \right),
\end{align*}
and
\begin{align*}
(\mathcal{H}_\mathcal{Q}^3)^T\mathcal{H}_\mathcal{Q}^4 + (\mathcal{H}_\mathcal{Q}^4)^T \mathcal{H}_\mathcal{Q}^3 &= \gamma \beta_{3, 1}({\mathcal D}^{1, 2} - {\mathcal G}^{3, 1, 2}) + \gamma \beta_{3, 2}({\mathcal D}^{2, 1}- {\mathcal G}^{3, 1, 2}) \\
&+ 2\gamma(\beta_{3, 1}+ \beta_{3, 2}) \left (\mathbf{X}_{S^{1, 3}}\left ( \mathbf{X}_{S^{2, 3}}\right )^T+\mathbf{X}_{S^{2, 3}}\left ( \mathbf{X}_{S^{1, 3}}\right )^T \right).
\end{align*}

\medskip
\medskip

\noindent {\bf Step 4.}  Consequences of ${\rm BTr}({\mathcal Q}^T{\mathcal Q}) = \alpha {I}(3)$.  From all our previous computations we now have the following:
\begin{align*}
{\rm BTr}({\mathcal Q}^T{\mathcal Q}) &= 2((\beta_{1, 2}^2 + \beta_{2, 1}^2) + (\beta_{1, 3}^2 + \beta_{3, 1}^2 ) + \beta_{1, 2}\beta_{1,3} + \gamma^2) P^1 \\
&+ 2((\beta_{1, 2}^2 + \beta_{2, 1}^2) + (\beta_{2, 3}^2 + \beta_{3, 2}^2 ) + \beta_{2, 1}\beta_{2,3} + \gamma^2) P^2 \\
&+ 2((\beta_{1, 3}^2 + \beta_{3, 1}^2) + (\beta_{2, 3}^2 + \beta_{3, 2}^2 ) + \beta_{3, 1}\beta_{3, 2} + \gamma^2) P^3\\
&+(\beta_{1, 3}\beta_{2, 3}- \beta_{1, 3}\beta_{2, 1}-\beta_{1, 2}\beta_{2, 3} + 2\gamma(\beta_{3, 1}+\beta_{3, 2})) S^{1, 2}\\
&+ (\beta_{1, 2}\beta_{3, 2}- \beta_{1, 2}\beta_{3, 1}-\beta_{1, 3}\beta_{3, 2} + 2\gamma(\beta_{2, 1}+\beta_{2, 3})) S^{1, 3} \\
&+ (\beta_{2, 1}\beta_{3, 1}- \beta_{2, 1}\beta_{3, 2}-\beta_{2, 3}\beta_{3, 1} + 2\gamma(\beta_{1, 2}+\beta_{1, 3})) S^{2, 3}.
\end{align*}
From ${\rm BTr}({\mathcal Q}^T{\mathcal Q}) = \alpha I(3)$ we deduce that the coefficients of $S_{i, j}$, $i \neq j$, are zero, so
$$
\beta_{1, 3}\beta_{2, 3}- \beta_{1, 3}\beta_{2, 1}-\beta_{1, 2}\beta_{2, 3} + 2\gamma(\beta_{3, 1}+\beta_{3, 2}) = 0, 
$$
$$
\beta_{1, 2}\beta_{3, 2}- \beta_{1, 2}\beta_{3, 1}-\beta_{1, 3}\beta_{3, 2} + 2\gamma(\beta_{2, 1}+\beta_{2, 3}) = 0,
$$
and
$$
\beta_{2, 1}\beta_{3, 1}- \beta_{2, 1}\beta_{3, 2}-\beta_{2, 3}\beta_{3, 1} + 2\gamma(\beta_{1, 2}+\beta_{1, 3}) = 0.
$$
Looking carefully at all the terms we have obtained for ${\mathcal Q}^T{\mathcal Q}$ we deduce that there is ${\mathcal S} \in \Hsym(9, 2)$ such that
\begin{align}
{\mathcal Q}^T {\mathcal Q}-{\mathcal S} &= (2(\beta_{1, 2}^2 + \beta_{2, 1}^2) + \beta_{1, 2}\beta_{1, 3} + \beta_{2, 1}\beta_{2, 3}- \beta_{3, 1}\beta_{3, 2} + \gamma^2) \mathbf{X}_{S^{1, 2}}\left ( \mathbf{X}_{S^{1, 2}}\right )^T \nonumber \\
&+ (2(\beta_{1, 3}^2 + \beta_{3, 1}^2) + \beta_{1, 2}\beta_{1, 3} + \beta_{3, 1}\beta_{3, 2}- \beta_{2, 1}\beta_{2, 3} + \gamma^2) \mathbf{X}_{S^{1, 3}}\left ( \mathbf{X}_{S^{1, 3}}\right )^T \nonumber \\
&+ (2(\beta_{2, 3}^2 + \beta_{3, 2}^2) + \beta_{2, 1}\beta_{2, 3} + \beta_{3, 1}\beta_{3, 2}- \beta_{1, 2}\beta_{1, 3} + \gamma^2) \mathbf{X}_{S^{2, 3}}\left ( \mathbf{X}_{S^{2, 3}}\right )^T.\label{first_Q-S}
\end{align}
Next, we observe that ${\rm BTr}({\mathcal Q}^T{\mathcal Q}) = \alpha I(3)$ also implies that
$$
(\beta_{1, 2}^2 + \beta_{2, 1}^2) + (\beta_{1, 3}^2 + \beta_{3, 1}^2 ) + \beta_{1, 2}\beta_{1,3} + \gamma^2 = \frac{\alpha}{2},
$$
$$
(\beta_{1, 2}^2 + \beta_{2, 1}^2) + (\beta_{2, 3}^2 + \beta_{3, 2}^2 ) + \beta_{2, 1}\beta_{2,3} + \gamma^2 = \frac{\alpha}{2},
$$
and
$$
(\beta_{1, 3}^2 + \beta_{3, 1}^2) + (\beta_{2, 3}^2 + \beta_{3, 2}^2 ) + \beta_{3, 1}\beta_{3, 2} + \gamma^2 = \frac{\alpha}{2}.
$$
Adding the first two of these equations and subtracting the third we obtain the identity
$$
2(\beta_{1, 2}^2 + \beta_{2, 1}^2) + \beta_{1, 2}\beta_{1, 3} + \beta_{2, 1}\beta_{2, 3}- \beta_{3, 1}\beta_{3, 2} + \gamma^2 = \frac{\alpha}{2}.
$$
Adding the first and third, and subtracting the second we obtain
$$
2(\beta_{1, 3}^2 + \beta_{3, 1}^2) + \beta_{1, 2}\beta_{1, 3} + \beta_{3, 1}\beta_{3, 2}- \beta_{2, 1}\beta_{2, 3} + \gamma^2 = \frac{\alpha}{2}.
$$
A similar procedure gives
$$
2(\beta_{2, 3}^2 + \beta_{3, 2}^2) + \beta_{2, 1}\beta_{2, 3} + \beta_{3, 1}\beta_{3, 2}- \beta_{1, 2}\beta_{1, 3} + \gamma^2 = \frac{\alpha}{2}.
$$
Using this in \eqref{first_Q-S} we obtain
$$
{\mathcal Q}^T {\mathcal Q}-{\mathcal S} = \frac{\alpha}{2}\left (\mathbf{X}_{S^{1, 2}}\left ( \mathbf{X}_{S^{1, 2}}\right )^T + \mathbf{X}_{S^{1, 3}}\left ( \mathbf{X}_{S^{1, 3}}\right )^T + \mathbf{X}_{S^{2, 3}}\left ( \mathbf{X}_{S^{2, 3}}\right )^T \right).
$$
Now recall from Proposition \ref{basis_H_4_block_traces} that
$$
\mathbf{X}_{S^{i, j}}\left ( \mathbf{X}_{S^{i, j}}\right )^T = {\mathcal F}^{i, j} - \mathbf{X}_{P^i}\left ( \mathbf{X}_{P^j} \right )^T - \mathbf{X}_{P^j} \left ( \mathbf{X}_{P^i} \right )^T.
$$
We deduce that
\begin{align*}
&\mathbf{X}_{S^{1, 2}}\left ( \mathbf{X}_{S^{1, 2}}\right )^T + \mathbf{X}_{S^{1, 3}}\left ( \mathbf{X}_{S^{1, 3}}\right )^T + \mathbf{X}_{S^{2, 3}}\left ( \mathbf{X}_{S^{2, 3}}\right )^T\\
& = {\mathcal F}^{1, 2} + {\mathcal F}^{1, 3} + {\mathcal F}^{2, 3} \\
& \quad - \left (\mathbf{X}_{P^1}\left (\mathbf{X}_{P^2}\right )^T  + \mathbf{X}_{P^2} \left (\mathbf{X}_{P^1}\right )^T + \mathbf{X}_{P^1} \left (\mathbf{X}_{P^3}\right )^T + \mathbf{X}_{P^3} \left ( \mathbf{X}_{P^1}\right )^T + \mathbf{X}_{P^2}\left (\mathbf{X}_{P^3}\right )^T + \mathbf{X}_{P^3}\left (\mathbf{X}_{P^2}\right )^T \right).
\end{align*}
Finally, by adding and subtracting terms of the form $\mathbf{X}_{P^i}\left (\mathbf{X}_{P^i}\right )^T$, we observe that
\begin{align*}
& \mathbf{X}_{P^1}\left (\mathbf{X}_{P^2}\right )^T  + \mathbf{X}_{P^2} \left (\mathbf{X}_{P^1}\right )^T + \mathbf{X}_{P^1} \left (\mathbf{X}_{P^3}\right )^T + \mathbf{X}_{P^3} \left ( \mathbf{X}_{P^1}\right )^T + \mathbf{X}_{P^2}\left (\mathbf{X}_{P^3}\right )^T + \mathbf{X}_{P^3}\left (\mathbf{X}_{P^2}\right )^T \\
& \qquad = \mathbf{X}_{I(3)} \left (\mathbf{X}_{I(3)} \right )^T - {\mathcal P}^1 - {\mathcal P}^2 - {\mathcal P}^3.
\end{align*}
At this point we move every permutation invariant term in $\mathbf{X}_{S^{1, 2}}\left ( \mathbf{X}_{S^{1, 2}}\right )^T + \mathbf{X}_{S^{1, 3}}\left ( \mathbf{X}_{S^{1, 3}}\right )^T + \mathbf{X}_{S^{2, 3}}\left ( \mathbf{X}_{S^{2, 3}}\right )^T$ to the $\mathcal S$ on the left hand side of the identity
$$
{\mathcal Q}^T {\mathcal Q}-{\mathcal S} = \frac{\alpha}{2} \left (\mathbf{X}_{S^{1, 2}}\left ( \mathbf{X}_{S^{1, 2}}\right )^T + \mathbf{X}_{S^{1, 3}}\left ( \mathbf{X}_{S^{1, 3}}\right )^T + \mathbf{X}_{S^{2, 3}}\left ( \mathbf{X}_{S^{2, 3}}\right )^T \right),
$$
and redefine $\mathcal S$ accordingly.  This gives us
$$
{\mathcal Q}^T {\mathcal Q} - {\mathcal S} = -\frac{\alpha}{2} \mathbf{X}_{I(3)} \left (\mathbf{X}_{I(3)} \right )^T,
$$
where ${\mathcal S}\in \Hsym(3,4)$ by construction.  

\end{proof}

\subsection{Recovery Procedure}

Our main result in this section are our two recovery theorems in Subsection~\ref{ss:3Dresults}.

\begin{proof}[Proof of Theorem~\ref{t:recovery_3_d}]

Let $\mathcal{Q}\in \Htrace(3, 3)$, and define $
\eta_{\mathcal{Q}} : \mathbb{S}^2 \to \Hsym(3, 2)$ by
\begin{equation}\label{def_eta}
\eta_{\mathcal{Q}}(\mathbf{a}) = \sum_{j=1}^3 \langle \mathbf{e}^j, \mathbf{a} \rangle \mathcal{Q}_j.
\end{equation}
Then define $\mu_{\mathcal{Q}} : \mathbb{S}^2 \to \R{}$ by
\[
\mu_{\mathcal{Q}}(\mathbf{a}) = {\rm det}(\eta_{\mathcal{Q}}(\mathbf{a})).
\]
We will show that the maxima of $\mu_{\mathcal{Q}}$ in $\mathbb{S}^2$ are the vectors we seek.  We divide the proof in steps.

\medskip
\medskip

\noindent {\bf Step 1.}  The critical point condition for $\mu_{\mathcal{Q}}$ in $\mathbb{S}^2$ is
$$
(\nabla_{\mathbf{a}} \mu_{\mathcal{Q}} )(\mathbf{a}) = \sum_{i, j, k=1}^3  \langle \mathbf{a}, \mathbf{e}^j \rangle \langle \mathbf{a}, \mathbf{e}^k \rangle \langle \mathcal{Q}_i, \mathcal{Q}_j\mathcal{Q}_k \rangle \mathbf{e}^i = \lambda\, \mathbf{a},
$$
where $\lambda \in \R{}$ is a Lagrange multiplier for the constraint $\mathbf{a}\in \mathbb{S}^2$.  To prove this, let us recall that if $A \in \mathbb{H}(3, 2,\mathbb{R})$
has ${\rm tr}(A) = 0$, then Cayley-Hamilton Theorem tells us that
$$
A^3 -\frac{\abs{A}^2}{2}A - {\rm det}(A) {I}(3) = 0.
$$
From here we deduce that
$$
{\rm det}(A) = \frac{1}{3}\left ( {\rm tr} \left (A^3 \right ) \right ).
$$
Since $\mathcal{Q}\in \Htrace(3, 3)$, then $\eta_{\mathcal{Q}}(\mathbf{a}) \in \Htrace(3, 2)$ whenever $\mathbf{a}\in \mathbb{S}^2$.  Hence
$$
\mu_{\mathcal{Q}}(\mathbf{a}) = \frac{1}{3} \sum_{i, j, k=1}^3 \langle \mathbf{a}, \mathbf{e}^i \rangle \langle \mathbf{a}, \mathbf{e}^j \rangle \langle \mathbf{a}, \mathbf{e}^k \rangle {\rm tr}(\mathcal{Q}_i\mathcal{Q}_j\mathcal{Q}_k).
$$
We then deduce that the critical point condition for $\mu_{\mathcal{Q}}$ among $\mathbf{a}\in \mathbb{S}^2$ reads
$$
(\nabla_{\mathbf{a}} \mu_{\mathcal{Q}} )(\mathbf{a}) = \sum_{i, j, k=1}^3  \langle \mathbf{a}, \mathbf{e}^j \rangle \langle \mathbf{a}, \mathbf{e}^k \rangle {\rm tr}(\mathcal{Q}_i\mathcal{Q}_j\mathcal{Q}_k) \mathbf{e}^i = \lambda\, \mathbf{a},
$$
where $\lambda \in \R{}$ is a Lagrange multiplier for the constraint $\mathbf{a}\in \mathbb{S}^2$.  This is the claim of the step once we recall that for $A, B \in \mathbb{H}(3, 2)$ we have
$$
\langle A, B \rangle = {\rm tr}(B^TA).
$$

\medskip
\medskip

\noindent {\bf Step 2.}  If $\mathbf{a}\in \mathbb{S}^2$ is a critical point of $\mu_{\mathcal{Q}}$, then we claim that $\mathbf{a}$ satisfies the identity
$$
\eta_{\mathcal{Q}}(\mathbf{a}) \,\mathbf{a} = \frac{2\lambda}{\alpha} \, \mathbf{a},
$$
where $\lambda$ is the Lagrange multiplier from Step 1, and $\alpha$ is the real number in the hypothesis $\mathcal{Q}\mathcal{Q}^T = \alpha {I}(3)$.  To prove this, let us first recall that \eqref{product_of_blocks} tells us that
$$
\mathcal{Q}_j \mathcal{Q}_k = \sum_{l=1}^3 \mathcal{Q}_l \mathbf{e}^j \left (\mathbf{e}^k \right )^T \mathcal{Q}_l.
$$
From here we obtain
$$
\sum_{j, k=1}^3 \langle \mathbf{a}, \mathbf{e}^j\rangle \langle \mathbf{a}, \mathbf{e}^k \rangle \mathcal{Q}_j \mathcal{Q}_k = \sum_{l=1}^3 \mathcal{Q}_l \mathbf{a} \mathbf{a}^T \mathcal{Q}_l.
$$
We use this in the claim of Step 1 to deduce that critical points $\mathbf{a}$ of $\mu_{\mathcal{Q}}$ satisfy
$$
\sum_{i, l=1}^3 \langle \mathcal{Q}_i , \mathcal{Q}_l \mathbf{a}\mathbf{a}^T \mathcal{Q}_l \rangle \mathbf{e}^i = \lambda \,\mathbf{a}.
$$
From here and Corollary \ref{block_sum_property} we obtain
\begin{equation}\label{a_eigenvalue}
\sum_{i=1}^3 \langle \mathcal{Q}_i , \mathbf{a}\mathbf{a}^T \rangle \mathbf{e}^i = \frac{2\lambda}{\alpha} \,\mathbf{a}.
\end{equation}
Now we observe the following:
\begin{align*}
\sum_{i=1}^3 \langle \mathcal{Q}_i , \mathbf{a}\mathbf{a}^T \rangle \mathbf{e}^i &= \sum_{i, j, k=1}^3 \langle \mathbf{a}, \mathbf{e}^j\rangle \langle \mathbf{a}, \mathbf{e}^k\rangle \langle \mathcal{Q}_i , \mathbf{e}^j\left (\mathbf{e}^k\right )^T \rangle \mathbf{e}^i \\
&= \sum_{i, j, k=1}^3 \langle\mathbf{a}, \mathbf{e}^j\rangle \langle \mathbf{a}, \mathbf{e}^k\rangle \langle \mathcal{Q}_i\mathbf{e}^k , \mathbf{e}^j \rangle \mathbf{e}^i \\
&= \sum_{i, j, k=1}^3\langle \mathbf{a}, \mathbf{e}^j\rangle \langle \mathbf{a}, \mathbf{e}^k\rangle \langle \mathcal{Q}_k\mathbf{e}^i , \mathbf{e}^j \rangle \mathbf{e}^i,
\end{align*}
where the last identity is a consequence of \eqref{Q_i_e_j}.  Recalling the definition of $\eta_{\mathcal{Q}}$ from \eqref{def_eta}, this last sequence of identities can be summarized as follows:
$$
\sum_{i=1}^3 \langle \mathcal{Q}_i , \mathbf{a}\mathbf{a}^T \rangle \mathbf{e}^i = \sum_{i=1}^3 \langle \eta_{\mathcal{Q}}(\mathbf{a}) \mathbf{e}^i, \mathbf{a}\rangle \,\mathbf{e}^i.
$$
However, $\eta_{\mathcal{Q}}(\mathbf{a}) \in \Htrace(3, 2)$, so
$$
\sum_{i=1}^3 \langle \mathcal{Q}_i , \mathbf{a}\mathbf{a}^T \rangle \mathbf{e}^i = \sum_{i=1}^3 \langle \eta_{\mathcal{Q}}(\mathbf{a}) \mathbf{e}^i, \mathbf{a}\rangle \,\mathbf{e}^i = \eta_{\mathcal{Q}}(\mathbf{a})\mathbf{a}.
$$
Using this in \eqref{a_eigenvalue} we obtain the claim of this step.

\medskip
\medskip

\noindent {\bf Step 3.}  We claim there is an $R\in SO(3)$ such that $\tilde{\mathcal{Q}} = R\,\mathcal{Q} \mathcal{M}_{R^T} \in \Htrace(3, 3)$ has $\tilde{\mathcal{Q}}\tilde{\mathcal{Q}}^T = \alpha\,I(3)$ and
$$
\eta_{\tilde{\mathcal{Q}}}(\mathbf{e}^1)\mathbf{e}^1 = \frac{2\lambda}{\alpha} \,\mathbf{e}^1.
$$
To prove this let us start by observing that $\tilde{\mathcal{Q}} = R\mathcal{Q}\mathcal{M}_{R^T} \in \Hsym(3, 3)$ follows from Remark~\ref{rm:change_basis}.  

Next, if $R\in O(3)$,  $\mathcal{Q}\mathcal{Q}^T = \alpha\,{I}(3)$, and $\tilde{\mathcal{Q}} = R\,\mathcal{Q} \mathcal{M}_{R^T}$, then
$$
\tilde{Q}\tilde{Q}^T = R\mathcal{Q}\mathcal{M}_{R^T}\mathcal{M}_{R}\mathcal{Q}^T R^T = R\mathcal{Q}\mathcal{Q}^T R^T = \alpha\,R {I}(3)R^T = \alpha\, {I}(3).
$$
Next, a direct computation shows that
$$
\eta_{\mathcal{Q}}(\mathbf{b}) \mathbf{b} = R^T \eta_{\tilde{\mathcal{Q}}}(R\mathbf{b})R\mathbf{b}.
$$
Since $\eta_{\mathcal{Q}}(\mathbf{b}) \mathbf{b} = \frac{2\lambda}{\alpha} \, \mathbf{b}$, we conclude that $\eta_{\tilde{\mathcal{Q}}}(R\mathbf{b})R\mathbf{b} = \frac{2\lambda}{\alpha} \, R\mathbf{b}$.  Finally, the step follows by choosing $R\in SO(3)$ such that $R\mathbf{b} = \mathbf{e}^1$.

\medskip
\medskip

\noindent {\bf Step 4.}  We claim that 
$$
\tilde{\mathcal{Q}} = \left ( \begin{array}{ccccccccc} \gamma & 0 & 0 & 0 & a & b & 0 & b & -\gamma-a \\ 0 & a & b & a & c & d & b & d & -c \\ 0 & b & -\gamma-a & b & d & -c & -\gamma-a & -c & -d \end{array} \right ),
$$
where $\gamma = \frac{2\lambda}{\alpha}$ with  the constraints
$$
0 = -2\gamma \,b, \,\,\,  c^2 + d^2 = \gamma^2 +\gamma  \,a \,\,\,\mbox{and}\,\,\,  \gamma(\gamma+2a) = 0.
$$
To see this let us write $\tilde{\mathcal{Q}} = \left ( \tilde{\mathcal{Q}}_1 | \tilde{\mathcal{Q}}_2 | \tilde{\mathcal{Q}}_3 \right )$, and observe that Step 3 tells us that $\tilde{\mathcal{Q}}_1 \mathbf{e}^1 = \eta_{\tilde{\mathcal{Q}}}(\mathbf{e}^1)\mathbf{e}^1 = \frac{2\lambda}{\alpha}\mathbf{e}^1$.  Since $\tilde{\mathcal{Q}}_j \in \Htrace(3, 2)$, we obtain
$$
\tilde{\mathcal{Q}}_1 = \left ( \begin{array}{ccc} \gamma & 0 & 0 \\ 0 & a & b \\ 0 & b & -\gamma-a \end{array} \right ),
$$
for some real numbers $a, b \in \R$.  From here we deduce the rest of $\tilde{\mathcal{Q}}$ by imposing the condition $\tilde{\mathcal{Q}}\in \Htrace(3, 4)$.  The constraints $0 = -2\gamma \,b$, $c^2 + d^2 = \gamma^2 +\gamma  \,a$ and $\gamma(\gamma+2a) = 0$ follow from imposing the $\langle \tilde{\mathcal{Q}}_2 , \tilde{\mathcal{Q}}_3 \rangle = 0$, $\abs{\tilde{\mathcal{Q}}_1}^2 = \abs{\tilde{\mathcal{Q}}_2}^2$ and $\abs{\tilde{\mathcal{Q}}_2}^2 = \abs{\tilde{\mathcal{Q}}_3}^2$, respectively.

\medskip
\medskip

\noindent {\bf Step 5.}  We now analyze the possible cases for $\lambda$.  We start with the case $\lambda=0$.  From out Step 4, in this case $\tilde{\mathcal{Q}}$ reduces to
$$
\tilde{\mathcal{Q}} = \left ( \begin{array}{ccccccccc} 0 & 0 & 0 & 0 & a & b & 0 & b & -a \\ 0 & a & b & a & 0 & 0 & b & 0 & 0 \\ 0 & b & -a & b & 0 & 0 & -a & 0 & 0 \end{array} \right ).
$$
In this case we define
$$
\mathbf{w} = \frac{\sqrt{-b+ia}}{\sqrt{a^2+b^2}}, \,\,\,  \mathbf{z} = i\mathbf{w}.
$$
Observe that the vectors $\mathbf{c}^2 = \left ( {\rm Re}(\mathbf{w}), {\rm Im}(\mathbf{w}) \right )^T$, $\mathbf{c}^3 = \left ( {\rm Re}(\mathbf{z}), {\rm Im}(\mathbf{z}) \right )^T$, have $\langle \mathbf{c}^i, \mathbf{c}^j \rangle = \delta_{ij}$.  Define next
$$
\mathbf{a}^1 = \mathbf{e}^1, \,\,\,  \mathbf{a}^2 = \left (  \begin{array}{c} 0 \\ \mathbf{c}^2 \end{array}\right ), \,\,\, \mathbf{a}^3 = \left (  \begin{array}{c} 0 \\ \mathbf{c}^3 \end{array}\right ),
$$
and let $\mathbf{b}= -\frac{1}{\sqrt{3}} (\mathbf{a}^1 + \mathbf{a}^2 + \mathbf{a}^3)$.  Clearly $\langle \mathbf{a}^i, \mathbf{a}^j \rangle = \delta_{ij}$, so $\abs{\mathbf{b}} \in \mathbb{S}^2$.  A straight forward computation whoes that
$$
\eta_{\tilde{\mathcal{Q}}}(\mathbf{b})\mathbf{b} = \frac{2\sqrt{a^2+b^2}}{3}\mathbf{b}.  
$$
In other words, we have found $\mathbf{b}\in \mathbb{S}^2$ such that $\eta_{\tilde{\mathcal{Q}}}(\mathbf{b})$ has $\mathbf{b}$ as an eigenvector with non-zero eigenvalue.  This is the situation we consider in the next step.

\medskip
\medskip

\noindent {\bf Step 6.}  We consider here the case $\lambda\neq 0$.  Let us first observe that if $\lambda < 0$, we can always change $\mathbf{b}$ by $-\mathbf{b}$.  This will change $\lambda$ by $-\lambda$.  Hence, we can actually assume $\lambda > 0$.  Under this assumption, from Step 4 we obtain
$$
\tilde{\mathcal{Q}} = \left ( \begin{array}{ccccccccc} \gamma & 0 & 0 & 0 & a & b & 0 & b & -\gamma-a \\ 0 & a & b & a & c & d & b & d & -c \\ 0 & b & -\gamma-a & b & d & -c & -\gamma-a & -c & -d \end{array} \right ),
$$
where $\gamma = \frac{2\lambda}{\alpha}$.  From the constraints of Step 4 we also have $b=0$, $a=-\frac{\gamma}{2}$ and $c^2+d^2 = \frac{\gamma^2}{2}$.  This reduces $\mathcal{Q}$ to
$$
\tilde{\mathcal{Q}} = \left ( \begin{array}{ccccccccc} \gamma & 0 & 0 & 0 & -\frac{\gamma}{2} & 0 & 0 & 0 & -\frac{\gamma}{2} \\ 0 & -\frac{\gamma}{2} & 0 & -\frac{\gamma}{2} & c & d & 0 & d & -c \\ 0 & 0 & -\frac{\gamma}{2} & 0 & d & -c & -\frac{\gamma}{2} & -c & -d \end{array} \right ).
$$
Since we also have $\tilde{\mathcal{Q}}\tilde{\mathcal{Q}}^T = \alpha \,\mathbb{I}(3)$, we deduce that $\gamma^2 = \frac{2\alpha}{3}$.

\medskip
\medskip

Observe now that
$$
\mathcal{R} = \left (  \mathcal{R}_1 | \mathcal{R}_2 \right ) = \left (  \begin{array}{cccc} c & d & d & -c \\ d & -c & -c & -d \end{array} \right )
$$
satisfies $\mathcal{R}\mathcal{R}^T = \frac{\alpha}{3} {I}(2)$, and $\mathcal{R} \in \Htrace(2, 3)$.  Because of this, we can apply the 2D recovery procedure for such a $2\times 4$ tensor.  This yields vectors $\mathbf{c}^j \in \R^2$, $j=2, 3, 4$, such that $\langle \mathbf{c}^i, \mathbf{c}^j \rangle = \frac{\delta}{2}(3\delta_{ij}-1)$, and
$$
\mathcal{R}_j = \sum_{k=2}^4 \langle \mathbf{c}^k, \mathbf{f}^j \rangle \,  \mathbf{c}^k\left ( \mathbf{c}^k \right )^T, \,\,\,  j=1, 2.
$$
Here $\{\mathbf{f}^1, \mathbf{f}^2\}$ is the canonical basis of $\R^2$ such that
$$
\mathbf{e}^2 = \left (  \begin{array}{c} 0 \\ \mathbf{f}^1 \end{array}\right ) \,\,\,  \mbox{and}\,\,\, \mathbf{e}^3 = \left (  \begin{array}{c} 0 \\ \mathbf{f}^2 \end{array}\right ).
$$
We then define
$$
\mathbf{a}^1 = \mathbf{e}^1, \,\,\,  \mathbf{a}^j = -\frac{1}{3}\mathbf{e}^1 + \frac{2\sqrt{2}}{3}\left (  \begin{array}{c} 0 \\ \mathbf{c}^j \end{array}\right ), \,\,\,  j=2, 3, 4.
$$
A long straight forward computation shows that $\tilde{\mathcal{Q}} = \left (  \tilde{\mathcal{Q}}_1 | \tilde{\mathcal{Q}}_2 | \tilde{\mathcal{Q}}_3 \right )$ has
$$
\tilde{\mathcal{Q}}_j = \sum_{k=1}^4 \langle \mathbf{e}^j, \mathbf{a}^k\rangle \mathbf{a}^k \left (  \mathbf{a}^k \right )^T.
$$
This completes the proof of Theorem \ref{t:recovery_3_d}.
\end{proof}

\section{Appendix C. Potential for bent-core liquid crystals}
\label{sec:rad}

In this section we consider the more general potential for tensors $\mathcal{Q} \in \Htrace(3, 3)$ proposed in \cite{LR2002} to model bent-core liquid crystals.  The potential appears in equation (4.4) of \cite{LR2002}, and in view of equation (4.5b) of the same paper, it can be expressed as
\begin{equation}
    \label{eq:radlub}
    W(\mathcal{Q}) = \frac{\abs{\mathcal{Q}}^4}{4} - \frac{\alpha}{2}\abs{\mathcal{Q}}^2 + \frac{\beta}{4}\sum_{i, j = 1}^3 \langle \mathcal{Q}_i, \mathcal{Q}_j\rangle ^2.
\end{equation}

Here, for $\mathcal{Q} \in \Htrace(3, 3)$, we write $\mathcal{Q} = \left ( \mathcal{Q}_1 | \mathcal{Q}_2 | \mathcal{Q}_3 \right )$, where $\mathcal{Q}_j \in \Hsym(3, 2)$ has ${\rm tr}(\mathcal{Q}_j)=0$.  Our main result in this section is a rigorous version of a similar statement in \cite{LR2002} as described in the following proposition.

\begin{proposition}
For $\beta \leq -2$ the potential $W$ is unbounded from below.  For $\beta > -\frac{29}{15}$, $\beta \neq 0$, if $\alpha = 0$, the only global minimizer of $W$ is the tensor $\mathcal{Q}=0$.  Assume now $\alpha \neq 0$.  For $-\frac{29}{15}<\beta < 0$ the global minimizers of $W$ satisfy the condition that $\mathcal{Q}\mathcal{Q}^T$ is a rank-2 projection and correspond to the MB frames.  For $\beta > 0$ the global minimizers of $W$ satisfy the condition that $\mathcal{Q}\mathcal{Q}^T$ is a multiple of the $3\times 3$ identity matrix, and correspond to tetrahedral frames.  For $\beta = 0$, if $\alpha \leq 0$, the global only minimizer of $W$ is the tensor $\mathcal{Q}=0$.  For $\beta = 0$ and $\alpha > 0$, the set of global minimizers is the set tensors $\mathcal{Q}\in \Htrace(3, 3)$ that satisfy $\abs{\mathcal{Q}}^2 = \alpha$.
\end{proposition}

As discussed in \cite{LR2002}, this proposition indicates that there is a temperature at which phase transition occurs between two phases---one with tetrahedral and another with MB symmetry---that can be modeled using the energy of the type \eqref{eq:grafstrong}, but with the potential \eqref{eq:radlub} if one were to ignore contributions from lower moments.

\begin{proof}
For $R\in SO(3)$ define
$$
\mathcal{R}_R = \left ( \begin{array}{ccc} R_{11}R & R_{12}R & R_{13}R \\ R_{21}R & R_{22}R & R_{23}R \\ R_{31}R & R_{32}R & R_{33}R \end{array}\right ).
$$
We know that for $\mathcal{Q} \in \Htrace(3, 3)$ we have
$$
R\,\mathcal{Q}\,\mathcal{R}_R^T \in \Htrace(3, 3).
$$
Also, a direct computation shows that for any $\mathcal{Q}\in \Htrace(3, 3)$ and any $R\in SO(3)$ we have
$$
W(\mathcal{Q}) = W(R\,\mathcal{Q}\, \mathcal{R}_R^T).
$$
Since $\mathcal{Q}\mathcal{Q}^T \in \Hsym(3, 2)$ and is non-negative definite, we can find $R\in SO(3)$ such that
$$
R\mathcal{Q}\mathcal{Q}^TR^T = \left ( \begin{array}{ccc}\lambda_1 &0 &0 \\ 0 & \lambda_2 & 0 \\ 0 & 0 & \lambda_3 \end{array} \right ),
$$
where $\lambda_j = \abs{\mathcal{Q}_j}^2 \geq 0$, $j=1, 2, 3$, and $\abs{\mathcal{Q}}^2 = \lambda_1 + \lambda_2 + \lambda_3$.  For this $R\in SO(3)$, and writing $\tilde{\mathcal{Q}} = R\,\mathcal{Q}\,\mathcal{R}_R^T$, we have
$$
W(\mathcal{Q}) = W(\tilde{\mathcal{Q}}) = \frac{1}{4}\left ( \sum_{j=1}^3 \lambda_j \right )^2 - \frac{\alpha}{2} \sum_{j=1}^3 \lambda_j + \frac{\beta}{4}\sum_{j=1}^3 \lambda_j^2.
$$
Next, the fact that $\mathcal{Q}\in \Htrace(3, 3)$ gives us some constraints the $\lambda_j$, $j=1, 2, 3$, must satisfy.  To see what these are we let $\mathcal{Q}\in \Htrace(3, 3)$, and write it in the form $\mathcal{Q} = ( \mathcal{Q}_1 | \mathcal{Q}_2| \mathcal{Q}_3 )$, where
$$
\mathcal{Q}_1 = \left ( \begin{array}{ccc}a & b & c \\ b & d & e \\ c & e & -(a+d) \end{array}  \right ), \,\,\,  \mathcal{Q}_2 = \left ( \begin{array}{ccc} b & d & e \\ d & f & g \\  e & g & -(b+f) \end{array}  \right ) \,\,\,\mbox{and}\,\,\, \mathcal{Q}_3 = \left ( \begin{array}{ccc} c & e & -(a+d) \\ e & g & -(b+f) \\  -(a+d)&-(b+f)&-(c+g) \end{array}  \right ).
$$
Recalling the notation $\lambda_j = \abs{\mathcal{Q}_j}^2$, a straight forward computation shows that
$$
\lambda_1 + \lambda_2 - \frac{2\lambda_3}{3}\geq \frac{5}{8}(a^2 + f^2) + \frac{2}{3}(c-g)^2
$$
and
$$
\lambda_3 \geq \frac{3}{2}(c+g)^2 + 2(b+f)^2 + 2(a+d)^2.
$$
From here we obtain
\begin{align}
\frac{\abs{\lambda_1-\lambda_2}}{2} &= \abs{a^2+ad+c^2-g^2-bf-b^2} = \abs{(a, -f, c-g)\cdot (a+d, b+f, c+g)} \nonumber \\ &\leq \left (  \lambda_3 \left ( \lambda_1 + \lambda_2 - \frac{2\lambda_3}{3} \right ) \right )^{\frac{1}{2}}. \label{e:est_lam_1_minus_Lam_2}
\end{align}
The fact that $\mathcal{Q}\in \Htrace(3, 3)$ implies, in summary, that the following inequalities hold:
\begin{align}
\frac{(\lambda_1-\lambda_2)^2}{4} &\leq \lambda_3 \left ( \lambda_1+\lambda_2 - \frac{2\lambda_3}{3}\right ), \nonumber \\
\frac{(\lambda_2-\lambda_3)^2}{4} &\leq \lambda_1 \left ( \lambda_2+\lambda_3 - \frac{2\lambda_1}{3}\right )\,\,\,\mbox{and} \nonumber \\
\frac{(\lambda_1-\lambda_3)^2}{4} &\leq \lambda_2 \left ( \lambda_1+\lambda_3 - \frac{2\lambda_2}{3}\right ).  \label{eq:restr_lambdas}
\end{align}
Adding these three inequalities, and passing the cross terms all to the right hand side, we obtain
$$
\frac{7}{6}\sum_{j=1}^3 (\lambda_j)^2 \leq \frac{5}{2}(\lambda_1\lambda_2 + \lambda_1\lambda_3 + \lambda_2\lambda_3).
$$
Finally, adding $\frac{5}{4}\mathop{\sum}\limits_{j=1}^3 (\lambda_j)^2$ to both sides of this last inequality, we obtain
\begin{equation}\label{eq:restr_sum}
\frac{29}{15}\sum_{j=1}^3 ( \lambda_j)^2 \leq \left ( \sum_{j=1}^3 \lambda_j \right )^2.
\end{equation}

\medskip
\medskip

We then seek the critical points of
$$
\omega(\lambda_1, \lambda_2, \lambda_3) = \frac{1}{4}\left ( \sum_{j=1}^3 \lambda_j \right )^2 - \frac{\alpha}{2} \sum_{j=1}^3 \lambda_j + \frac{\beta}{4}\sum_{j=1}^3 (\lambda_j)^2,
$$
under the restrictions that $\lambda_j \geq 0$, $j=1, 2, 3$, and that contained in \eqref{eq:restr_sum}.

We first observe that $\beta > -\frac{29}{15}$ implies that, as $\max\{\lambda_1, \lambda_2, \lambda_3\}\to \infty$ under condition \eqref{eq:restr_sum}, we obtain $\omega(\lambda_1, \lambda_2, \lambda_3) \to \infty$.

Assume now all the restrictions hold with strict inequalities.  Differentiating $\omega$ with respect to $\lambda_j$ we obtain
$$
2\frac{\partial \omega}{\partial \lambda_j} = \sum_{i=1}^3 \lambda_i - \alpha + \beta \,\lambda_j.
$$
For $\lambda_j > 0$, $j=1, 2, 3$, $\nabla \omega = 0$ implies
$$
\beta\,\lambda_j = \alpha - \sum_{i=1}^3 \lambda_j.
$$
We deduce then that the critical points $\lambda^* = ( \lambda_1^*, \lambda_2^*, \lambda_3^*)$ satisfy $\lambda_1^* = \lambda_2^* = \lambda_3^* = \frac{\alpha}{3+\beta}$.  In this case $\mathcal{Q}\mathcal{Q}^T$ is a multiple of the $3\times 3$ identity matrix, and
$$
W(\lambda_1^*, \lambda_2^*, \lambda_3^*) = -\frac{3\alpha^2}{3+\beta}.
$$

If one of the $\lambda_j=0$, say $\lambda_3=0$, we have
$$
\omega(\lambda_1, \lambda_2, 0) = \frac{1}{4}\left ( \sum_{j=2}^3 \lambda_j \right )^2 - \frac{\alpha}{2} \sum_{j=1}^2 \lambda_j + \frac{\beta}{4}\sum_{j=2}^2 (\lambda_j)^2,
$$
and the critical point condition for $\lambda_* = (\lambda_1^*, \lambda_2^*, 0)$ then becomes $\lambda_3=0$, $\lambda_1^* = \lambda_2^* = \frac{\alpha}{2+\beta}$, and
$$
\omega(\lambda_1^*, \lambda_2^*, 0) = \frac{-2\alpha^2}{2+\beta}.
$$
In this case $\mathcal{Q}\mathcal{Q}^T$ is a multiple of a rank-$2$ projection.

\medskip
\medskip

Next, if two of the $\lambda_j=0$, condition \eqref{eq:restr_sum} implies that the third $\lambda_j$ is also $0$.  

Next, if condition \eqref{eq:restr_sum} holds with equality, $\omega(\lambda_1, \lambda_2, \lambda_3)$ depends of only one variable, which we may call $t = \mathop{\sum}\limits_{j=1}^3 \lambda_j$.  Under this restriction, the minimum value $\omega_*$ of $\omega$ is
$$
\omega_* = -\frac{29\alpha^2}{4(29+15\beta)}.
$$
Finally, from all these computations, it is easy to check that the conclusions of the Proposition hold.

\end{proof}




\section{Acknowledgements}
DG was supported in part by the NSF grant DMS-2106551. DS was supported in part by the NSF grant DMS-2009352. The authors would like to thank the IMA where the project was initiated.

\vskip1cm

\noindent {\bf Declarations } 

\noindent Availability of Data and Materials: Data will be made available on reasonable request.

\noindent Conflict of interest: The authors declare that they have no conflict of interests.

\bibliographystyle{abbrv}

\bibliography{Frames,GLFbib}

\end{document}